\documentclass[a4paper,10pt]{amsart}

\usepackage[english]{babel}

\usepackage{verbatim}

\usepackage{amssymb} %for \Bbbk

\usepackage{tikz, tikz-3dplot} % for graphs

\usepackage{lmodern} % Paket zur Verwendung einer verbesserten Schriftart

\usepackage{fancyvrb} % Paket zur besseren Implementierung von Code in den Fließtext

\usepackage{url} % Paket zur Einbindung von URL-Links

\usepackage{hyperref} % Paket zur Verwendung von Links im Dokument

\usepackage{color} % Paket zur Aenderung der Schriftfarben

\newtheorem{thm}{Theorem}[section]
\newtheorem{trm}[thm]{Theorem}
\newtheorem{cor}[thm]{Corollary}
\newtheorem{prop}[thm]{Proposition}
\newtheorem{lem}[thm]{Lemma}
\newtheorem{lma}[thm]{Lemma}

\theoremstyle{definition}

\newtheorem{defi}[thm]{Definition}

\newtheorem{ex}[thm]{Example}

\theoremstyle{remark}
\newtheorem{rem}[thm]{Remark}

\makeatletter
\let\c@equation\c@thm
\makeatother
\numberwithin{equation}{section}
\makeatletter
\let\c@figure\c@thm
\makeatother
\numberwithin{figure}{section}

\newcommand{\C}{\mathbb{C}}
\newcommand{\Q}{\mathbb{Q}}
\newcommand{\Pbb}{\mathbb{P}}

\newcommand{\Z}{\mathbb{Z}}
\newcommand{\N}{\mathbb{N}}
\newcommand{\Kbb}{\Bbbk}
\newcommand{\f}{\textbf}
\newcommand{\mr}{\mathrm}
\newcommand{\mb}{\mathbf}
\newcommand{\bdim}{\mathbf{dim}\,}
\newcommand{\fa}{ \ \mathrm{for}\ \mathrm{all}\ }

\title{GKM-Theory for Torus Actions on Cyclic Quiver Grassmannians}
\author{Martina Lanini}
\author{Alexander P\"{u}tz}
\address{Dipartimento di Matematica, Universit\`a di Roma ``Tor Vergata'',  Via della Ricerca Scientifica 1, I-00133 Rome, Italy}
\email{lanini@mat.uniroma2.it}
\email{puetz@mat.uniroma2.it}

\begin{document}

% Abstract ---------------------------------------------------------------------------------

\begin{abstract}
We define and investigate algebraic torus actions on quiver Grassmannians for nilpotent representations of the equioriented cycle. Examples of such varieties are type $\tt A$ flag varieties, their linear degenerations, finite dimensional approximations of both the affine flag variety and affine Grassmannian for $\mr{GL}_n$. We show that these quiver Grassmannians equipped with our specific torus action are GKM-varieties and that their moment graph admits a combinatorial description in terms of the coefficient quiver of the underlying quiver representations. By adapting to our setting results by Gonzales, we are able to prove that moment graph techniques can be applied to construct module bases for the equivariant cohomology of the quiver Grassmannians listed above.
\end{abstract}

\maketitle
\footnote{AMS Subject Classifications: Primary 16G20; Secondary 14L30.}

%-------------------------------------------------------------------------------------
\section*{Introduction}
%-------------------------------------------------------------------------------------
\pagenumbering{arabic}
%-------------------------------------------------------------------------------------
GKM-theory is named after the seminal paper \cite{GKM1998} by Goresky, Kottwitz and MacPherson, where the authors establish several localisation results in the derived category setting.
 In the present article, we do not make use of the full strength of \cite{GKM1998}, as we only deal with equivariant cohomology.

Let $X$ be a complex projective algebraic variety equipped with an action of an algebraic torus $T$. For instance, consider the projective plane $X=\mathbb{P}^2(\mathbb{C})$ equipped with the following action of $T=\mathbb{C}^*\times \mathbb{C}^*\times \mathbb{C}^*$: 
\[
(\gamma_1, \gamma_2, \gamma_3)\cdot [x_1:x_2:x_3]=[\gamma_1 x_1:\gamma_2 x_2: \gamma_3 x_3],
\]
for  $[x_1:x_2:x_3]\in X$ and $(\gamma_1, \gamma_2, \gamma_3)\in T$. 

GKM-theory aims to identify the equivariant cohomology ring with the image of the pullback $H_T^\bullet(X)\rightarrow H_T^\bullet(X^T)$, and to describe this image in terms of the corresponding moment graph. This is the one-skeleton of the $T$-action on $X$ (that is the set of fixed points and one-dimensional orbits) plus some extra information coming for the torus action on the one-dimensional orbits.  In the case of the projective plane equipped with the 3-dimensional torus action above, the $T$-fixed points are
\[
P_1=[1:0:0], \quad P_2=[0:1:0], \quad P_3=[0:0:1].
\]
and there are three one-dimensional $T$-orbits, say $O_i$ for $i=1,2,3$, each given by the vanishing of the $i$-th coordinate, and each containing in its closure the pair of fixed points $P_j, P_k$ with $i\neq j,k$. The one-skeleton of this torus action is hence a triangle. This is not the desired moment graph yet, since we need to keep track of the torus action on the one-dimensional orbits. For the moment, let us say that this is equivalent to put on any edge a degree one homogeneous polynomial from $S=\mathbb{Q}[\epsilon_1,\epsilon_2,\epsilon_3]$ by following a specific recipe (see \S\ref{subsec:EqvtEuler}). In fact, it will be useful to equip the above graph with an orientation, but for now we can ignore this. All in all, the (unoriented) moment graph of our example is
\begin{center}
\begin{tikzpicture}[scale=.8]

    \draw[shorten >=9, shorten <=9] (2,2) -- (2,-2);
    \draw[shorten >=9, shorten <=9] (-1.464,0) -- (2,-2);
    \draw[shorten >=9, shorten <=9] (-1.464,0) -- (2,2);
	
	\node at (-1.464,0) {$1$};
	\node at (2,2) {$3$};
    \node at (2,-2) {$2$};

    \node[rotate=30] at (-0.1,1.2) {$\epsilon_1-\epsilon_3$};
    
    \node[rotate=-30] at (-0.1,-1.2) {$\epsilon_1-\epsilon_2$};
    
        \node at (2.75,0) {$\epsilon_2-\epsilon_3$};
\end{tikzpicture}
\end{center}
Once the above moment graph is obtained, GKM-theory reduces the determination of the equivariant cohomology to a problem of commutative algebra. In our example, $H_T^\bullet(X)$ can be identified with the following module over $S$:
\[
\{(f_1, f_2,f_3)\in S\oplus S\oplus S\mid f_i-f_j\equiv 0\mod \epsilon_i-\epsilon_j\},
\]
which can be read off from the graph: an element of $H_T^\bullet(X)$ can be realised as a tuple of polynomials, one for any vertex of the moment graph, chosen in such a way that if two vertices are related by an edge, then the corresponding polynomials have to agree modulo the label of such an edge. Observe that the module we have described is free over $S$, and that the following is an $S$-basis
\[
(1,1,1), \quad (0,\epsilon_1-\epsilon_2, \epsilon_1-\epsilon_3), \quad \left(0,0,(\epsilon_1-\epsilon_3)(\epsilon_2-\epsilon_3)\right).
\]

Goresky, Kottwitz and MacPherson studied a big class of varieties acted upon by a torus whose equivariant cohomology can be read off from the corresponding moment graph as in our example. We refer to them as GKM-variety (see Definition~\ref{dfn:GKMvariety}). Examples of GKM-varieties are flag varieties and their Schubert varieties (see, for example, \cite{Carrell02}), as well as rationally smooth standard embeddings of reductive groups \cite{Gonzales2011}.

The aim of this paper is to apply GKM-theory to certain varieties coming from quiver representation theory.

A quiver $Q$ is a finite oriented graph, for instance  \begin{center}
 \begin{tikzpicture}[scale=.6]

    \node at (-2.3,0) {$\Delta_3 =$};

    \draw[arrows={-angle 90},shorten >=9, shorten <=9] (2,2) -- (2,-2);
    \draw[arrows={-angle 90},shorten >=9, shorten <=9] (-1.464,0) -- (2,2);
    \draw[arrows={-angle 90},shorten >=9, shorten <=9] (2,-2) -- (-1.464,0);

	\node at (-1.464,0) {$1$};
	\node at (2,2) {$2$};
    \node at (2,-2) {$3$};
    
\end{tikzpicture}.   
\end{center}
We refer to this as the equioriented cycle of length three.
A representation $M$ of a quiver $Q$ is a configuration of finite dimensional vector spaces $M^{(i)}$ (one for each vertex) and linear maps $M_{i\to j}:M^{(i)}\rightarrow M^{(j)}$ among them (one for each arrow). For example
\begin{center}
 \begin{tikzpicture}[scale=.8]
    \node at (-3.164,0) {$M:=$};

    \draw[arrows={-angle 90},shorten >=8, shorten <=8] (2,2) -- (2,-2);
    \draw[arrows={-angle 90},shorten >=13, shorten <=13] (-1.464,0) -- (2,2);
    \draw[arrows={-angle 90},shorten >=13, shorten <=13] (2,-2) -- (-1.464,0);

	\node at (-1.464,0) {$M^{(1)}=\mathbb{C}^3$};
	\node at (2,2) {$M^{(2)}=\mathbb{C}^3$};
    \node at (2,-2) {$M^{(3)}=\mathbb{C}^3$};
    
      \node[rotate=30] at (-0.1,1.2) {$M_{1\to 2}$};
    
    \node[rotate=-30] at (-0.1,-1.2) {$M_{3\to 1}$};
    
        \node at (2.75,0) {$M_{2\to 3}$};
    
\end{tikzpicture}
\end{center}
where for the standard basis of $\mathbb{C}^3$ the linear maps have the following matrix presentation
\[ M_{1\to 2}=M_{2\to 3}=M_{3\to 1}=\begin{pmatrix} 0 &  0 & 0\\
 1 & 0 &  0\\
 0 & 1 &  0\\ \end{pmatrix}.\]
Note that if we keep composing the linear maps of the above representation following the orientation of the edges we always end up with the zero homomorphism:
\[M_{3\to1}\circ M_{2\to3}\circ M_{1\to2} = M_{1\to2}\circ M_{3\to1}\circ M_{2\to3} =  M_{2\to3}\circ M_{1\to2}\circ M_{3\to1} = 0.\]
This is an example of a nilpotent representation. 
 For a collection ${\bf e}$ of nonnegative integers $e_i\leq\dim M^{(i)}$, the quiver Grassmannian $\mr{Gr}_{\mb{e}}(M)$ is the variety of configurations of vector spaces $U^{(i)}$ of dimensions prescribed by ${\mb e}$ which are compatible with the maps $M_{i\rightarrow j}$. 
For the collection $\mb{e}=(1,1,1)$ and $M$ as above the quiver Grassmannian $\mr{Gr}_{\mb{e}}(M)$ contains the point $U$ with 
\[ U^{(1)} := \mr{span}(e_1), \quad  U^{(2)} := \mr{span}(e_2), \quad U^{(3)} := \mr{span}(e_3),\]
 where $\{e_1,e_2,e_3\}$ denotes the standard basis of $\mathbb{C}^3$.
Note that if the quiver is just the graph with one vertex and no arrows, a representation is just a finite dimensional vector space, and quiver Grassmannians are classical Grassmann varieties. 

In this work we will focus on the special case in which $Q$ is the equioriented cycle of length $n$ and $M$ is a nilpotent representation.
%has the property that all concatenations of the maps along the arrows vanish
%that is the graph whose vertices are labelled by the numbers from 1 to n and whose arrows are of the form $i\rightarrow i+1$ (
Our primary goal is to equip these class of quiver Grassmannians with a torus action which provides them with a GKM-variety structure.

This is not the first time that GKM-theory meets representation theory of quivers: in \cite{CFR2013} the moment graph of a torus action on a quiver Grassmannian for a very special representation of the equioriented quiver of type $\tt A$ is described (see \S\ref{subsec:FeiginDegeneration} of this paper for more details). In \cite{Weist2013} a torus action on quiver moduli is introduced with localisation results in mind. Observe that quiver Grassmannians for a fixed quiver are quiver moduli for the one point extension of the same quiver. Both articles work with one explicit torus depending on the representation. The results of \cite{CFR2013} are limited to this special torus whereas the results from \cite{Weist2013} can be generalised as described in \cite[Remark~3.2]{BoFr2020}. Also, the action as introduced in \cite{Weist2013} has been applied recently, for example, in \cite{Franzen2019, BoFr2020}.  

Unluckily, Weist's torus action does not equip the corresponding quiver moduli with the structure of a GKM-variety in general (some of the obstructions are explained in the introduction of \cite{Franzen2019}). A known class where this works requires strong restrictions, among them acyclicity of the quiver. Our torus action, instead, turns every quiver Grassmannian for a nilpotent representation of the equioriented cycle into a GKM-variety, with no further restrictions.

We hope that this paper will motivate both, the reader familiar with GKM-theory, as well as the reader familiar with quiver representations, to (further) apply moment graph techniques to quiver Grassmannians.

In order to reach both communities, we have decided to spend some time recalling the basics of both theories. To help the reader navigate the paper, we now describe the content of the various sections.

In Section~\ref{sec:TorusActionCellsGKMTheory}, we deal with varieties equipped with a torus action, and describe the properties we want them to satisfy (equivariant formality, $T$-skeletality, BB-filterability). We also state the GKM-version of the Localisation Theorem for
equivariant cohomology (Theorem~\ref{thm:GKM}).

The primary goal of Section~\ref{sec:ModuleBases} is to produce a cohomology module basis (under GKM-localisation). We adapt Gonzales' work \cite{Gonzales2014} and show that the normality assumption in his article can be dropped if the variety is BB-filterable. This is needed since quiver Grassmannians are not normal in general. The main result of this section is Theorem~\ref{trm:cohomology-generators-general-setting}, which provides existence and uniqueness of an equivariant basis with certain suitable properties. The basis we propose generalises the equivariant Schubert cycle basis for the cohomology of the flag variety. Following Gonzales' recipe, the definition of the basis relies on the concept of local indices and equivariant Euler classes (see \S\ref{subsec:EqvtEuler}).

In Section~\ref{sec:GenQuiverRep}, we provide some background material on quiver representations and quiver Grassmannians. In particular, we recall the definition of the coefficient quiver of a quiver representation (Definition~\ref{def:CoeffQuiver}), a combinatorial gadget encoding all information about the given representation and one particular chosen basis for the representation. This object will play a central role in the rest of paper.

From Section~\ref{sec:quiver-grass-for-cycle} on, we restrict our attention to the equioriented cycle with $n$ vertices (denoted by $\Delta_n$). We show that in this case, any nilpotent representation admits a basis, whose corresponding coefficient quiver behaves in a particularly convenient way (see \S\ref{sec:nilpot-rep-cycle}). 

In Section~\ref{sec:TorusAction}, we use this good combinatorial behaviour to define torus actions on quiver Grassmannians for nilpotent representations of $\Delta_n$. 
We start by defining a $\C^*$-action, which induces a cellular decomposition of the variety
(Theorem~\ref{trm:cell_decomp-approx-lin-deg-aff-flag}). Then we define an action of a larger-rank torus $T$ (see \S\ref{sec:act-bigger-T}) and show that the previously-defined $\C^*$-action corresponds to a generic cocharacter of the larger torus. We conclude the section by showing that the quiver Grassmannian equipped with the $T$-action is a BB-filterable variety (Corollary~\ref{cor:BBfiltrDeltan}).

Finally, we describe the moment graph for the $T$-action on the quiver Grassmannian in Section~\ref{sec:GKM-VarietyStructure}.
More precisely, we show that this oriented graph with labelled edges has a combinatorial description: the vertices of the graph are given by successor closed subquivers (see Definition~\ref{dfn:successorClosed}) of the coefficient quiver and the edges by fundamental mutations (Definition~\ref{defn:FundMutation}). The precise statement, which also explains how to label the edges of the graph via torus characters, is Theorem~\ref{trm:comb-moment-graph}.

Section~\ref{sec:SpecialCases} deals with some special cases. We start by focusing on quiver Grassmannians for the equioriented type $\tt A_n$ Dynkin quiver. Our results apply since any of its representations can be trivially extended to a nilpotent representation of $\Delta_n$. We hence show that in the case of the variety of complete flags and Feigin's degeneration of it, Theorem~\ref{trm:comb-moment-graph} allows us to recover known moment graphs: the Bruhat graph and the graph described in \cite{CFR2013} respectively. %We then observe that 
Our constructions also applies to certain finite dimensional approximations of the affine flag variety and affine Grassmannian for $\mr{GL}_n$ as defined in \cite{Pue2020} (see Lemma~\ref{lma:approx-are-GKM}). For one example of such degenerations, we draw its moment graph, determine the module basis from Theorem~\ref{trm:cohomology-generators-general-setting}, and describe the ring structure of the equivariant cohomology.

In Appendix~\ref{app:Desing}, we explain how to construct equivariant resolutions of singularities in the explicit example from Section~\ref{sec:SpecialCases}. This allows to compute the equivariant Euler classes of $T$-varieties (at singular points).
\\[4pt]
%-------------------------------------------------------------------------------------

\subsection*{Acknowledgements} We would like to thank Richard Gonzales for helpful correspondence. We acknowledge the PRIN2017 CUP E8419000480006, and the MIUR Excellence Department Project awarded to the Department of Mathematics, University of Rome Tor Vergata, CUP E83C18000100006.
%-------------------------------------------------------------------------------------
\section{Torus Actions, Cellular Decompositions and GKM-Theory}\label{sec:TorusActionCellsGKMTheory}
%-------------------------------------------------------------------------------------
\subsection{GKM-Varieties} Throughout this section, $X$ will denote a complex projective algebraic variety. We say that $X$ is a $T$-variety if it is acted upon by an algebraic torus $T \cong (\C^*)^r$. If $X$ is a $T$-variety, we denote by $H_T^\bullet(X)$ the $T$-equivariant cohomology of $X$ with rational coefficients. 

We are interested in a class of $T$-varieties with a particularly nice $T$-action.% with respect to $H_T^*(X)$.
%We are interested in a class of $T$-varieties on which the torus acts in a particularly nice way
\begin{defi}\label{def:EquivariantlyFormal}
A $T$-variety $X$ is \f{equivariantly formal} if one of the following equivalent conditions is satisfied:
\begin{enumerate}
\item the Serre spectral sequence degenerates at $E_2$, %TODO optimise this part?
\item the ordinary rational cohomology can be recovered by extension of scalars: %Gonzales p.7
\[ H^\bullet(X) \cong H_T^\bullet(X) \otimes_{H_T^\bullet(pt)}\Q, \]
\item $H_T^\bullet(X)$ is a free $H_T^\bullet(pt)$-module. 
\end{enumerate}
\end{defi}
Condition (1) of the above definition is discussed in details in 
%The original definition dates back to the construction by Borel 
\cite[Section XII]{Borel1960}. A proof that the other conditions are equivalent can be found in \cite[Theorem~1.6.2]{GKM1998} or \cite[Lemma~1.2]{Brion2000} where also the following lemma is proven.
\begin{lma}
\label{lma:equiformality}
$X$ is equivariantly formal if the rational cohomology of $X$ vanishes in odd degrees. Both conditions are equivalent if $X$ has finitely many $T$-fixed points. 
\end{lma}
Since the variety $X$ is equivariantly formal with respect to the $T$-action, we will often denote an equivariantly formal variety by $(X,T)$.
In order to apply localisation techniques, we require more than equivariant formality. 
\begin{defi} We say that the $T$-action on $X$ is
\begin{enumerate}
\item  \f{skeletal} if the number of $T$-fixed points and one-dimensional $T$-orbits in $X$ is finite; 
\item \f{locally linearisable} if for each one-dimensional orbit $E$ in $X$ there is a linear action of $T$ on $\C\Pbb^1$ and a $T$-equivariant isomorphism $h : \overline{E} \to \C\Pbb^1$. 
\end{enumerate}
\end{defi}  
\begin{defi}\label{dfn:GKMvariety}
 We say that $X$, or $(X,T)$, is a \f{GKM-variety} if it is equivariantly formal and the $T$-action is skeletal.
 \end{defi}
  \begin{rem}\label{rem:1dimlOrbits}Recall that for us $X$ is always a projective variety. Then, by \cite[(1.2)]{GKM1998}), the $T$-action is locally linearisable, as soon as $(X,T)$ is a GKM-variety. 
\end{rem}

\begin{rem}
Our definition of GKM-variety differs from the definition by Gonzales \cite[Definition~1.4.13]{Gonzales2011}, as we do not assume normality. This is central for us, since the varieties we want to deal with fail to be normal in general \cite[Theorem~13]{CFFFR2017}. 
By \cite[Corollary~2]{Sumihiro1974}, the $T$-action on normal varieties is locally linearisable. %normal varieties satisfy the one-dimensional orbits' property from Remark~\ref{rem:1dimlOrbits}. 
\end{rem}
The above definition of GKM-variety is based on the assumptions by Goresky, Kottwitz and MacPherson \cite[\S~7.1]{GKM1998}. 
\begin{ex}Examples of GKM-varieties are (finite dimensional) Schubert varieties (of flag varieties for a Kac-Moody group) \cite{Carrell02}, toric varieties \cite{Brion1997}, rationally smooth embeddings of reductive groups \cite{Gonzales2011}.
\end{ex}

\subsection{BB-filterable Varieties} Assume that $X$ is equipped with a $\C^*$-action, and denote by $X_1, \ldots, X_m$ the connected components of the fixed point set of $X$, which we denote by $X^{\C^*}$.
This induces a decomposition
\begin{equation}\label{eqn:BBdecomposition}
X=\bigcup_{i\in [m]} W_i, \quad\hbox{ with }\quad   W_i := \left\{ x \in X \mid \lim_{z \to 0} z.x \in X_i \right\},
\end{equation}
for $[m] := \{1,\dots,m\}$.  We call this a \f{BB-decomposition} since decompositions of this type were first studied by Bialynicki-Birula in \cite{Birula1973}. 
\begin{defi}We say that $W_i$ from \eqref{eqn:BBdecomposition} is a \f{rational cell} if it is rationally smooth at all $w\in W_i$. This in turn holds if 
\[H^{2\textrm{dim}_{\C}(W_i)}(W_i, W_i\setminus \{w\})\simeq \Q  \quad\hbox{ and  }\quad  H^m(W_i,W_i\setminus\{w\})=0 \]
for any $m\neq 2 \textrm{dim}_{\C}(W_i)$ (cf. \cite[p.292, Definition~3.4]{Gonzales2014}). 
\end{defi}
\begin{rem}
These $W_i$ are called attractive sets and are isomorphic to affine spaces in the original BB-decomposition. Requiring the attractive sets to be affine spaces is a strong restriction, so that usually
%This not automatically true for the varieties that we study because they might contain singular points.
%Since the varieties we study might contain singular points, it is not automatically true that the attractive sets are isomorphic to affine spaces.
 the BB-decomposition does not have to be a cellular decomposition. Nevertheless, the notion of rational cells provides a reasonable replacement of this condition for the study of topological properties in the case of singular varieties (see \cite{Gonzales2014}).
\end{rem}
\begin{rem}
We will show in Theorem ~\ref{trm:cell_decomp-approx-lin-deg-aff-flag} that it is possible to obtain attractive sets which are in fact affine spaces for the class of varieties we are interested in. We decided nevertheless to deal with rational cells in this section, as the results we achieved are intended to 
\end{rem}
Let $\mathfrak{X}_*(T)$ be the cocharacter lattice of an algebraic torus $T$. If $X$ is a $T$-variety, then every $\chi\in \mathfrak{X}_*(T)$ determines a $\C^*$-action on $X$.

\begin{defi}A cocharacter $\chi$ is \f{generic} (for $T$ acting on $X$) if $X^{\chi(\C^*)}=X^T$.
\end{defi}
\begin{rem}\label{rem:genCochar}Recall that $X$ always denotes a complex projective variety. Under such an assumption, it is enough to have $|X^{\chi(\C^*)}|<\infty$ to conclude that the cocharacter $\chi$ is generic. Indeed, since $\chi(\C^*)$ is a subgroup of $T$, then $X^T\subseteq X^{\chi(\C^*)}$. It is a known fact that the Euler characteristic of $X$ agrees with the number of fixed points of any algebraic torus action on $X$, as soon as the latter number is finite. It follows that the two fixed point sets have the same cardinality and hence have to coincide.
\end{rem}
 
\begin{defi}
\label{def:BB-filterable}
A projective $T$-variety $X$ is \f{BB-filterable} if: 
\begin{enumerate}
%\item the $T$-action on $X$ is locally linearisable, 
%\item[(BB0)] $X$ admits a $T$-equivariant embedding into some projective space $\mathbb{P}^N$ where $T$ acts linearly,
\item[(BB1)] the fixed point set $X^T$ is finite,
\item[(BB2)] there exists a generic cocharacter $\chi: \C^* \rightarrow T$, i.e. $X^{\chi(\C^*)} = X^T$, such that 
 the associated BB-decomposition consists of rational cells. 
\end{enumerate}
\end{defi}
The above definition is very much inspired by Gonzales' definition of $\Q$-filterable variety (see \cite[Definition 4.6]{Gonzales2014}). Here we relax the assumptions in \cite{Gonzales2014} and do not require that $X$ is normal. The following theorem extends \cite[Theorem~4.7]{Gonzales2014} to the class of BB-filterable varieties. Its proof is based on Gonzales' idea but has to be adapted to the setting of BB-filterable varieties.

%Gonzales' proof still works, so that \cite[Theorem~4.7]{Gonzales2014} extends to the class of BB-filterable varieties as follows.
\begin{trm}
\label{trm:t-stable_filtration-general-setting} Let $X$ be a BB-filterable projective $T$-variety. Then: 
\begin{enumerate}
\item $X$ admits a filtration into $T$-stable closed subvarieties $Z_i$ such that
\[ \emptyset = Z_0 \subset Z_1 \subset \dots \subset Z_{m-1} \subset Z_m = X. \]
\item Each $W_i = Z_i \setminus Z_{i-1}$ is a rational cell, for all $i \in [m]$.
\item The singular rational cohomology of $Z_i$ vanishes in odd degrees, for $i \in [m]$. In other words, each $Z_i$ is equivariantly formal.
\item If, additionally, the $T$-action on $X$ is skeletal, each $Z_i$ is a GKM-variety.%, for $i \in [m]$.
\end{enumerate}
\end{trm}
\begin{proof}
$X$ is BB-filterable which by Definition~\ref{def:BB-filterable} implies that the attractive loci of the BB-decomposition are rational cells. These cells are $T$-stable since $\C^*$ acts via some generic cocharacter $\chi\in \mathfrak{X}_*(T)$. By \cite[Lemma~4.12]{Carrell02} there exists a total order of the fixed points such that, if we define the subvarieties $Z_i$ inductively by removing the rational cell $W_{i+1}$, they are nested and closed in $X$.

We apply \cite[Lemma~4.4]{Gonzales2014} inductively to the filtered BB-decomposition and get that the $Z_i$ have no odd cohomology. Lemma~\ref{lma:equiformality} implies that they are equivariantly formal. %By assumption and part (3) we have that the $T$-action on $X$ is locally linearisable and each $Z_i$ is equivariantly formal. 
Hence a $T$-skeletal action implies that the $Z_i$'s are GKM-varieties.
\end{proof}
\begin{rem} 
In particular, we obtain that the $Z_i$ are GKM-varieties, as soon as we have finitely many one-dimensional $T$-orbits. This suffices since $|X^T|<\infty$ holds by the definition of BB-filterable varieties. 
\end{rem}
\begin{rem} 
By Remark~\ref{rem:1dimlOrbits}, Theorem~\ref{trm:t-stable_filtration-general-setting} implies that the $T$-action on BB-filterable projective $T$-varieties is locally linearisable.
\end{rem}
\begin{rem}\label{rem:Zi} If $\{Z_0, Z_1, \ldots, Z_m\}$ and $\{W_1, \ldots, W_m\}$ are as in Theorem~\ref{trm:t-stable_filtration-general-setting}, then for any $i$ we have that $W_i$ is open in $Z_i$ and $Z_i\setminus W_i$ is a (closed) $T$-stable subvariety of $Z_i$.
\end{rem}

\subsection{Equivariant Localisation after Goresky-Kottwitz-MacPherson} \label{subsec:EqvtEuler}
The equivariant cohomology of a GKM-variety $(X, T)$ can be described by looking at the one-skeleton of the $T$-action. The idea of extracting all needed data from the zero- and one-dimensional $T$-orbits is actually due to Chang and Skjelbred \cite{CS74}, but such an approach is nowadays known as GKM-theory after the paper \cite{GKM1998}.

Functoriality of equivariant cohomology implies that there is a $\N$-graded algebra homomorphism
\[
\psi: H_T^\bullet(X)\rightarrow H_T^\bullet\Big(\cup_ {i=1}^m X_i\Big)\simeq\bigoplus_{i\in [m]}H_T^\bullet(X_i),
\]
where $X_1, \ldots, X_m$ are the connected components of the fixed point set as in \eqref{eqn:BBdecomposition}. In particular, if $X$ has a finite number of (isolated) $T$-fixed points, we can identify $H_T^\bullet(X^T)$ with $\bigoplus_{x\in X^T} H_T^\bullet(\textrm{pt)}$. From now on we use $S:=H_T^\bullet(\textrm{pt)}$ as shorthand notation. $S$ can be identified with the symmetric algebra of the $\Q$-vector space over the torus character lattice $\mathfrak{X}^*(T)\otimes_{\Z}\Q$. 

If the $T$-action on $X$ is locally linearisable, any one-dimensional orbit $E$ contains exactly two fixed points in its closure, say $x_E$ and $y_E$. Clearly, the torus acts on $E$ via a character (uniquely defined up to a sign, depending on the isomorphism $\overline{E}\simeq\mathbb{P}^1$). Since the sign choice does not play any role in the following theorem, we just pick a torus character, denoted by $\alpha_E$, for each one-dimensional orbit $E$.  

The above data concerning $T$-fixed points, one-dimensional orbits and their closure is encoded in an oriented graph whose edges are labelled by torus characters.
\begin{defi}\label{def:moment-graph}Let $(G,T)$ be a GKM-variety, and let $\chi\in \mathfrak{X}_*(T)$ be a generic cocharacter. The corresponding \f{moment graph} $\mathcal{G}=\mathcal{G}(X,T, \chi)$ of a GKM-variety is given by the following data:
\begin{itemize}
\item[(MG0)] the $T$-fixed points as vertices, i.e.: $\mathcal{G}_0 = X^T$,
\item[(MG1)] the closures of one-dimensional $T$-orbits $\overline{E} = E \cup \{x,y\}$ as edges in $\mathcal{G}_1$, oriented from $x$ to $y$ if $\lim_{\lambda \to 0}\chi(\lambda).p=x$ for $p \in E$,
\item[(MG2)] every $\overline{E}$ is labelled by a character $\alpha_E \in \mathfrak{X}^*(T)$ describing the $T$-action on $E$. 
\end{itemize}
\end{defi}
If the choice of the cocharacter $\chi$ is clear from the context or the orientation is not relevant, we sometimes drop it from the notation for the moment graph.
%\begin{defi}Let $(X, T)$ be a GKM-variety. The \f{moment graph} $\mathcal{G}(X,T)$ of $(X,T)$ is the graph whose vertex set $\mathcal{G}(X,T)_0$ is given by $X^T$ and whose set of edges $\mathcal{G}(X,T)_1$ is given by the one-dimensional orbit closures. Every edge $\overline{E}$ is labelled by the corresponding torus character $\alpha_E$ and oriented from $y_E$ to $x_E$ if $x_E$ lies in the closure of the BB-cell containing $y_E$. \end{defi} 

\begin{trm}\label{thm:GKM}(\cite[Theorem~1.2.2]{GKM1998}) Let $(X, T)$ be a GKM-variety. Then $\psi$ is injective and its image is
\[\textrm{Im}(\psi)=\left\{(f_x)\in\bigoplus_{x\in \mathcal{G}(X,T)_0}S \ \Big| \
\begin{array}{c}
f_{x_E}-f_{y_E}\in \alpha_{E}S\\
 \hbox{ for any }\overline{E}=E\cup\{x_E, y_E\}\in\mathcal{G}(X,T)_1\end{array}
\right\}.\]  
\end{trm}

\begin{rem}Since the appearance of \cite{GKM1998}, moment graph techniques have been extensively --and successfully-- applied to the study of equivariant cohomology of Schubert varieties (in Kac-Moody flag varieties) \cite{Carrell02}, Hessenberg varieties, standard group embeddings, and more. For more examples see the excellent survey \cite{Tymoczko2005}.  The aim of our paper is to further expand the class of varieties whose equivariant cohomology ring can be investigated by looking at their moment graphs.
\end{rem}

\section{Construction of Cohomology Module Bases}\label{sec:ModuleBases} By definition, the equivariant cohomology of an equivariantly formal space $X$ is a free module over $S$. It is hence  natural to look for an $S$-basis of $H_T^\bullet(X)$. %(after the identification with $\textrm{Im}(\psi)$). 
In this section, we address this question in the generality of GKM-varieties.

\subsection{Equivariant Euler classes}
To construct our basis, we will use the same recipe as Gonzales in \cite{Gonzales2014}, and hence need equivariant Euler classes and local indices. 
For a $T$-variety $Y$ and a fixed point $y\in Y^T$, we denote by $\mr{Eu}_T(y, Y)$ the equivariant Euler class of $y$ in $Y$. This is an element of the fraction field $Q$ of $S$, whose inverse (up to a sign) is obtained by localising the fundamental class in Borel-Moore homology.  We refer the reader to \cite[Section~2.2.1]{Arabia98} for the precise definition, and limit ourselves to three properties, which very often are enough to determine the equivariant Euler classes. 
 
 \begin{lem}\label{lma:euler-class-along-resolution}(cf. \cite[Corollary~15, Lemma~16, Theorem~18]{Brion1997}) Let $Y$ be a $T$-variety and $y\in Y^T$. 
 \begin{enumerate}
 \item If $Y$ is smooth at $y$ then $\mr{Eu}_T(y,Y)=(-1)^{\mr{dim}(Y)} \mr{det} \, T_yY$, where $\mr{det} \, T_yY$ is the product of the characters by which $T$ acts on the tangent space $T_yY$.
 \item If $Y$ is rationally smooth at $y$ then $\mr{Eu}_T(y,Y)=z\cdot \mr{det} \, T_yY$, for some $z\in\Q\setminus\{0\}$. 
\item If $\pi: Y\to X$ is a $T$-equivariant resolution of singularities and $|Y^T|<\infty$, then
\[ \mr{Eu}_T(x,X)^{-1} = \sum_{ y \in Y^T, \pi(y)=x} \mr{Eu}_T(y,Y)^{-1}.\]
 \end{enumerate}
 \end{lem}
 
  \begin{rem}Actually, Brion in \cite{Brion1997} studies equivariant multiplicities rather than Euler classes; they are inverse to each other (up to some sign which has been taken care of in the statement of Lemma~\ref{lma:euler-class-along-resolution}).
 \end{rem}
 
\begin{rem}
By using the properties in the previous lemma, Arabia in \cite[\S2.7(27)]{Arabia98} determines (the inverse) equivariant Euler classes of Schubert varieties by looking at Bott-Samelson resolutions. The above lemma also allows us to determine equivariant Euler classes, and hence the desired module basis for the equivariant cohomology, by constructing desingularisations of the quiver Grassmannians, we are looking at (see Appendix~\ref{app:Desing}).
\end{rem} 
 
In the following, thanks to Theorem~\ref{thm:GKM}, we identify $H^\bullet(X)$ with $\textrm{Im}(\psi)$, so that $f\in H_T^\bullet(X)$ will be given by a collection $(f_x)\in\bigoplus S$, satisfying the conditions given by the edge labels of the moment graph $\mathcal{G}{(X,T)}$. 
\begin{lem}\label{lem:tau}Let $(X,T)$ be a BB-filterable GKM-variety with filtration 
$$\emptyset=Z_0\subset Z_1\subset\ldots\subset Z_m=X$$ as in Theorem~\ref{trm:t-stable_filtration-general-setting}. Let $X^T=\{x_1, \ldots, x_m\}$ with $x_i\in W_i=Z_i\setminus Z_{i-1}$. For $i\in [m]$ define
\[
\tau^{(i)}_{x_j}:=\left\{\begin{array}{ll}
0&\hbox{ if }j\neq i\\
\mr{Eu}_T(x_i, Z_i)&\hbox{ if }j=i\\
\end{array}\right.\qquad (j \in [m]).
\]
Then, $\tau^{(i)}:=(\tau^{(i)}_{x_j})_{j\in [m]}\in H_T^\bullet(Z_i)$.
\end{lem}
\begin{proof}By Theorem~\ref{trm:t-stable_filtration-general-setting}(4), $(Z_i, T)$ is itself a GKM-variety for any $i\in [m]$. Therefore, by Theorem~\ref{thm:GKM}, $\tau^{(i)}\in H_T^\bullet(Z_i)$ if and only if all relations coming from the edges are verified. Since all but one entry of $\tau^{(i)}$ vanish, we only have to check that
\[
\mr{Eu}_T(x_i, Z_i)\equiv 0\mod \alpha_E
\]
 for any $\overline{E}\in \mathcal{G}(Z_i, T)_1$ adjacent to $x_i$. To obtain this, we just notice that the proof of \cite[Lemma~6.4]{Gonzales2014} works under our assumptions too. Indeed, by \cite[Corollary~5.6]{Gonzales2014}, there exists a non-zero $z\in\Q$ such that
\[
Eu_T(x_i, W_i)=z\cdot\alpha_{E_1}\cdot\ldots\cdot\alpha_{E_r}
,\]
where $E_1 \ldots, E_r$ are the 1-dimensional $T$-orbits lying in $W_i$ and whose closure contains $x_i$. Recall that $W_i$ is open in $Z_i$ and its complement is $T$-stable.  Hence, thanks to local linearisability, we can apply the proof of \cite[Lemma~6.3]{Gonzales2014}, to deduce that all one-dimensional $T$-orbits, lying in $Z_i$ and containing $x_i$ in their closures, are actually contained in $W_i$.  We conclude that the edges in $\mathcal{G}(Z_i, T)_1$ that are adjacent to $x_i$ are exactly $\{\overline{E_1}, \ldots, \overline{E_r}\}$ and the product of their labels is a non-zero multiple of $\mr{Eu}_T(x_i, Z_i)$. 
\end{proof}

The following theorem is due to Gonzales, and our only contribution is to notice that, once again, his proof works also under our hypotheses:
\begin{trm}\label{trm:Basis1}
Let $(X,T)$ be a BB-filterable GKM-variety with filtration 
$$\emptyset=Z_0\subset Z_1\subset\ldots\subset Z_m=X$$ as in Theorem~\ref{trm:t-stable_filtration-general-setting}. Let $X^T=\{x_1, \ldots, x_m\}$ with $x_i\in W_i=Z_i\setminus Z_{i-1}$. There exists a basis $\{\varphi^{(i)}\}_{i\in [m]}$ of $H_T^\bullet(X)$ as a free $S$-module satisfying the following two properties:
\begin{enumerate}
\item $\varphi^{(i)}_{x_j}=0$ for any $j<i$,
\item $\varphi^{(i)}_{x_i}=\mr{Eu}_T(x_i, Z_i)$.
\end{enumerate}
\end{trm}
\begin{proof}The proof is by induction on the length $m$ of the filtration. If $m=1$, then $X$ is a point and the statement is trivial. Since $(Z_{m-1}, T)$ is a BB-filterable GKM-variety, we get cohomology generators $\{\tilde{\varphi}^{(i)}\}$ of $H_T^\bullet(Z_{m-1})$ satisfying (1) and (2). These elements, can be lifted to $H_T^\bullet(Z_m)=H_T^\bullet(X)$, in a way which is compatible with the localisation map $\psi$, thanks to the commuting diagram \cite[(1)]{Gonzales2014}. At this point, we have $m-1$ elements $\varphi^{(1)}, \ldots, \varphi^{(m-1)}$ satisfying the desired properties (1) and (2). For the missing generator we set $\varphi^{(m)}:=\tau^{(m)}$, where $\tau^{(m)}$ is the one from Lemma~\ref{lem:tau}. 

Standard arguments imply that a set of elements satisfying properties (1) and (2) is linearly independent and generates $H_T^\bullet(X)$ (cf. \cite[Lemma~6.2]{Gonzales2014}).
\end{proof}

\begin{rem}
Notice that Theorem~\ref{trm:Basis1} gives existence but not uniqueness of the basis. Indeed, the induction step of the proof consists in lifting classes from $H_T^\bullet(Z_{m-1})$ to $H_T^\bullet(X)$ and in general this lift does not need to be unique. 
It is hence natural to ask whether there is a preferred basis, among the ones which satisfy properties (1) and (2) of Theorem~\ref{trm:Basis1}, and if so, how to choose it. 
\end{rem}
\begin{rem}
Observe that $\tau^{(i)}$ is a special element of the $T$-equivariant cohomology of the $i$-th piece in the filtration of $X$, whereas $\varphi^{(i)}$ denotes the $i$-th element in the ordered basis for the $T$-equivariant cohomology of $X$. For the runnig example from the introduction we have $\tau^{(2)}=(0,\epsilon_1-\epsilon_2, 0)$ and $\varphi^{(2)}=(0,\epsilon_1-\epsilon_2, \epsilon_1-\epsilon_3)$.
\end{rem}
\subsection{Local Indices and a special Module Basis} 
Before constructing the desired basis, we need to introduce another ingredient: the local index of a cohomology class at a fixed point. We then show that Theorem~\ref{trm:Basis1} produces only one basis that also satisfies a particular condition with respect to the local index.
%Requiring a special behavior of a basis from Theorem~\ref{trm:Basis1}, with respect to the local index, will give us uniqueness.

The local index of $f \in H_T^\bullet(X)$ at $x_i \in X^T$ is defined in terms of what is called integration map introduced in \cite[\S1.4]{Arabia98}. Instead of the original definition, we will define it via an explicit formula (under the localisation map $\psi$):
\begin{defi}(cf. \cite[Lemma~6.7]{Gonzales2014}) Let $X^T=\{x_1, \ldots, x_m\}$. For $i \in [m]$, the local index of $f\in H_T^\bullet(X)$ at $x_i \in X^T$ is 
\begin{equation}\label{eqn:localIndex}
I_i(f) = \sum_{\substack{j \in [m] \ : \ \\x_j\in Z_i}} \frac{f_{x_j}}{\mr{Eu}_T(x_j,Z_i) \ \ }.
\end{equation}
\end{defi}
\begin{ex}
Consider the action of $T= (\mathbb{C}^*)^3$ on $X=\mathbb{P}^2(\mathbb{C})$ as studied in the introduction. The local index of $f=(0,\epsilon_1-\epsilon_2, \epsilon_1-\epsilon_3)$ at $2$ is
\[
I_2(f) = \frac{f_1}{\mr{Eu}_T(x_1,Z_2)}+\frac{f_2}{\mr{Eu}_T(x_2,Z_2)}= 0 + \frac{\epsilon_1-\epsilon_2}{\epsilon_1-\epsilon_2}=1.
%(1,1,1), \quad (0,\epsilon_1-\epsilon_2, \epsilon_1-\epsilon_3), \quad \left(0,0,(\epsilon_1-\epsilon_3)(\epsilon_2-\epsilon_3)\right).
\]
\end{ex}
\begin{rem}The above definition is useful for computations, but has the disadvantage that by the formula one cannot tell that $I_i(f)$ is actually polynomial. %$I_i(f)\in S$ 
Luckily, this is the case, and it is immediate by the definition in terms of the integration map. 
\end{rem}

The following theorem provides us with a preferred choice among the bases from Theorem~\ref{trm:Basis1}. Since everything depends on the order of enumeration of the fixed points (and hence of the filtration), which is not unique, we refrain from referring to this basis as canonical.
\begin{trm}
\label{trm:cohomology-generators-general-setting} 
Let $(X,T)$ be a BB-filterable GKM-variety with filtration 
$$\emptyset=Z_0\subset Z_1\subset\ldots\subset Z_m=X$$ as in Theorem~\ref{trm:t-stable_filtration-general-setting}. Let $X^T=\{x_1, \ldots, x_m\}$ with $x_i\in W_i=Z_i\setminus Z_{i-1}$. There exists a unique basis $\{\theta^{(i)}\}_{i \in [m]}$ of $H_T^\bullet(X)$ as an $S$-free module, such that for any $i \in [m]$ the following properties hold:
\begin{enumerate}
\item $I_i(\theta^{(i)})= 1$,
\item  $I_j(\theta^{(i)}) = 0$  for all $j \neq i$,
\item  $\theta^{(i)}_{x_j}=0$ for all $j < i$,
\item  $\theta^{(i)}_{x_i}= \mr{Eu}_T(x_i,Z_i)$.
 \end{enumerate}
\end{trm}
\begin{proof}
As for previous results, Gonzales' proof of \cite[Theorem~6.9]{Gonzales2014} goes through and hence we limit ourselves to give only a sketch.

Firstly, we show existence. Let $i \in [m]$, and consider $\widetilde{\theta^{(i)}}:=z^{-1}\cdot\varphi^{(i)}$, where $\varphi^{(i)}$ is any element of $H_T^\bullet(X)$ satisfying (1) and (2) from Theorem~\ref{trm:Basis1}, and $z\in \Q$ is such that $\varphi^{(i)}_{x_i}=z\cdot \mr{Eu}_T(x_i, Z_i)$. Thanks to \eqref{eqn:localIndex}, it is easy to check that (1), (3) and (4) hold. If (2) holds too, we are done, otherwise, we inductively modify $\widetilde{\theta^{(i)}}$ as follows: let $k_0:=\min\{j>i\mid I_j(\widetilde{\theta^{(i)}})\neq 0\}$ and replace $\widetilde{\theta^{(i)}}$ by 
$\widetilde{\theta^{(i)}}-I_{k_0}(\widetilde{\theta^{(i)}})\widetilde{\theta^{(k_0)}}$. It is again an easy check to see that the local index of this new element vanishes at any point $x_j$ with $j\leq k_0$ and $j\neq i$, and that (1), (3), (4) still hold. At the end of this process we get an element of $H_T^\bullet(X)$ that we denote by $\theta^{(i)}$ and that satisfies (1), (2), (3), (4).

Secondly, they freely generate $H_T^\bullet(X)$ by standard arguments (cf. proof of \cite[Lemma~6.2]{Gonzales2014}).

Finally, the uniqueness is shown by contradiction. Assume that we can find $\theta^{(i)}$ and $\psi^{(i)}$ both satisfying (1)--(4) and such that $\theta^{(i)}\neq \psi^{(i)}$. As they are distinct, we can find $k_0:=\min\{j\mid \theta^{(i)}_{x_{j}}-\psi^{(i)}_{x_j}\neq 0\}$. Since they both satisfy (4), $k_0\neq i$ and we have that $I_{k_0}(\theta^{(i)}-\psi^{(i)})=0$. But from \eqref{eqn:localIndex}, we get
\[
0\neq\theta^{(i)}_{x_{k_0}}-\psi^{(i)}_{x_{k_0}}=\underbrace{I_{k_0}(\theta^{(i)}-\psi^{(i)})}_{=0}\cdot \mr{Eu}_T(x_{k_0}, Z_{k_0}),
\] 
which gives us the desired contradiction.
\end{proof}

\begin{rem}If $G\supset P\supset T$ are, respectively, a complex linear reductive algebraic group, a parabolic subgroup and a maximal torus, then the above basis of $H_T^\bullet(G/P)$ coincides with the one given by equivariant Schubert classes. 
\end{rem}

The rest of this article is devoted to provide a class of applications for this result. Namely, we want to introduce certain quiver Grassmannians and show that they are projective BB-filterable GKM-varieties.
%-------------------------------------------------------------------------------------
\section{Generalities on Quiver Grassmannians}\label{sec:GenQuiverRep}
%-------------------------------------------------------------------------------------
We recall here some definitions concerning quivers, their representations and quiver Grassmannians which are required later. For more details we refer the reader to the articles by Cerulli Irelli \cite{Cerulli2011,Cerulli2016} as well as the references therein and the book by Schiffler \cite{Schiffler2014}. 

\begin{defi}A (finite) \textbf{quiver} $Q = (Q_0 , Q_1)$ is an ordered pair where 
\begin{itemize}
\item $Q_0$ is a finite set of vertices, 
\item $Q_1$ is a finite set of oriented edges.  
\end{itemize}
\end{defi}
For an edge $a \in Q_1$, we denote the source of $a$ by $s_a$ and the target by $t_a$.
\begin{defi}\label{def:quiver-rep} Let $Q$ be a quiver. 
\begin{enumerate}
\item A (finite-dimensional) \f{$Q$-representation} $M$ over the field $\Kbb$ is given by $((M^{(i)})_{i\in Q_0}, (M_{a})_{a\in Q_1})$ where
\begin{itemize}
\item $M^{(i)}$ is a (finite-dimensional) $\Kbb$-vector space for any $i\in Q_0$,
\item $M_a: M^{(s_a)}\to M^{({t_a})}$ is a $\Kbb$-linear map for any $a \in Q_1$.
\end{itemize}
\item The \f{dimension vector} of a finite dimensional $Q$-representation $M$ is %the tuple 
\[\bdim M:=(\dim_{\Kbb}M^{(i)})_{i\in Q_0}\in\Z_{\geq 0}^{Q_0}.\]
\item A \f{morphism} between two $Q$-representations $M$ and $N$ is a collection of $\Kbb$-linear maps $(\psi_i:M^{(i)}\to N^{(i)})_{i\in Q_0}$ such that the following diagram commutes.
\begin{center}
\begin{tikzpicture}

	\draw[arrows={-angle 90}, shorten >=9, shorten <=9]  (-.1,0) -- (-.1,-1.6);
	\draw[arrows={-angle 90}, shorten >=9, shorten <=9]  (0,0) -- (1.6,0);
	\draw[arrows={-angle 90}, shorten >=9, shorten <=9]  (1.6,0) -- (1.6,-1.6);
	\draw[arrows={-angle 90}, shorten >=9, shorten <=9]  (0,-1.6) -- (1.6,-1.6);
	
	\node at (0,0) {$M^{(s_a)}$};
	\node at (0,-1.6) {$M^{(t_a)}$};
    \node at (1.7,0) {$N^{(s_a)}$};
    \node at (1.7,-1.6) {$N^{(t_a)}$};
    
    	\node at (-.4,-.8) {$M_{a}$};
	\node at (0.8,.3) {$\psi_{s_a}$};
    \node at (2.0,-.8) {$N_{a}$};
    \node at (0.8,-1.9) {$\psi_{t_a}$};
    
    	\node[rotate=45] at (.8,-.8) {$\equiv$};
	 
\end{tikzpicture}
\end{center}
\end{enumerate}
\end{defi}
Observe that part (1) and (3) of the above definition work with finite and infinite dimensional vector spaces. The above-defined $Q$-representations together with the morphisms among them form a category, which is denoted by $\mathrm{Rep}_\Kbb(Q)$. By $\mathrm{rep}_\Kbb(Q)$ we denote the full subcategory whose objects are the finite-dimensional $Q$-representations. The following theorem tells us that $\mathrm{rep}_\Kbb(Q)$ is Krull-Schmidt.
\begin{trm}\label{trm:KirillovKSThm}(cf.\cite[Theorem~1.11]{Kirillov2016}) Every object of $\mathrm{rep}_\Kbb(Q)$ is isomorphic to a direct sum of indecomposable objects, and this decomposition is unique up to reordering.
\end{trm}
\begin{defi}
Let $Q$ be a quiver and $M$ an object of $\mathrm{rep}_\Kbb(Q)$. 
\begin{enumerate}
\item  A subobject of $M$ is called a \f{subrepresentation}.
\item For any $\mb{e}\in\Z_{\geq 0}^{Q_0}$, the \f{quiver Grassmannian} $\mr{Gr}_{\mb{e}}(M)$ is the variety that parametrises all $\mb{e}$-dimensional subrepresentations of $M$.
\end{enumerate}
\end{defi}
\begin{rem}\label{rem:EmptyQuiverGr}
It is immediate to see that if there is an $i\in Q_0$ such that $\mb{e}_i>\dim_{\Kbb} M^{(i)}$, then $\mr{Gr}_{\mb{e}}(M)$ is empty. We will therefore only consider $\mb{e}$ such that $e_i\leq \dim_{\Kbb}M^{(i)}$ for all $i\in Q_0$. We will denote this relation between dimension vectors by $\mb{e}\leq \bdim M$.
\end{rem}
\begin{rem}\label{rem:QuiverGrIndptBasis}
The algebraic variety structure of the quiver Grassmannian is obtained by embedding it into the classical Grassmannian of $\sum_i e_i$-dimensional subspaces of $V=\bigoplus_{i\in Q_0}M^{(i)}$, therefore it does not depend on the choice of bases for the $M^{(i)}$'s.
\end{rem}
\begin{ex}\label{ex:DynkinAEquioriented}Let $Q$ be the type $\tt A_{n}$ equioriented Dynkin quiver, that is the quiver with $Q_0=\{1, 2, \ldots, n\}$, $Q_1=\{i\to i+1\mid i=1\, \ldots n-1\}$. 

Consider the complex $Q$-representation $M$ given by $M^{(i)}=\C^{n+1}$ for any $i\in Q_0$ and $M_{a}=\mr{id}_{\C^{n+1}}$ for any $a\in Q_1$. Then $\mr{Gr}_{(1,2,\ldots, n)}(M)$ is isomorphic to the variety $\mathcal{F}l_{n+1}$ of complete flags in $\C^{n+1}$.

If we relax the conditions on maps, that is if we consider any complex $Q$-representation $N$ with $N^{(i)}=\C^{n+1}$ for any $i\in Q_0$, then $\mr{Gr}_{(1,2,\ldots, n)}(N)$ is a linear degeneration of $\mathcal{F}l_{n+1}$ (see \cite{CFFFR2017}).
\end{ex}
For any $Q$-representation $M$ and any collection of bases for the vector spaces $M^{(i)}$ with $i\in Q_0$, it is possible to define a new quiver. This will help to provide a combinatorial description of the moment graph of a torus action on the class of quiver Grassmannians, which we will be interested in later.
\begin{defi}\label{def:CoeffQuiver}Let $Q$ be a quiver and let $M$ be an object of $\mathrm{rep}_\Kbb(Q)$. For $i\in Q_0$, let $B^{(i)}:=\{v^{(i)}_k\}$ be a basis of $M^{(i)}$ and let $B:=\bigcup_{i\in Q_0}B^{(i)}$. The \f{coefficient quiver} $Q(M,B)$ is given by
\begin{itemize}
\item $Q(M,B)_0:=B$,
\item $v^{(i)}_k\to v^{(j)}_\ell\in Q(M,B)_1$ if and only if there exists an $a\in Q_1$ such that $s_a=i$, $t_a=j$ and the coefficient of $v^{(j)}_\ell$ in $M_a(v^{(i)}_k)$ is non-zero.
\end{itemize}
\end{defi}
\begin{rem}\label{rem:ConnCompsCoeffQuiver} 
By \cite[Theorem~1.11]{Kirillov2016} (see Theorem~\ref{trm:KirillovKSThm}) every quiver representation is isomorphic to a direct sum of indecomposable quiver representations, which is unique up to the order of the summands. This isomorphism translates to a base change of the representation $M$ and implies that there exists a basis $B$ such that the connected components of the coefficient quiver $Q(M,B)$ are in bijection with the indecomposable summands in the decomposition of $M$. From now on we always work with bases satisfying this property.
\end{rem}
We conclude this subsection by introducing the notion of attractive grading on the vertex set of the coefficient quiver. This is a crucial tool to study cellular decompositions of quiver Grassmannians. 
\begin{defi}\label{defi:attractive-grading}Let $M$ and $B$ be as in Definition~\ref{def:CoeffQuiver}, and let $Q(M,B)$ be the corresponding coefficient quiver.
\begin{enumerate}
\item A \f{grading} on $Q(M,B)_0$ is a tuple $\mb{wt}=\big(\mr{wt}(v^{(i)}_k)\big)\in \mathbb{Z}^B$.
\item A grading $\mb{wt}$ on $Q(M,B)_0$ is \f{attractive} if 
\begin{itemize}
\item[(AG1)] for any $i\in Q_0$ it holds that $\mr{wt}(v^{(i)}_k)>\mr{wt}(v^{(i)}_\ell)$ whenever $k>\ell$,
\item[(AG2)] for any $a\in Q_1$ there exists a weight $d(a)\in\Z$ such that \[\mr{wt}\big(v^{(t_a)}_\ell\big)=\mr{wt}\big(v^{(s_a)}_k\big)+d(a)\] whenever $v^{(s_a)}_k\to v^{(t_a)}_\ell\in Q(M,B)_1$.
\end{itemize}
\end{enumerate}
\end{defi} 
\begin{rem}
For a special class of quiver representations we describe an approach to construct attractive gradings of their coefficient quivers in Proposition~\ref{prop:attractiveGradingDeltan}.
\end{rem}
\begin{rem}\label{rem:weight-functions}
The above definition is inspired by \cite[Theorem~1]{Cerulli2011}, where a grading on $Q(M,B)_0$ with property (AG2) and (AG1) with $"\neq"$ instead of $">"$ is used to define a $\C^*$-action on $\mr{Gr}_\mb{e}(M)$ \cite[Lemma 1.1]{Cerulli2011} via \begin{equation}\label{eq:C*ActionAttractiveGrading} z\cdot b:=z^{\mr{wt}(b)}b\qquad \hbox{ for }z\in\C^*, b\in B. \end{equation}
Looking at the fixed point set of such an action allowed Cerulli Irelli to compute the Euler characteristic of $\mr{Gr}_\mb{e}(M)$ (see \cite[Theorem~1]{Cerulli2011}). His construction was generalised by Haupt \cite[Theorem~1.2]{Haupt2012}.
\end{rem}
\begin{rem}
%\cite{CEFR2018} present a different approach to compute cellular decompositions of quiver Grassmannians, one which does not rely on a $\C^*$-action.
A different approach to compute cellular decompositions of quiver Grassmannians, which does not rely on a $\C^*$-action, is presented in \cite{CEFR2018}. Since we are interested in cellular decompositions which are stable under the action of some larger-rank torus, it is convenient for us to start from a $\C^*$-action on our varieties.
\end{rem}

\section{Nilpotent Representations of the Equioriented Cycle}\label{sec:quiver-grass-for-cycle}
Let $\Delta_n$ denote the equioriented cycle on $n$ vertices with arrows $i \to i+1$ for $i \in {1,\dots,n-1}$ and $n\to1$. The set of vertices and the set of arrows in $\Delta_n$ are in bijection with $\Z_n := \Z/n\Z$.
\begin{defi}An object $M$ of $\mathrm{rep}_\Kbb(\Delta_n)$ is  \f{nilpotent} if there exists an $N\in \Z_{\geq 0}$ such that $M_{a+N}\circ M_{a+N-1}\circ\ldots\circ M_{a}=0$, for any $a\in Q_1=\Z/n\Z$ . The minimal $N$ such that this is satisfied is called \f{nilpotence parameter} of $M$. %An object $M$ of $\mathrm{rep}_\Kbb(\Delta_n)$ is said to be \f{nilpotent} if for all $a\in Q_1=\Z/n\Z$ there exists an $N\in \Z_{\geq 0}$ such that $M_{a+N}\circ M_{a+N-1}\circ\ldots\circ M_{a}=0$. 
\end{defi}
\begin{rem}\label{rem:nilpoRepDeltanGradedVecSpace}
Notice that a representation $M=((M^{(i)})_{i\in\Z/n\Z},(M_a)_{a\in\Z/n\Z})$ of $\Delta_n$, is the same as a $\Z/n\Z$-graded $\Kbb$-vector space $V=\bigoplus_{i\in \Z/n\Z}M^{(i)}$, together with a $\Kbb$-linear operator $A\in\mathrm{End}(V)$ such that $A M^{(i)}=M_{a}(M^{(i)})\subseteq M^{(i+1)}$, for any $i\in \Z/n\Z$ with $a \in Q_1$ such that $i=s_a$. Then $M$ is nilpotent if and only if $A$ is a nilpotent endomorphism. From now on we write $M_i$ for the map along the arrow $a$ with $i=s_a$.
\end{rem}
\begin{ex}\label{ex:NilpoIndcpDeltan}Let $i\in \Z/n\Z$ and let $\ell\in\Z_{\geq 1}$. Consider the $\Kbb$-vector space $V$ with basis $W = \{w_1, \ldots, w_{\ell}\}$ equipped with the $\Z/n\Z$-grading given by $\deg(w_k)=i+k-\ell \in\Z/n\Z$.  Consider moreover the operator $A\in \mathrm{End}(V)$ uniquely determined by setting $Aw_{k}=w_{k+1}$ for any $k < \ell$ and $Aw_{\ell}=0$. The corresponding $\Delta_n$-representation is immediately seen to be nilpotent. We denote this representation by $U(i;\ell)$.
For $n=4$, we draw $\Delta_4$ and the coefficient quivers of $U(1;4)$ and $U(2;3)$:
\begin{center}
\begin{tikzpicture}[scale=.5]
     \node at (-2.5,0) {$\Delta_4=$};

    \node at (1,1) {$1$};  %one

	\node at (1,-1) {$2$}; %two

	\node at (-1,-1) {$3$}; %three
	
	\node at (-1,1) {$4$};   %four

    \draw[arrows={-angle 90}, shorten >=5, shorten <=5]  (1,1) -- (1,-1); % 1-->2
	\draw[arrows={-angle 90}, shorten >=5, shorten <=5]  (1,-1) -- (-1,-1); % 2-->3
    \draw[arrows={-angle 90}, shorten >=5, shorten <=5]  (-1,-1) -- (-1,1); % 3-->4
    \draw[arrows={-angle 90}, shorten >=5, shorten <=5]  (-1,1) -- (1,1); % 4-->1

\end{tikzpicture}
$\quad $
\begin{tikzpicture}[scale=.5]
     \node at (-4,0) {$Q(U(1;4),W)=$};

    \draw[fill=black] (1,1) circle (.2);  %one

	\draw[fill=black] (1,-1) circle (.2); %two

	\draw[fill=black] (-1,-1) circle (.2); %three
	
	\draw[fill=black] (-1,1) circle (.2);   %four

    \draw[arrows={-angle 90}, shorten >=3, shorten <=3]  (1,1) -- (1,-1); % 1-->2
	\draw[arrows={-angle 90}, shorten >=3, shorten <=3]  (1,-1) -- (-1,-1); % 2-->3
    \draw[arrows={-angle 90}, shorten >=3, shorten <=3]  (-1,-1) -- (-1,1); % 3-->4

\end{tikzpicture} 
$\quad $
\begin{tikzpicture}[scale=.5]
    \node at (-4,0) {$Q(U(2;3),W)=$};

    %\draw[fill=black] (1,1) circle (.2);  %one

	\draw[fill=black] (1,-1) circle (.2); %two

	\draw[fill=black] (-1,-1) circle (.2); %three
	
	\draw[fill=black] (-1,1) circle (.2);   %four

    %\draw[arrows={-angle 90}, shorten >=3, shorten <=3]  (1,1) -- (1,-1); % 1-->2
	\draw[arrows={-angle 90}, shorten >=3, shorten <=3]  (1,-1) -- (-1,-1); % 2-->3
    \draw[arrows={-angle 90}, shorten >=3, shorten <=3]  (-1,-1) -- (-1,1); % 3-->4

\end{tikzpicture}
\end{center}
\end{ex}
The following theorem tells us that any indecomposable nilpotent representation of the cycle is isomorphic to some $U(i;\ell)$.
\begin{trm}\label{trm:NilpoIndcpDeltan}(cf.\cite[Theorem 7.6]{Kirillov2016})
\begin{enumerate}
\item The representation $U(i;\ell)$ defined in Example~\ref{ex:NilpoIndcpDeltan} is indecomposable.
\item Let $M$ be an indecomposable nilpotent representation of $\Delta_n$. Then there exist $i\in \Z/n\Z$ and $\ell\in\Z_{>0}$ such that $M\simeq U(i;\ell)$.
\end{enumerate}
\end{trm}
By the above theorem, together with Theorem~\ref{trm:KirillovKSThm}, we deduce that if $M$ is a nilpotent representation of $\Delta_n$, then there exists a nilpotence parameter $N \in \N$ such that
\begin{equation}\label{eqn:decompNilpo} 
M \cong U(\mb{d}):= \bigoplus_{i \in \Z_n} \bigoplus_{\ell=1}^N U(i;\ell) \otimes \C^{d_{i,\ell}}
 \end{equation}
with $d_{i,\ell} \in \Z_{\geq 0}$.
The investigation of torus actions on quiver Grassmannians for nilpotent representations of $\Delta_n$ is the main purpose of the rest of this paper.

\subsection{Coefficient Quivers for Nilpotent Representations of $\Delta_n$}\label{sec:nilpot-rep-cycle} 
Observe that in Example~\ref{ex:NilpoIndcpDeltan} we defined the representation $U(i;\ell)$ by choosing a basis $B=\{w_1, \ldots, w_{\ell}\}$ of the underlying $\Z/n\Z$-graded vector space. This can be obviously rearranged into the union of ordered bases $B^{(j)}$,  for $j\in\Z/n\Z$. We fix these bases once and for all, and therefore we write $Q(U(i;\ell))$ for $Q(U(i;\ell),B)$. Notice that $Q(U(i;\ell))$ is a segment on $\ell$ points, which starts at vertex $v^{(j)}_1$ (for $j = i-\ell+1 \ \mod \ n$) and ends at the vertex $v^{(i)}_k$ (for $k = 1 + \lfloor (\ell-1)/n \rfloor$).

\begin{ex}
With basis as above, the coefficient quiver of $U(i;\ell)$ has the form:
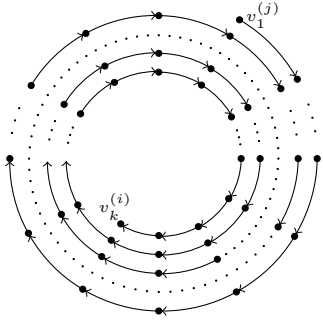
\begin{center}
\begin{tikzpicture}[scale=.6]

  \node at ($(0,0)+(-306:4.7)$) {\tiny $v^{(j)}_1$};%starting point 
  \node at ($(0,0)+(-132:1.7)$) {\tiny $v^{(i)}_k$};%end point

% vertices on the outer ring 
\foreach \ang\dist in {-300/4.5, -330/4.4166, 0/4.3333, -30/4.25, -60/4.1666, -90/4.0833, -120/4.0, -150/3.9166, -180/3.8333, -210/3.75, -240/3.6666, -270/3.5833, -300/3.5, -330/3.4166, 0/3.3333, -30/3.25}{
  \draw[fill=black] ($(0,0)+(\ang:2+0.5*\dist)$) circle (.08);
}

% vertices on the inner ring 
\foreach \ang\dist in {-60/2.1666, -90/2.0833, -120/2.0, -150/1.9166, -210/1.75, -240/1.6666, -270/1.5833, -300/1.5, -330/1.4166, -0/1.3333, -30/1.25, -60/1.1666, -90/1.0833, -120/1.0, -150/0.9166, -210/0.75, -240/0.6666, -270/0.5833, -300/0.5, -330/0.4166, 0/0.3333, -30/0.25, -60/0.1666, -90/0.0833, -120/0.0}{
  \draw[fill=black] ($(0,0)+(\ang:2+0.5*\dist)$) circle (.08);
}

%arrows outer ring 
\foreach \ang\dist in {-300/4.5, 0/4.3333, -30/4.25, -60/4.1666, -90/4.0833, -120/4.0, -150/3.9166, -210/3.75, -240/3.6666, -270/3.5833, -300/3.5, 0/3.3333}{
  %\draw[fill=black] ($(0,0)+(\ang:2+0.5*\dist)$) circle (.08);
  \draw[->,shorten <=1pt, shorten >=0pt] ($(0,0)+(\ang:2+0.5*\dist)$) arc (\ang-1.15:\ang-30:2+0.5*\dist);
}

%arrows inner ring 
\foreach \ang\dist in {-60/2.1666, -90/2.0833, -120/2.0, -150/1.9166, -210/1.75, -240/1.6666, -270/1.5833, -300/1.5, -0/1.3333, -30/1.25, -60/1.1666, -90/1.0833, -120/1.0, -150/0.9166 , -210/0.75, -240/0.6666, -270/0.5833, -300/0.5, 0/0.3333, -30/0.25, -60/0.1666, -90/0.0833}{
  %\draw[fill=black] ($(0,0)+(\ang:2+0.5*\dist)$) circle (.08);
  \draw[->,shorten <=1pt, shorten >=0pt] ($(0,0)+(\ang:2+0.5*\dist)$) arc (\ang-1.15:\ang-29.5:2+0.5*\dist);
}

%dotted snake 
\foreach \ind in {1,2,3,4,5,6,7,8,9,10,11,12,13,14,15,16,17,18,19,20,21,22,23,24,25,26,27,28,29,30,31,32,33,34,35,36}{  
  \draw[fill=black] ($(0,0)+(-30-360*\ind/72:3.5433-\ind/144+11/144)$) circle (.015);
  \draw[fill=black] ($(0,0)+(-30-360*\ind/72-180:3.05-\ind/144+44/144)$) circle (.015);
}

%end of dotted snake 
\foreach \ind in {1,2,3,4,5,6}{  
  \draw[fill=black] ($(0,0)+(-30-360*\ind/72:3.05-\ind/144+11/144)$) circle (.015);
}

%dotted something 
\foreach \ind in {2,3,4}{  
  \draw[fill=black] ($(0,0)+(180-360*\ind/72:3.629-\ind/144+44/144)$) circle (.015);
  \draw[fill=black] ($(0,0)+(180-360*\ind/72:3.629-1-\ind/144+44/144)$) circle (.015);
  \draw[fill=black] ($(0,0)+(180-360*\ind/72:3.629-1.5-\ind/144+44/144)$) circle (.015);
}

%dotted something 
\foreach \ind in {2,3,4}{  
  \draw[fill=black] ($(0,0)+(30-360*\ind/72:4.129-\ind/144+11/144)$) circle (.015);
  \draw[fill=black] ($(0,0)+(30-360*\ind/72:4.129-0.5-\ind/144+11/144)$) circle (.015);
  \draw[fill=black] ($(0,0)+(30-360*\ind/72:4.129-1.5-\ind/144+11/144)$) circle (.015);
  \draw[fill=black] ($(0,0)+(30-360*\ind/72:4.129-2-\ind/144+11/144)$) circle (.015);

}

\end{tikzpicture}
\end{center}
\end{ex}
Let $M$ be a nilpotent representation of $\Delta_n$. By the above discussion, Remark~\ref{rem:ConnCompsCoeffQuiver} and Theorem~\ref{trm:NilpoIndcpDeltan}, we deduce that there exists a basis $B$ such that the connected components of the coefficient quiver are segments, parametrised by a terminal vertex $i$ and a length parameter $\ell$. 
Now we want to rearrange these segments in a particular way, which allows us to prove the existence of attractive gradings on the coefficient quiver (see Proposition~\ref{prop:attractiveGradingDeltan}). We use these gradings to compute a cellular decomposition of $\mr{Gr}_\mb{e}(M)$ as in \cite[Theorem~4.13]{Pue2020}. This new arrangement corresponds to a base change for the representation $M$ and hence does not affect the geometry of the quiver Grassmannian (see Remark~\ref{rem:QuiverGrIndptBasis}).
\begin{defi}
A nilpotent $\Delta_n$-representation $M$ is \f{alignable} if there exists a basis $B$, such that for $Q(M,B)$ the following holds over each $i \in \Z_n$:
\begin{itemize}
\item[(QM0)] $B$ is the union of standard basis of the indecomposable direct summands of $M$.
\item[(QM1)] end points of segments have larger indices than every point with outgoing arrows:
if $M_i v^{(i)}_j = 0$ and $M_i v^{(i)}_k \neq 0$, then $ j > k$. 
%\[ \mr{if} \ M_i v^{i}_j = 0 \ \mr{and} M_i v^{i}_k \neq 0, \ \mr{then} \  j > k. \]
\item[(QM2)] outgoing arrows are order preserving:\\
if $M_i v^{(i)}_j =  v^{(i+1)}_{j'}$ and $M_i v^{(i)}_k = v^{(i+1)}_{k'}$ with $j>k$, then $ j' > k'$. 
\end{itemize}
\end{defi}
\begin{prop}\label{prop:AlignabilityDeltan}
Every nilpotent $\Delta_n$-representation $M$ is alignable.
\end{prop}
\begin{proof}
All nilpotent $\Delta_n$-representations decompose as in \eqref{eqn:decompNilpo}. For $i\in\Z/n\Z$, set  
\[ d_i := \sum_{\ell=1}^N d_{i,\ell}, \quad r_i := d_i - d_{i,1} \quad q_i := \dim_{\Kbb} M^{(i)} - d_i,\]
so that in the coefficient quiver of $M$ there will be $d_i$ segments ending in vertices corresponding to basis elements in $B^{(i)}$. We construct the coefficient quiver inductively, by truncating and then extending the various segments step-by-step.

In Step 1 we draw the coefficient quiver of 
\[ M^1=\bigoplus_{i \in \Z_n} \bigoplus_{\ell=1}^N U(i;1) \otimes \C^{d_{i,\ell}}, \]
which has no edges, and the vertices are $\{v^{(1)}_{q_1+1}, \ldots, v^{(1)}_{q_1+d_1},\ldots, v^{(n)}_{q_n+1},\ldots, v^{(n)}_{q_n+d_n}\}$.

In step 2 we extend the segments of the $U(i;\ell)$ with $\ell \geq 2$, to get the coefficient quiver of 
\[
M^2=\bigoplus_{i \in \Z_n} \Big(U(i;1) \otimes \C^{d_{i,1}}\oplus\bigoplus_{\ell=2}^N U(i;2) \otimes \C^{d_{i,\ell}}\Big),
\]
so we do not touch the vertices
\( \big\{ v^{(i)}_{q_i+r_i+1}, \ldots, v^{(i)}_{q_i+d_i} \ \vert \ i \in \Z_n\big\} \) corresponding to the $U(i;1)$-segments. 

For $i \in \Z_n$ and $k_i \in [r_i]$ each $v^{(i)}_{q_i+k_i}$ is connected via an edge to $v^{(i-1)}_{q_{i-1}+k_i-r_i}$ in $B^{(i-1)}$. This procedure is continued until all segments are fully rearranged. In the $k$-th step we modify segments corresponding to $U(i;\ell)$ with $\ell\geq k$, while the shorter ones are already complete and remain unchanged. Hence the procedure ends after $N$ steps. 
\end{proof}
\begin{rem}\label{rem:aligned-coeff-quiv} Given a nilpotent representation $M$ as before, the aligned coefficient quiver we have obtained is uniquely determined by the decomposition \eqref{eqn:decompNilpo} of $M$, up to the order of segments of the same length, but this does not change the isomorphism type of the graph. Without any ambiguity, we denote it by $Q(M)$. 
\end{rem}
\begin{rem}
Observe that this is not the only way to obtain an aligned coefficient quiver of $M$. For example 
\begin{center}
\begin{tikzpicture}[scale=0.35]
\draw[fill=black] (0,2) circle (.12);
\draw[arrows={-angle 90}, shorten >=2, shorten <=2]  (0,2) -- (1.5,2);
\draw[fill=black] (1.5,2) circle (.12);
\draw[fill=black] (1.5,1) circle (.12);
\end{tikzpicture}$\quad \mr{and} \quad$
\begin{tikzpicture}[scale=0.35]
%\node at (-1.7,2) {$S_1 = $};
\draw[fill=black] (0,2) circle (.12);
\draw[arrows={-angle 90}, shorten >=2, shorten <=2]  (0,2) -- (1.5,2);
\draw[fill=black] (1.5,2) circle (.12);
\draw[fill=black] (1.5,3) circle (.12);
\end{tikzpicture}
\end{center}
are aligned coefficient quivers of the $\Delta_2$-representation $U(2;2)\oplus U(2;1)$. The role of different alignments will be discussed in Example~\ref{ex:fund-mutation}.
The explicit alignment as in Proposition~\ref{prop:AlignabilityDeltan} allows us to prove the existence of attractive gradings by constructing one specific attractive grading of $Q(M)$.
\end{rem}
We hope that an example will make the above construction clear.
\begin{ex}\label{ex:coeffQuiver}
Let $n=4$ and \[M= U(1;4) \oplus U(1;2) \oplus U(2;3) \oplus U(2;2) \oplus U(2;1) \oplus U(4;6).\]
We compute $d_1=2$, $d_2=3$, $d_3=0$ and $d_4=1$.
Following the procedure described above, the coefficient quiver $Q(M)$ is constructed as
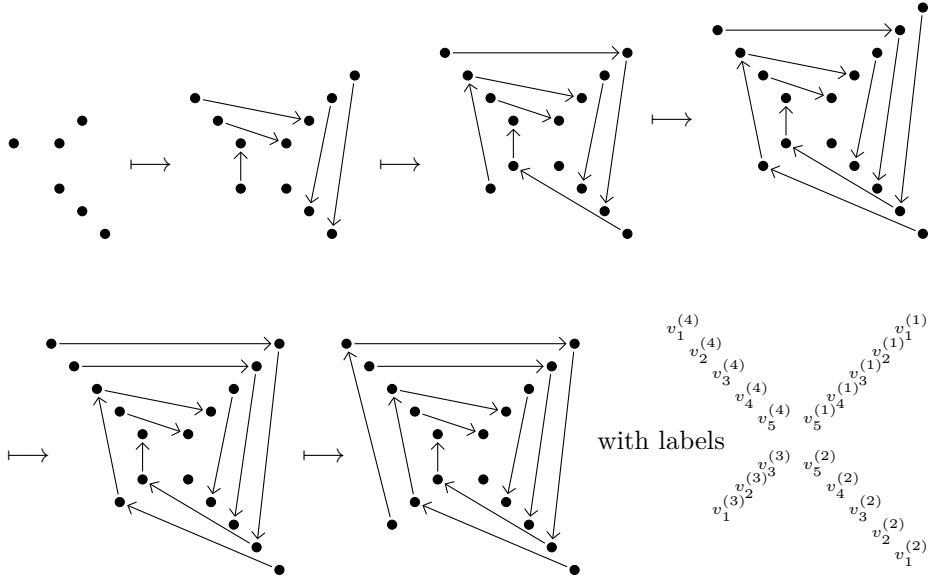
\begin{center}
\begin{tikzpicture}[scale=.3]

    \draw[fill=black] (1,1) circle (.2);  %one
    \draw[fill=black] (2,2) circle (.2);  %one
	\draw[fill=black] (1,-1) circle (.2); %two
	\draw[fill=black] (2,-2) circle (.2); %two
	\draw[fill=black] (3,-3) circle (.2); %two
	%\draw[fill=black] (-1,-1) circle (.2); %three
	\draw[fill=black] (-1,1) circle (.2);   %four

	%\draw[arrows={-angle 90}, shorten >=3, shorten <=3]  (1,1) -- (1,-1); % 1-->2
    \node at (5,0) {$\longmapsto$};

\end{tikzpicture} 
\begin{tikzpicture}[scale=.3]

    \draw[fill=black] (1,1) circle (.2);  %one
    \draw[fill=black] (2,2) circle (.2);  %one
    \draw[fill=black] (3,3) circle (.2);  %one
    \draw[fill=black] (4,4) circle (.2);  %one
	\draw[fill=black] (1,-1) circle (.2); %two
	\draw[fill=black] (2,-2) circle (.2); %two
	\draw[fill=black] (3,-3) circle (.2); %two
	\draw[fill=black] (-1,-1) circle (.2); %three
	\draw[fill=black] (-1,1) circle (.2);   %four
	\draw[fill=black] (-2,2) circle (.2);   %four
	\draw[fill=black] (-3,3) circle (.2);   %four
	
	\draw[arrows={-angle 90}, shorten >=3, shorten <=3]  (4,4) -- (3,-3); % 1-->2
	\draw[arrows={-angle 90}, shorten >=3, shorten <=3]  (3,3) -- (2,-2); % 1-->2

    %\draw[arrows={-angle 90}, shorten >=3, shorten <=3]  (1,-1) -- (-1,-1); % 2-->3
    
    \draw[arrows={-angle 90}, shorten >=3, shorten <=3]  (-1,-1) -- (-1,1); % 3-->4
    
    \draw[arrows={-angle 90}, shorten >=3, shorten <=3]  (-3,3) -- (2,2); % 4-->1
    \draw[arrows={-angle 90}, shorten >=3, shorten <=3]  (-2,2) -- (1,1); % 4-->1
    
    \node at (6,0) {$\longmapsto$};

\end{tikzpicture} 
\begin{tikzpicture}[scale=.3]

    \draw[fill=black] (1,1) circle (.2);  %one
    \draw[fill=black] (2,2) circle (.2);  %one
    \draw[fill=black] (3,3) circle (.2);  %one
    \draw[fill=black] (4,4) circle (.2);  %one
	\draw[fill=black] (1,-1) circle (.2); %two
	\draw[fill=black] (2,-2) circle (.2); %two
	\draw[fill=black] (3,-3) circle (.2); %two
	\draw[fill=black] (4,-4) circle (.2); %two
	\draw[fill=black] (-1,-1) circle (.2); %three
	\draw[fill=black] (-2,-2) circle (.2); %three
	\draw[fill=black] (-1,1) circle (.2);   %four
	\draw[fill=black] (-2,2) circle (.2);   %four
	\draw[fill=black] (-3,3) circle (.2);   %four
	\draw[fill=black] (-4,4) circle (.2);   %four
	
	\draw[arrows={-angle 90}, shorten >=3, shorten <=3]  (4,4) -- (3,-3); % 1-->2
	\draw[arrows={-angle 90}, shorten >=3, shorten <=3]  (3,3) -- (2,-2); % 1-->2

    \draw[arrows={-angle 90}, shorten >=3, shorten <=3]  (4,-4) -- (-1,-1); % 2-->3
    
    \draw[arrows={-angle 90}, shorten >=3, shorten <=3]  (-1,-1) -- (-1,1); % 3-->4
    \draw[arrows={-angle 90}, shorten >=3, shorten <=3]  (-2,-2) -- (-3,3); % 3-->4
    
    \draw[arrows={-angle 90}, shorten >=3, shorten <=3]  (-3,3) -- (2,2); % 4-->1
    \draw[arrows={-angle 90}, shorten >=3, shorten <=3]  (-2,2) -- (1,1); % 4-->1
    \draw[arrows={-angle 90}, shorten >=3, shorten <=3]  (-4,4) -- (4,4); % 4-->1

\end{tikzpicture} 
\begin{tikzpicture}[scale=.3]
	\node at (-6,0) {$\longmapsto$};

    \draw[fill=black] (1,1) circle (.2);  %one
    \draw[fill=black] (2,2) circle (.2);  %one
    \draw[fill=black] (3,3) circle (.2);  %one
    \draw[fill=black] (4,4) circle (.2);  %one
    \draw[fill=black] (5,5) circle (.2);  %one
	\draw[fill=black] (1,-1) circle (.2); %two
	\draw[fill=black] (2,-2) circle (.2); %two
	\draw[fill=black] (3,-3) circle (.2); %two
	\draw[fill=black] (4,-4) circle (.2); %two
	\draw[fill=black] (5,-5) circle (.2); %two
	\draw[fill=black] (-1,-1) circle (.2); %three
	\draw[fill=black] (-2,-2) circle (.2); %three
	\draw[fill=black] (-1,1) circle (.2);   %four
	\draw[fill=black] (-2,2) circle (.2);   %four
	\draw[fill=black] (-3,3) circle (.2);   %four
	\draw[fill=black] (-4,4) circle (.2);   %four
	
	\draw[arrows={-angle 90}, shorten >=3, shorten <=3]  (5,5) -- (4,-4); % 1-->2
	\draw[arrows={-angle 90}, shorten >=3, shorten <=3]  (4,4) -- (3,-3); % 1-->2
	\draw[arrows={-angle 90}, shorten >=3, shorten <=3]  (3,3) -- (2,-2); % 1-->2

    \draw[arrows={-angle 90}, shorten >=3, shorten <=3]  (5,-5) -- (-2,-2); % 2-->3
    \draw[arrows={-angle 90}, shorten >=3, shorten <=3]  (4,-4) -- (-1,-1); % 2-->3
    
    \draw[arrows={-angle 90}, shorten >=3, shorten <=3]  (-1,-1) -- (-1,1); % 3-->4
    \draw[arrows={-angle 90}, shorten >=3, shorten <=3]  (-2,-2) -- (-3,3); % 3-->4
    
    \draw[arrows={-angle 90}, shorten >=3, shorten <=3]  (-3,3) -- (2,2); % 4-->1
    \draw[arrows={-angle 90}, shorten >=3, shorten <=3]  (-2,2) -- (1,1); % 4-->1
    \draw[arrows={-angle 90}, shorten >=3, shorten <=3]  (-4,4) -- (4,4); % 4-->1

\end{tikzpicture} 

$ $

$ $

\begin{tikzpicture}[scale=.3]
	\node at (-6,0) {$\longmapsto$};

    \draw[fill=black] (1,1) circle (.2);  %one
    \draw[fill=black] (2,2) circle (.2);  %one
    \draw[fill=black] (3,3) circle (.2);  %one
    \draw[fill=black] (4,4) circle (.2);  %one
    \draw[fill=black] (5,5) circle (.2);  %one
	\draw[fill=black] (1,-1) circle (.2); %two
	\draw[fill=black] (2,-2) circle (.2); %two
	\draw[fill=black] (3,-3) circle (.2); %two
	\draw[fill=black] (4,-4) circle (.2); %two
	\draw[fill=black] (5,-5) circle (.2); %two
	\draw[fill=black] (-1,-1) circle (.2); %three
	\draw[fill=black] (-2,-2) circle (.2); %three
	\draw[fill=black] (-1,1) circle (.2);   %four
	\draw[fill=black] (-2,2) circle (.2);   %four
	\draw[fill=black] (-3,3) circle (.2);   %four
	\draw[fill=black] (-4,4) circle (.2);   %four
	\draw[fill=black] (-5,5) circle (.2);   %four
	
	\draw[arrows={-angle 90}, shorten >=3, shorten <=3]  (5,5) -- (4,-4); % 1-->2
	\draw[arrows={-angle 90}, shorten >=3, shorten <=3]  (4,4) -- (3,-3); % 1-->2
	\draw[arrows={-angle 90}, shorten >=3, shorten <=3]  (3,3) -- (2,-2); % 1-->2

    \draw[arrows={-angle 90}, shorten >=3, shorten <=3]  (5,-5) -- (-2,-2); % 2-->3
    \draw[arrows={-angle 90}, shorten >=3, shorten <=3]  (4,-4) -- (-1,-1); % 2-->3
    
    \draw[arrows={-angle 90}, shorten >=3, shorten <=3]  (-2,-2) -- (-3,3); % 3-->4
    \draw[arrows={-angle 90}, shorten >=3, shorten <=3]  (-1,-1) -- (-1,1); % 3-->4
       
    \draw[arrows={-angle 90}, shorten >=3, shorten <=3]  (-5,5) -- (5,5); % 4-->1
    \draw[arrows={-angle 90}, shorten >=3, shorten <=3]  (-4,4) -- (4,4); % 4-->1
    \draw[arrows={-angle 90}, shorten >=3, shorten <=3]  (-3,3) -- (2,2); % 4-->1
    \draw[arrows={-angle 90}, shorten >=3, shorten <=3]  (-2,2) -- (1,1); % 4-->1
    
\end{tikzpicture} 
\begin{tikzpicture}[scale=.3]
	\node at (-6,0) {$\longmapsto$};

    \draw[fill=black] (1,1) circle (.2);  %one
    \draw[fill=black] (2,2) circle (.2);  %one
    \draw[fill=black] (3,3) circle (.2);  %one
    \draw[fill=black] (4,4) circle (.2);  %one
    \draw[fill=black] (5,5) circle (.2);  %one
	\draw[fill=black] (1,-1) circle (.2); %two
	\draw[fill=black] (2,-2) circle (.2); %two
	\draw[fill=black] (3,-3) circle (.2); %two
	\draw[fill=black] (4,-4) circle (.2); %two
	\draw[fill=black] (5,-5) circle (.2); %two
	\draw[fill=black] (-1,-1) circle (.2); %three
	\draw[fill=black] (-2,-2) circle (.2); %three
	\draw[fill=black] (-3,-3) circle (.2); %three
	\draw[fill=black] (-1,1) circle (.2);   %four
	\draw[fill=black] (-2,2) circle (.2);   %four
	\draw[fill=black] (-3,3) circle (.2);   %four
	\draw[fill=black] (-4,4) circle (.2);   %four
	\draw[fill=black] (-5,5) circle (.2);   %four
	
	\draw[arrows={-angle 90}, shorten >=3, shorten <=3]  (5,5) -- (4,-4); % 1-->2
	\draw[arrows={-angle 90}, shorten >=3, shorten <=3]  (4,4) -- (3,-3); % 1-->2
	\draw[arrows={-angle 90}, shorten >=3, shorten <=3]  (3,3) -- (2,-2); % 1-->2

    \draw[arrows={-angle 90}, shorten >=3, shorten <=3]  (5,-5) -- (-2,-2); % 2-->3
    \draw[arrows={-angle 90}, shorten >=3, shorten <=3]  (4,-4) -- (-1,-1); % 2-->3
    
    \draw[arrows={-angle 90}, shorten >=3, shorten <=3]  (-3,-3) -- (-5,5); % 3-->4
    \draw[arrows={-angle 90}, shorten >=3, shorten <=3]  (-2,-2) -- (-3,3); % 3-->4
    \draw[arrows={-angle 90}, shorten >=3, shorten <=3]  (-1,-1) -- (-1,1); % 3-->4
       
    \draw[arrows={-angle 90}, shorten >=3, shorten <=3]  (-5,5) -- (5,5); % 4-->1
    \draw[arrows={-angle 90}, shorten >=3, shorten <=3]  (-4,4) -- (4,4); % 4-->1
    \draw[arrows={-angle 90}, shorten >=3, shorten <=3]  (-3,3) -- (2,2); % 4-->1
    \draw[arrows={-angle 90}, shorten >=3, shorten <=3]  (-2,2) -- (1,1); % 4-->1
    
\end{tikzpicture} 
\begin{tikzpicture}[scale=.3]
	\node at (-6,0) { with labels};

    \node at (1,1) {\tiny $v^{(1)}_5$};  %one
    \node at (2,2) {\tiny $v^{(1)}_4$};  %one
    \node at (3,3) {\tiny $v^{(1)}_3$};  %one
    \node at (4,4) {\tiny $v^{(1)}_2$};  %one
    \node at (5,5) {\tiny $v^{(1)}_1$}; %one
	\node at (1,-1) {\tiny $v^{(2)}_5$}; %two
	\node at (2,-2) {\tiny $v^{(2)}_4$}; %two
	\node at (3,-3) {\tiny $v^{(2)}_3$}; %two
	\node at (4,-4) {\tiny $v^{(2)}_2$}; %two
	\node at (5,-5) {\tiny $v^{(2)}_1$}; %two
	\node at (-1,-1) {\tiny $v^{(3)}_3$}; %three
	\node at (-2,-2) {\tiny $v^{(3)}_2$}; %three
	\node at (-3,-3) {\tiny $v^{(3)}_1$}; %three
	\node at (-1,1) {\tiny $v^{(4)}_5$};   %four
	\node at (-2,2) {\tiny $v^{(4)}_4$};   %four
	\node at (-3,3) {\tiny $v^{(4)}_3$};   %four
	\node at (-4,4) {\tiny $v^{(4)}_2$};   %four
	\node at (-5,5) {\tiny $v^{(4)}_1$};   %four
	
	%\draw[arrows={-angle 90}, shorten >=6, shorten <=6]  (5,5) -- (4,-4); % 1-->2
	%\draw[arrows={-angle 90}, shorten >=6, shorten <=6]  (4,4) -- (3,-3); % 1-->2
	%\draw[arrows={-angle 90}, shorten >=6, shorten <=6]  (3,3) -- (2,-2); % 1-->2

    %\draw[arrows={-angle 90}, shorten >=6, shorten <=6]  (5,-5) -- (-2,-2); % 2-->3
    %\draw[arrows={-angle 90}, shorten >=6, shorten <=6]  (4,-4) -- (-1,-1); % 2-->3
    
    %\draw[arrows={-angle 90}, shorten >=6, shorten <=6]  (-3,-3) -- (-5,5); % 3-->4
    %\draw[arrows={-angle 90}, shorten >=6, shorten <=6]  (-2,-2) -- (-3,3); % 3-->4
    %\draw[arrows={-angle 90}, shorten >=6, shorten <=6]  (-1,-1) -- (-1,1); % 3-->4
       
    %\draw[arrows={-angle 90}, shorten >=6, shorten <=6]  (-5,5) -- (5,5); % 4-->1
    %\draw[arrows={-angle 90}, shorten >=6, shorten <=6]  (-4,4) -- (4,4); % 4-->1
    %\draw[arrows={-angle 90}, shorten >=6, shorten <=6]  (-3,3) -- (2,2); % 4-->1
    %\draw[arrows={-angle 90}, shorten >=6, shorten <=6]  (-2,2) -- (1,1); % 4-->1
    
\end{tikzpicture}
\end{center}
\end{ex}
%-------------------------------------------------------------------------------------
\section{Torus Actions}\label{sec:TorusAction}
\subsection{$\C^*$-action and Cellular Decomposition}
\label{sec:C*-action}
%-------------------------------------------------------------------------------------
Now we describe one explicit attractive grading of the coefficient quiver of any nilpotent representation of the cycle.
\begin{prop}\label{prop:attractiveGradingDeltan}
Let $M$ be a nilpotent representation of $\Delta_n$ with decomposition as in \eqref{eqn:decompNilpo}. 
There exists an attractive grading of $Q(M)$ with (constant) weight function on the edges given by 
\[d(a):= D := \mr{max}\{ d_i - d_{i,1} \ \vert \ i \in \Z_n\}\qquad \hbox{ for all edges } a \in \Z/n\Z.
\] 
\end{prop}
\begin{proof} Take $i_0 \in \Z_n$ so that $d_{i_0,N} \geq d_{i,N}$ for all $i \in \Z_n$, this is equivalent to picking a vertex of $\Delta_n$ such that the number of segments of length $N$ ending on that vertex is maximal. This choice is not unique and indeed the grading depends on it. 

Set $j_0 := i_0 -N +1 \mod n$ and define $\mr{wt}(v_1^{(j_0)}):=1$. Condition (AG2), with $d(a)$ as in the statement of the proposition, uniquely determines the weights on any vertex belonging to the (length $N$) segment starting in $v_1^{(j_0)}$. Notice that such a segment ends in $v^{(i_0)}_{q_{i_0}+1}$ (cf.(QM1)). Observe that (AG2) implies in particular $w:=\mr{wt}(v^{(i_0)}_{q_{i_0}+1})=1+D(N-1)$.
Next, we let $k:=w+d_{i_0}-d_{i_0,1}$ and, for any $i\in\Z/n\Z$, set
\[ \mr{wt}\big(v_{q_{i}+p}^{(i)}\big) := k+p-1+ d_{i,1} -d_{i}, \qquad \hbox{ for all } p \in [d_i]. \]
Observe that if $i=i_0$ and $p=1$ we obtain the already defined weight of $v^{(i_0)}_{q_{i_0}+1}$. This formula allows to compute the weight of the end point of any segment. The remaining weights are determined by imposing (AG2).  

To conclude, we have to show the attractiveness of the above defined grading. Observe that by definition of the grading, (AG2) is automatically satisfied. Therefore, we only have to make sure that also (AG1) holds, that is $\mr{wt}(v_{h+1}^{(i)})>\mr{wt}(v_{h}^{(i)})$ for any $i \in \Z/n\/Z$ and any $h < \dim_\Kbb M^{(i)}$. We prove this by induction on the length $s$ of the segments in
$Q(M)$, just as we did in the construction of $Q(M)$. If we restrict to the vertices belonging to $B^{(i)}$ for some $i\in \Z/n\Z$, it is clear for the end points that all weights are distinct and strictly increasing with the indices of the basis vectors. We hence assume that if we consider the truncated representation $M^{s-1}$ for $s> 1$ and restrict the grading to its coefficient quiver $Q(M^{s-1})$, we obtain an attractive grading. 

Recall that to get $Q(M^{s})$ we have to add an arrow and its starting point to all segments corresponding to isotypical components $U(i;\ell)$ with $\ell\geq s$. 
Consider $v_{h+1}^{(i)}, v_{h}^{(i)}\in B^{(i)}$ such that they both are vertices of $Q(M^s)$ and $v_{h}^{(i)}\not\in Q(M^{s-1})_0$ (otherwise the claim follows immediately by induction). If $v_{h+1}^{(i)}$ is not an end point, then $v_{h}^{(i)}$ and $v_{h+1}^{(i)}$ are sources of two arrows, say $a_1$ and $a_2$, respectively, whose targets lie in $B^{(i+1)}$ and by induction $\mr{wt}(t_{a_1})<\mr{wt}(t_{a_2})$. Thus, by (AG2),
 \[\mr{wt}\big(v_{h+1}^{(i)}\big)=\mr{wt}(t_{a_2})-D>\mr{wt}(t_{a_1})-D=\mr{wt}\big(v_h^{(i)}\big).\]
Assume now that $v_{h+1}^{(i)}$ is an end point. In this case, as $v_h^{(i)}$ is not an end point, we have $h=q_i$ and hence $\mr{wt}(v_{h+1}^{(i)})=k+d_{i,1}-d_{i}$. Now recall that $v_h^{(i)}$ belongs to a segment whose end point is $v^{(j)}_{q_j+p}\in B^{(j)}$ with $j=i+s-1 \mod n$ and $p\in[d_j-d_{j,1}]$, so that
\begin{align*}
\mr{wt}\big(v_h^{(i)}\big)=\mr{wt}\big(v^{(j)}_{q_j+p}\big)-D(s-1)&=k+p-1+d_{j,1}-d_j-D(s-1)\\
&\leq k-1-D(s-1).    
\end{align*}

The claim now follows from the fact that $D\geq d_{i}-d_{i,1}$ and $s-1\geq 1$.
\end{proof}

\begin{ex}
We compute the weights for the vertices in the coefficient quiver $Q(M)$ of the representation $M$ from Example~\ref{ex:coeffQuiver}. The edge weight is $D=2$. We now determine the attractive grading, following the procedure described in the proof of Proposition~\ref{prop:attractiveGradingDeltan}. There is a unique segment of length $N=6$, which corresponds to the subrepresentation $U(4;6)$, and hence we take $i_0=4$ and compute $d_{i_0}-d_{i_{0,1}}=1$.
\begin{center}
\begin{tikzpicture}[scale=.28]
	%\node at (-6,0) {$\longmapsto$};

    \draw[fill=black] (1,1) circle (.2);  %one
    \draw[fill=black] (2,2) circle (.2);  %one
    \draw[fill=black] (3,3) circle (.2);  %one
    \draw[fill=black] (4,4) circle (.2);  %one
    \draw[fill=black] (5,5) circle (.2);  %one
	\draw[fill=black] (1,-1) circle (.2); %two
	\draw[fill=black] (2,-2) circle (.2); %two
	\draw[fill=black] (3,-3) circle (.2); %two
	\draw[fill=black] (4,-4) circle (.2); %two
	\draw[fill=black] (5,-5) circle (.2); %two
	\draw[fill=black] (-1,-1) circle (.2); %three
	\draw[fill=black] (-2,-2) circle (.2); %three
	%\draw[fill=black] (-3,-3) circle (.2); %three
	\node at (-3,-3) {\tiny $1$};
	\draw[fill=black] (-1,1) circle (.2);   %four
	\draw[fill=black] (-2,2) circle (.2);   %four
	\draw[fill=black] (-3,3) circle (.2);   %four
	\draw[fill=black] (-4,4) circle (.2);   %four
	\draw[fill=black] (-5,5) circle (.2);   %four
	
	\draw[arrows={-angle 90}, shorten >=3, shorten <=3]  (5,5) -- (4,-4); % 1-->2
	\draw[arrows={-angle 90}, shorten >=3, shorten <=3]  (4,4) -- (3,-3); % 1-->2
	\draw[arrows={-angle 90}, shorten >=3, shorten <=3]  (3,3) -- (2,-2); % 1-->2

    \draw[arrows={-angle 90}, shorten >=3, shorten <=3]  (5,-5) -- (-2,-2); % 2-->3
    \draw[arrows={-angle 90}, shorten >=3, shorten <=3]  (4,-4) -- (-1,-1); % 2-->3
    
    \draw[arrows={-angle 90}, shorten >=3, shorten <=3]  (-3,-3) -- (-5,5); % 3-->4
    \draw[arrows={-angle 90}, shorten >=3, shorten <=3]  (-2,-2) -- (-3,3); % 3-->4
    \draw[arrows={-angle 90}, shorten >=3, shorten <=3]  (-1,-1) -- (-1,1); % 3-->4
       
    \draw[arrows={-angle 90}, shorten >=3, shorten <=3]  (-5,5) -- (5,5); % 4-->1
    \draw[arrows={-angle 90}, shorten >=3, shorten <=3]  (-4,4) -- (4,4); % 4-->1
    \draw[arrows={-angle 90}, shorten >=3, shorten <=3]  (-3,3) -- (2,2); % 4-->1
    \draw[arrows={-angle 90}, shorten >=3, shorten <=3]  (-2,2) -- (1,1); % 4-->1
    
\end{tikzpicture} 
\begin{tikzpicture}[scale=.28]
	\node at (-6,0) {$\longmapsto$};

    \draw[fill=black] (1,1) circle (.2);  %one
    \draw[fill=black] (2,2) circle (.2);  %one
    \draw[fill=black] (3,3) circle (.2);  %one
    \draw[fill=black] (4,4) circle (.2);  %one
    \node at (5,5) {\tiny $5$};
	\draw[fill=black] (1,-1) circle (.2); %two
	\draw[fill=black] (2,-2) circle (.2); %two
	\draw[fill=black] (3,-3) circle (.2); %two
	\node at (4,-4) {\tiny $7$};
	\draw[fill=black] (5,-5) circle (.2); %two
	\node at (-1,-1) {\tiny $9$};
	\draw[fill=black] (-2,-2) circle (.2); %three
	%\draw[fill=black] (-3,-3) circle (.2); %three
	\node at (-3,-3) {\tiny $1$};
	\node at (-1,1) {\tiny $11$};
	\draw[fill=black] (-2,2) circle (.2);   %four
	\draw[fill=black] (-3,3) circle (.2);   %four
	\draw[fill=black] (-4,4) circle (.2);   %four
	\node at (-5,5) {\tiny $3$};
	
	\draw[arrows={-angle 90}, shorten >=3, shorten <=3]  (5,5) -- (4,-4); % 1-->2
	\draw[arrows={-angle 90}, shorten >=3, shorten <=3]  (4,4) -- (3,-3); % 1-->2
	\draw[arrows={-angle 90}, shorten >=3, shorten <=3]  (3,3) -- (2,-2); % 1-->2

    \draw[arrows={-angle 90}, shorten >=3, shorten <=3]  (5,-5) -- (-2,-2); % 2-->3
    \draw[arrows={-angle 90}, shorten >=3, shorten <=3]  (4,-4) -- (-1,-1); % 2-->3
    
    \draw[arrows={-angle 90}, shorten >=3, shorten <=3]  (-3,-3) -- (-5,5); % 3-->4
    \draw[arrows={-angle 90}, shorten >=3, shorten <=3]  (-2,-2) -- (-3,3); % 3-->4
    \draw[arrows={-angle 90}, shorten >=3, shorten <=3]  (-1,-1) -- (-1,1); % 3-->4
       
    \draw[arrows={-angle 90}, shorten >=3, shorten <=3]  (-5,5) -- (5,5); % 4-->1
    \draw[arrows={-angle 90}, shorten >=3, shorten <=3]  (-4,4) -- (4,4); % 4-->1
    \draw[arrows={-angle 90}, shorten >=3, shorten <=3]  (-3,3) -- (2,2); % 4-->1
    \draw[arrows={-angle 90}, shorten >=3, shorten <=3]  (-2,2) -- (1,1); % 4-->1
    
\end{tikzpicture} 
\begin{tikzpicture}[scale=.28]
	\node at (-6,0) {$\longmapsto$};

    \node at (1,1) {\tiny $11$};
    \node at (2,2) {\tiny $10$};
    \draw[fill=black] (3,3) circle (.2);  %one
    \draw[fill=black] (4,4) circle (.2);  %one
    \node at (5,5) {\tiny $5$};
	\node at (1,-1) {\tiny $12$};
	\node at (2,-2) {\tiny $11$};
	\node at (3,-3) {\tiny $10$};
	\node at (4,-4) {\tiny $7$};
	\draw[fill=black] (5,-5) circle (.2); %two
	\node at (-1,-1) {\tiny $9$};
	\draw[fill=black] (-2,-2) circle (.2); %three
	%\draw[fill=black] (-3,-3) circle (.2); %three
	\node at (-3,-3) {\tiny $1$};
	\node at (-1,1) {\tiny $11$};
	\draw[fill=black] (-2,2) circle (.2);   %four
	\draw[fill=black] (-3,3) circle (.2);   %four
	\draw[fill=black] (-4,4) circle (.2);   %four
	\node at (-5,5) {\tiny $3$};
	
	\draw[arrows={-angle 90}, shorten >=3, shorten <=3]  (5,5) -- (4,-4); % 1-->2
	\draw[arrows={-angle 90}, shorten >=3, shorten <=3]  (4,4) -- (3,-3); % 1-->2
	\draw[arrows={-angle 90}, shorten >=3, shorten <=3]  (3,3) -- (2,-2); % 1-->2

    \draw[arrows={-angle 90}, shorten >=3, shorten <=3]  (5,-5) -- (-2,-2); % 2-->3
    \draw[arrows={-angle 90}, shorten >=3, shorten <=3]  (4,-4) -- (-1,-1); % 2-->3
    
    \draw[arrows={-angle 90}, shorten >=3, shorten <=3]  (-3,-3) -- (-5,5); % 3-->4
    \draw[arrows={-angle 90}, shorten >=3, shorten <=3]  (-2,-2) -- (-3,3); % 3-->4
    \draw[arrows={-angle 90}, shorten >=3, shorten <=3]  (-1,-1) -- (-1,1); % 3-->4
       
    \draw[arrows={-angle 90}, shorten >=3, shorten <=3]  (-5,5) -- (5,5); % 4-->1
    \draw[arrows={-angle 90}, shorten >=3, shorten <=3]  (-4,4) -- (4,4); % 4-->1
    \draw[arrows={-angle 90}, shorten >=3, shorten <=3]  (-3,3) -- (2,2); % 4-->1
    \draw[arrows={-angle 90}, shorten >=3, shorten <=3]  (-2,2) -- (1,1); % 4-->1
    
\end{tikzpicture} 

$ $

\begin{tikzpicture}[scale=.28]
	\node at (-6,0) {$\longmapsto$};

    \node at (1,1) {\tiny $11$};%one
    \node at (2,2) {\tiny $10$};
    \node at (3,3) {\tiny $9$};
    \node at (4,4) {\tiny $8$};
    \node at (5,5) {\tiny $5$};
	\node at (1,-1) {\tiny $12$};%two
	\node at (2,-2) {\tiny $11$};
	\node at (3,-3) {\tiny $10$};
	\node at (4,-4) {\tiny $7$};
	\draw[fill=black] (5,-5) circle (.2); %two
	\node at (-1,-1) {\tiny $9$};
	\draw[fill=black] (-2,-2) circle (.2); %three
	%\draw[fill=black] (-3,-3) circle (.2); %three
	\node at (-3,-3) {\tiny $1$}; %four
	\node at (-1,1) {\tiny $11$};
	\node at (-2,2) {\tiny $9$};
	\node at (-3,3) {\tiny $8$};
	\draw[fill=black] (-4,4) circle (.2);   %four
	\node at (-5,5) {\tiny $3$};

	\draw[arrows={-angle 90}, shorten >=3, shorten <=3]  (5,5) -- (4,-4); % 1-->2
	\draw[arrows={-angle 90}, shorten >=3, shorten <=3]  (4,4) -- (3,-3); % 1-->2
	\draw[arrows={-angle 90}, shorten >=3, shorten <=3]  (3,3) -- (2,-2); % 1-->2

    \draw[arrows={-angle 90}, shorten >=3, shorten <=3]  (5,-5) -- (-2,-2); % 2-->3
    \draw[arrows={-angle 90}, shorten >=3, shorten <=3]  (4,-4) -- (-1,-1); % 2-->3
    
    \draw[arrows={-angle 90}, shorten >=3, shorten <=3]  (-3,-3) -- (-5,5); % 3-->4
    \draw[arrows={-angle 90}, shorten >=3, shorten <=3]  (-2,-2) -- (-3,3); % 3-->4
    \draw[arrows={-angle 90}, shorten >=3, shorten <=3]  (-1,-1) -- (-1,1); % 3-->4
       
    \draw[arrows={-angle 90}, shorten >=3, shorten <=3]  (-5,5) -- (5,5); % 4-->1
    \draw[arrows={-angle 90}, shorten >=3, shorten <=3]  (-4,4) -- (4,4); % 4-->1
    \draw[arrows={-angle 90}, shorten >=3, shorten <=3]  (-3,3) -- (2,2); % 4-->1
    \draw[arrows={-angle 90}, shorten >=3, shorten <=3]  (-2,2) -- (1,1); % 4-->1
    
\end{tikzpicture} 
\begin{tikzpicture}[scale=.28]
	\node at (-6,0) {$\longmapsto$};

    \node at (1,1) {\tiny $11$};%one
    \node at (2,2) {\tiny $10$};
    \node at (3,3) {\tiny $9$};
    \node at (4,4) {\tiny $8$};
    \node at (5,5) {\tiny $5$};
	\node at (1,-1) {\tiny $12$};%two
	\node at (2,-2) {\tiny $11$};
	\node at (3,-3) {\tiny $10$};
	\node at (4,-4) {\tiny $7$};
	\draw[fill=black] (5,-5) circle (.2); %two
	\node at (-1,-1) {\tiny $9$};
	\node at (-2,-2) {\tiny $6$}; %three
	\node at (-3,-3) {\tiny $1$}; %four
	\node at (-1,1) {\tiny $11$};
	\node at (-2,2) {\tiny $9$};
	\node at (-3,3) {\tiny $8$};
	\node at (-4,4) {\tiny $6$};   %four
	\node at (-5,5) {\tiny $3$};

	\draw[arrows={-angle 90}, shorten >=3, shorten <=3]  (5,5) -- (4,-4); % 1-->2
	\draw[arrows={-angle 90}, shorten >=3, shorten <=3]  (4,4) -- (3,-3); % 1-->2
	\draw[arrows={-angle 90}, shorten >=3, shorten <=3]  (3,3) -- (2,-2); % 1-->2

    \draw[arrows={-angle 90}, shorten >=3, shorten <=3]  (5,-5) -- (-2,-2); % 2-->3
    \draw[arrows={-angle 90}, shorten >=3, shorten <=3]  (4,-4) -- (-1,-1); % 2-->3
    
    \draw[arrows={-angle 90}, shorten >=3, shorten <=3]  (-3,-3) -- (-5,5); % 3-->4
    \draw[arrows={-angle 90}, shorten >=3, shorten <=3]  (-2,-2) -- (-3,3); % 3-->4
    \draw[arrows={-angle 90}, shorten >=3, shorten <=3]  (-1,-1) -- (-1,1); % 3-->4
       
    \draw[arrows={-angle 90}, shorten >=3, shorten <=3]  (-5,5) -- (5,5); % 4-->1
    \draw[arrows={-angle 90}, shorten >=3, shorten <=3]  (-4,4) -- (4,4); % 4-->1
    \draw[arrows={-angle 90}, shorten >=3, shorten <=3]  (-3,3) -- (2,2); % 4-->1
    \draw[arrows={-angle 90}, shorten >=3, shorten <=3]  (-2,2) -- (1,1); % 4-->1
    
\end{tikzpicture} 
\begin{tikzpicture}[scale=.28]
	\node at (-6,0) {$\longmapsto$};

    \node at (1,1) {\tiny $11$};%one
    \node at (2,2) {\tiny $10$};
    \node at (3,3) {\tiny $9$};
    \node at (4,4) {\tiny $8$};
    \node at (5,5) {\tiny $5$};
	\node at (1,-1) {\tiny $12$};%two
	\node at (2,-2) {\tiny $11$};
	\node at (3,-3) {\tiny $10$};
	\node at (4,-4) {\tiny $7$};
	\node at (5,-5) {\tiny $4$}; %two
	\node at (-1,-1) {\tiny $9$};
	\node at (-2,-2) {\tiny $6$}; %three
	\node at (-3,-3) {\tiny $1$}; %four
	\node at (-1,1) {\tiny $11$};
	\node at (-2,2) {\tiny $9$};
	\node at (-3,3) {\tiny $8$};
	\node at (-4,4) {\tiny $6$};   %four
	\node at (-5,5) {\tiny $3$};
	
	\draw[arrows={-angle 90}, shorten >=3, shorten <=3]  (5,5) -- (4,-4); % 1-->2
	\draw[arrows={-angle 90}, shorten >=3, shorten <=3]  (4,4) -- (3,-3); % 1-->2
	\draw[arrows={-angle 90}, shorten >=3, shorten <=3]  (3,3) -- (2,-2); % 1-->2

    \draw[arrows={-angle 90}, shorten >=3, shorten <=3]  (5,-5) -- (-2,-2); % 2-->3
    \draw[arrows={-angle 90}, shorten >=3, shorten <=3]  (4,-4) -- (-1,-1); % 2-->3
    
    \draw[arrows={-angle 90}, shorten >=3, shorten <=3]  (-3,-3) -- (-5,5); % 3-->4
    \draw[arrows={-angle 90}, shorten >=3, shorten <=3]  (-2,-2) -- (-3,3); % 3-->4
    \draw[arrows={-angle 90}, shorten >=3, shorten <=3]  (-1,-1) -- (-1,1); % 3-->4
       
    \draw[arrows={-angle 90}, shorten >=3, shorten <=3]  (-5,5) -- (5,5); % 4-->1
    \draw[arrows={-angle 90}, shorten >=3, shorten <=3]  (-4,4) -- (4,4); % 4-->1
    \draw[arrows={-angle 90}, shorten >=3, shorten <=3]  (-3,3) -- (2,2); % 4-->1
    \draw[arrows={-angle 90}, shorten >=3, shorten <=3]  (-2,2) -- (1,1); % 4-->1
    
\end{tikzpicture} 
\end{center}
\end{ex}

Let $M$ be a nilpotent representation of $\Delta_n$ and let $(\mr{wt}(b))_{b\in B}$ be an attractive grading on an aligned $Q(M,B)$. Define a $\C^*$-action on $M$ by the formula \eqref{eq:C*ActionAttractiveGrading}. It is immediate to check that all hypotheses of \cite[Lemma~1.1]{Cerulli2011} are satisfied, hence the $\C^*$-action extends to the quiver Grassmannian:
\begin{lem}\label{lem:C-action-extends-to-quiver-grass}Let $M$ be a nilpotent representation of $\Delta_n$ and let $(\mr{wt}(b))_{b\in B}$ be an attractive grading on $Q(M,B)_0$. Then for any $U\in \mr{Gr}_\mb{e}(M)$ and any $z\in \C^*$, also $z\cdot U\in \mr{Gr}_\mb{e}(M)$.
\end{lem}
It is also possible to describe the fixed points of the introduced $\C^*$-action. As usual, $B^{(i)}=\big\{v_k^{(i)}\mid k\in [m_i]\big\}$ denotes the basis of $M^{(i)}$, which we use to construct the aligned $Q(M,B)$. The following lemma is just a special case of \cite[Theorem~1]{Cerulli2011}.

\begin{lem}\label{lem:fixedPtsAttractiveGrading}Let $M$ be a nilpotent $\Delta_n$-representation and let $(\mr{wt}(b))_{b\in B}$ be an attractive grading on $Q(M,B)$. Then, the fixed point set of the above defined $\C^*$-action is 
\[
\left\{L\in\mr{Gr}_\mb{e}(M)\mid \fa i \in \Z_n \ \mr{there\ are\ } K_i \in \binom{[m_i]}{e_i} \mr{\ s.t.\ }L^{(i)}=\big\langle v^{(i)}_{k}\vert k \in K_i \big\rangle \right\}.
\] %\left\{U\in\mr{Gr}_\mb{e}(M)\mid \exists k_1\ldots, k_{e_i}\mr{\ such\ that\ }U^{(i)}=\langle v^{(i)}_{k_1},\ldots, v^{(i)}_{k_{m_i}} \rangle\right\}.
\end{lem}
% and it induces a cellular decomposition of the quiver Grassmannian \cite[Theorem~4.13]{Pue2020}. 
\begin{ex}\label{ex:C-fixed-points}
Let $M = U(2;2)\otimes \mathbb{C}^2 \oplus U(2;1)\otimes \mathbb{C}^2$ a representation of $\Delta_2$ and set $\mathbf{e}:=(1,2)$. Following the construction in the proof of Proposition~\ref{prop:AlignabilityDeltan}, the aligned coefficient quiver of $M$ together with the corresponding basis vectors is given as
\begin{center}
\begin{tikzpicture}[scale=.5]
	
	\draw[fill=black] (1,-1) circle (.15); 
	\draw[fill=black] (1,-2) circle (.15); 

	\draw[fill=black] (2,-1) circle (.15); 
	\draw[fill=black] (2,-2) circle (.15); 
	\draw[fill=black] (2,-3) circle (.15); 
	\draw[fill=black] (2,-4) circle (.15);

	\draw[arrows={-angle 90}, shorten >=2, shorten <=2]  (1,-1) -- (2,-1); 
	\draw[arrows={-angle 90}, shorten >=2, shorten <=2]  (1,-2) -- (2,-2); 
	
    \node at (0,-1) {\tiny $v^{(1)}_1$};  
    \node at (0,-2) {\tiny $v^{(1)}_2$};  
    
    \node at (3,-1) {\tiny $v^{(2)}_1$};  
    \node at (3,-2) {\tiny $v^{(2)}_2$};  
    \node at (3,-3) {\tiny $v^{(2)}_3$};  
    \node at (3,-4) {\tiny $v^{(2)}_4$};  
\end{tikzpicture} 
\end{center}
The pair of vector spaces  
\[ U=\big(U^{(1)}=\big\langle v^{(1)}_{1}+  a v^{(1)}_{2} \big\rangle, U^{(2)}=\big\langle v^{(2)}_{1}+  a v^{(2)}_{2},v^{(2)}_{3}+  b v^{(2)}_{4} \big\rangle\big)\]
with $a,b \in \C$ describes a point in the quiver Grassmannian $\mr{Gr}_{\mb{e}}(M)$ since 
\[ M_1 U^{(1)} = \big\langle v^{(2)}_{1}+  a v^{(2)}_{2} \big\rangle \subset U^{(2)},\]
and $\dim_\C U^{(1)} = 1$ and $\dim_\C U^{(2)} = 2$. The grading $\mr{wt}(v_j^{(i)})$ is attractive and for the induced $\mathbb{C}^*$-action we compute 
\begin{align*}
    z.U &= \big(U^{(1)}=\big\langle z^1v^{(1)}_{1}+  a z^2v^{(1)}_{2} \big\rangle, U^{(2)}=\big\langle z^1v^{(2)}_{1}+  az^2 v^{(2)}_{2},z^3v^{(2)}_{3}+  bz^4 v^{(2)}_{4} \big\rangle\big)\\
    &=\big(U^{(1)}=\big\langle v^{(1)}_{1}+  az v^{(1)}_{2} \big\rangle, U^{(2)}=\big\langle v^{(2)}_{1}+  az v^{(2)}_{2},v^{(2)}_{3}+  bz v^{(2)}_{4} \big\rangle\big)
\end{align*}
which is contained in $\mr{Gr}_{\mb{e}}(M)$ by Lemma~\ref{lem:C-action-extends-to-quiver-grass}.

Following Lemma~\ref{lem:fixedPtsAttractiveGrading}, the fixed points of the $\mathbb{C}^*$-action as defined above are:
\begin{align*}
    L_1 &= \big(L_1^{(1)}=\big\langle  v^{(1)}_{1} \big\rangle, L_1^{(2)}=\big\langle v^{(2)}_{1},v^{(2)}_{2}\rangle\big),\\
    L_2 &= \big(L_2^{(1)}=\big\langle v^{(1)}_{1} \big\rangle, L_2^{(2)}=\big\langle v^{(2)}_{1},v^{(2)}_{3}\rangle\big),\\
    L_3 &= \big(L_3^{(1)}=\big\langle  v^{(1)}_{1} \big\rangle, L_3^{(2)}=\big\langle v^{(2)}_{1},v^{(2)}_{4}\rangle\big),\\
    L_4 &= \big(L_4^{(1)}=\big\langle  v^{(1)}_{2} \big\rangle, L_4^{(2)}=\big\langle v^{(2)}_{1},v^{(2)}_{2}\rangle\big),\\
    L_5 &= \big(L_5^{(1)}=\big\langle  v^{(1)}_{2} \big\rangle, L_5^{(2)}=\big\langle v^{(2)}_{2},v^{(2)}_{3}\rangle\big),\\    
    L_6 &= \big(L_6^{(1)}=\big\langle  v^{(1)}_{2} \big\rangle, L_6^{(2)}=\big\langle v^{(2)}_{2},v^{(2)}_{4}\rangle\big).    
\end{align*}
\end{ex}
By Lemma~\ref{lem:fixedPtsAttractiveGrading}, we have a finite number of fixed points and thus we can consider the corresponding attractive loci to get the decomposition \eqref{eqn:BBdecomposition}. Since we have not fixed an order on the fixed point set, we will use the notation
\begin{equation}\label{eqn:fixedPtsParameterisation}
W_L:= \Big\{ V \in \mr{Gr}_{\mb{e}}(M) \vert \lim_{z \to 0} z . V = L \Big\}, \qquad L\in \mr{Gr}_\mb{e}(M)^{\C^*}.
\end{equation}
\begin{trm}
\label{trm:cell_decomp-approx-lin-deg-aff-flag}Let $M$ be a nilpotent representation of $\Delta_n$ and consider the $\C^*$-action on $\mr{Gr}_\mb{e}(M)$ corresponding to an attractive grading on an aligned $Q(M,B)$. Then,
%Let $M = U(\mb{d})$ and $\mb{e} \leq \bdim M$. 
for every $L \in \mr{Gr}_{\mb{e}}(M)^{\C^*}$, the subset $W_L$ is an affine space and hence the quiver Grassmannian admits a cellular decomposition 
\[  
\mr{Gr}_{\mb{e}}(M) = \coprod_{L \in \mr{Gr}_{\mb{e}}(M)^{\C^*}} W_L.  
\]
\end{trm}
\begin{proof}
In the case of the specific attractive grading of the aligned coefficient quiver as described in the proof of Proposition~\ref{prop:AlignabilityDeltan}, the statement is \cite[Theorem~4.13]{Pue2020}. It is immediate to see that the proof can be extended, because it only relies on the attractiveness of the grading if the underlying coefficient quiver is aligned. For the convenience of the reader we summarise the main steps from the proof of \cite[Theorem~4.13]{Pue2020} and highlight the where the generalisation takes place.

The proof has two main steps. First we show that the BB-decomposition is an $\alpha$-partition, i.e. there exists a total order of the fixed points 
$ \mr{Gr}^{\Delta_n}_{\mb{e}}(M)^{\mathbb{C}^*} = \{L_1, \dots, L_r \}$ such that \( \bigsqcup_{j=1}^s \mathcal{C}(L_j) \)
is closed in $\mr{Gr}_{\mb{e}}(M)$ for all $s \in [r]$.
This follows from \cite[Lemma~4.12]{Carrell02} and does not depend on the attractive grading.

It remains to show that the $\mathcal{C}(L)$ are isomorphic to affine spaces. From the definition of quiver Grassmannians we know that
\[\mr{Gr}_{\mb{e}}(M) = \Big\{ (V^{(i)})_{i \in \Z_n} \in \prod_{i \in \Z_n} \mr{Gr}_{e_i}({m_i}) \ : \ M_iV^{(i)} \subseteq V^{(i+1)} \fa i \in \Z_n \Big\},\]
where $\mr{Gr}_{e}({m})$ is the Grassmannian of $e$-dimensional subspaces in $\C^{m}$. We start by showing that the attractive sets 
\(\mathcal{C}(L^{(i)}) := \mathcal{C}(L) \cap \mr{Gr}_{e_i}({m_i}) \) are isomorphic to affine spaces.

By Lemma~\ref{lem:fixedPtsAttractiveGrading} there exists an index set $K_i \in \binom{[m_i]}{e_i}$ such that $L^{(i)}$ is the span of the $v^{(i)}_k$ for $k \in K_i$. To apply this result we only need $\neq$ in (AG1) of Definition~\ref{defi:attractive-grading} (cf. Remark~\ref{rem:weight-functions}). From the properties (AG1) and (AG2) we deduce that a point $V^{(i)}$ in the attracting set $\mathcal{C}(L^{(i)})$ has generators
\begin{equation}\label{eqn:CoordinateDescriptionPointCell}
 w_{k}^{(i)} = v_{k}^{(i)} + \sum_{j \in [m_i] \setminus  [k] \, : \ j \notin K_i} u_{j,k}^{(i)} v_j^{(i)}\qquad\hbox{with } u_{j,k}^{(i)} \in \C.
\end{equation}
for $k \in K_i$ and $\mu_{\ell,k}^{(i)} \in \C$. Hence $\mathcal{C}(L^{(i)})$ is an affine space. It is the key observation of this proof that the above description of the generators is not only valid for the specific attractive grading from the proof of Proposition~\ref{prop:AlignabilityDeltan}, but also for all other attractive gradings. The following computation is exactly the same as in the proof of \cite[Theorem~4.13]{Pue2020}.

Observe that for a representation $V$ in an attracting set of a $\C^*$-fixed point $L$, it holds that
\begin{align*}
V \in \mathcal{C}(L) \Leftrightarrow  V \in  \mr{Gr}^{\Delta_n}_{\mb{e}}(M)  \cap \prod_{i \in\Z_n} \mathcal{C}(L^{(i)}) . 
\end{align*}

Now we describe the equations arising from the condition $M_iV^{(i)} \subseteq V^{(i+1)}$. By the arrangement of the segments in $Q(M)$, it follows that $M_i w_{k}^{(i)}=0$ if $M_i v_{k}^{(i)}=0$. In this case there are no relations. Assume $M_i v_{k}^{(i)} \neq 0$ and let $k' \in K_{i+1}$ be such that $M_i v_{k}^{(i)}=v_{k'}^{(i+1)}$. Analogously, we define the index set $K_i' \subseteq [m_{i+1}]$. If $M_i v_{k}^{(i)} \neq 0$, the coefficients are subject to the conditions 
\begin{align}
\label{eqn1:CoordinateDescriptionConditions}u_{j,k}^{(i)} &= u_{j',k'}^{(i+1)} + \sum_{\substack{\ell \in [j-1] \setminus [k]\, :  \\ M_iv_\ell^{(i)} \neq 0, \\ \ell' \in K_{i+1} \setminus K_i' }} u_{\ell,k}^{(i)} \, u_{j',\ell'}^{(i+1)}  \quad \big( \mr{if} \ j \in [m_i]\setminus [k] \, : \ M_i v_{j}^{(i)}\neq 0, \ j' \notin K_{i+1}\big), \\
\label{eqn2:CoordinateDescriptionConditions} 0 &= u_{h,k'}^{(i+1)} + \sum_{\substack{\ell \in [m_i] \setminus  [k] \, : \\ M_iv_\ell^{(i)} \neq 0, \ \ell' < h, \\ \ell' \in K_{i+1} \setminus K_i', }} u_{\ell,k}^{(i)} \, u_{h,\ell'}^{(i+1)} \quad \binom{\mr{if} \ h \in [m_{i+1}] \setminus  [k'] \, : \ h \notin K_{i+1},}{\nexists \ell \in [m_i]\setminus[k] \ \mr{s.t.:}  \ M_i v_{\ell}^{(i)}= v_h^{(i+1)}}.
\end{align}
Finally, it is shown as in \cite[Theorem~4.13]{Pue2020} that these equations parametrise an affine subspace in the product of Grassmannians $\mr{Gr}_{e_i}(m_i)$. 
\end{proof}
\begin{ex}\label{ex:C-fixed-points-and-cells} For the fixed point $L_2$ from Example~\ref{ex:C-fixed-points} the vector spaces of the points in the cell $\mathcal{C}(L_2)$ have the generators
\begin{align*}
 w_{1}^{(1)} &= v_{1}^{(1)} +  u_{2,1}^{(1)} v_2^{(1)}\\
 w_{1}^{(2)} &= v_{1}^{(2)} +  u_{2,1}^{(2)} v_2^{(2)}+  u_{4,1}^{(2)} v_4^{(2)}\\
 w_{3}^{(2)} &= v_{3}^{(2)} +  u_{4,3}^{(2)} v_4^{(2)}
\end{align*}
by Equation~(\ref{eqn:CoordinateDescriptionPointCell}) in the proof of the above theorem. Following (\ref{eqn1:CoordinateDescriptionConditions}) and (\ref{eqn2:CoordinateDescriptionConditions}) these generators are subject to the conditions 
\[  u_{2,1}^{(1)} =  u_{2,1}^{(2)} \quad \mr{and} \quad  u_{4,1}^{(2)}=0.\] 
Hence we obtain that the cell $\mathcal{C}(L_2)$ is two dimensional.
\end{ex}
\subsection{Action of a bigger Torus}
%-------------------------------------------------------------------------------------
\label{sec:act-bigger-T}
In this section we introduce an action of a bigger torus $T$ on $\mr{Gr}_{\mb{e}}(M)$ and we show that the $\C^*$-action, coming from an attractive grading as in Proposition~\ref{prop:attractiveGradingDeltan}, corresponds to a (generic) cocharacter of $T$.

Let $d_0 := \sum_{i \in \Z_0}d_i$ be the number of indecomposable summands of $M$. We fix once and for all an enumeration $U(i_1;\ell_1), \ldots, U(i_{d_0};\ell_{d_0})$ of the segments of $Q(M,B)$. Each point in the coefficient quiver and hence each basis vector $b \in B$ is uniquely determined by the index $j \in [d_0]$ of the segment it belongs to, and its position $p \in \{0,\dots,\ell_j-1\}$ on the segment itself. Here we declare that the position of a starting point is $p=0$. We will denote by $b_{j,p}$ such a basis vector.

Let $T := (\C^*)^{d_0 +1}$. For any $\gamma := (\gamma_0, (\gamma_j)_{j \in [d_0]}) \in T$ we set
\[ \gamma.b_{j,p} := \gamma^{p}_0 \gamma_j \cdot b_{j,p}.\]
With $T'$ we denote the subtorus obtained from $T$ by setting $\gamma_0=1$. By extending linearly, we get an action on the graded vector space $\bigoplus_{i\in \Z/n\Z} M_i$, which preserves each graded piece. 
\begin{lem}Let $M$ be a representation of $\Delta_n$, and let $T$ act on $\bigoplus M^{(i)}$ as above. Then for any $U\in\mr{Gr}_{\mb{e}}(M)$ and any $\gamma\in T$, $\gamma\cdot U\in \mr{Gr}_\mb{e}(U)$. 
\end{lem}
\begin{proof}Let $M=(\bigoplus_{i\in\Z/n\Z}M^{(i)}, A)$ as in Remark~\ref{rem:nilpoRepDeltanGradedVecSpace}. For $\gamma\in T$, we denote by $\gamma$ also the corresponding automorphism of $\bigoplus_{i\in\Z/n\Z}M^{(i)}$. It is easy to verify that, up to a non-zero scalar, $\gamma$ commutes with $A$. Indeed, it is enough to check this statement on the basis vectors. %Let $b=b_{j,p}\in B$ . If $p=1, then $Ab=0$ and the statement is trivial. 
For $b_{j,p}\in B$ with $p=\ell_j-1$, $Ab_{j,p}=0$ holds and the statement is trivial. If $p\neq \ell_j-1$, then $Ab_{j,p}=b_{j,p+1}$ and hence 
\[
(A\circ \gamma)(b_{j,p})=\gamma_0^{p}\gamma_j (A(b_{j,p}))=\gamma_0^{p}\gamma_j b_{j,p+1}=\gamma_0^{-1}\big(\gamma.(b_{j,p+1})\big)=
\gamma_0^{-1}(\gamma\circ A)(b_{j,p}).
\]
Let $U=(\bigoplus_{i\in\Z/n\Z} U^{(i)}, \overline{A})\in\mr{Gr}_{\mb{e}}(M)$, where $\overline{A}=A_{|\bigoplus U^{(i)}}$. Since $\gamma$ is an automorphism of $\bigoplus M^{(i)}$, which preserves the $\Z/n\Z$-grading, it preserves inclusions and dimensions of graded subspaces, i.e. $\dim_\Kbb U^{(i)}=\dim_\Kbb \gamma.U^{(i)}$. Moreover, by the previous computation we obtain
\[
A(\gamma.U_i)=\gamma_0^{-1}(\gamma.A(U^{(i)}))\subseteq \gamma_0^{-1}(\gamma.U^{(i+1)})=\gamma. U^{(i+1)}\qquad\hbox{for any }i\in\Z/n\Z.
\]
\end{proof}
\begin{rem}
In particular, we obtain that $T'$ commutes with $A$. Hence $T'$ is a subgroup in the automorphism group $\mr{Aut}_{\Delta_n}(M)$ of the $ \Delta_n$-representation $M$,  whereas $T$ has no embedding into $\mr{Aut}_{\Delta_n}(M)$. If the support of $M$ is acyclic (i.e. at least one of the maps $M_a$ is zero), it is possible to show that it is sufficient to work with the $T'$-action, in order to obtain the structure of a GKM-variety. For the special case of the Feigin degeneration of the flag variety of type $A$, the $T'$-action is studied in \cite{CFR2013}. The following computations for $T$ can be specialised to the $T'$-action on $M$ with acyclic support.
\end{rem}
\begin{thm}\label{trm:genericCocharDeltan} Let $M$ be a nilpotent representation of $\Delta_n$ and let $(\mr{wt}(b))_{b \in B}$ be an attractive grading on $Q(M,B)$ with $d(a)=D$ for all $a \in \Z_n$. Then
\[
\chi:\C^*\to T=(\C^*)^{d_0+1}, \qquad z\mapsto \big(z^D, (z^{wt(b_{j,1})})_{j\in[d_0]}\big)
\]
is a generic cocharacter for the above described $T$-action on $\mr{Gr}_\mb{e}(M)$.
\end{thm}
\begin{proof}For any $z\in\C^*$ and any $b_{j,p}\in B$, we have
\[
\chi(z).b_{j,p}=(\chi(z)_0^{p}\chi(z)_j)\cdot b_{j,p}=z^{Dp+\mr{wt}(b_{j,1})}b_{j,p}=z^{\mr{wt}(b_{j,p})}b_{j,p}.
\]
Thus the $\C^*$-action on $\mr{Gr}_\mb{e}(M),$ induced by the cocharacter $\chi$, coincides with the $\C^*$-action in \eqref{eq:C*ActionAttractiveGrading}, coming from the attractive grading, and we can apply Lemma~\ref{lem:fixedPtsAttractiveGrading} to deduce that $|\mr{Gr}_\mb{e}(M)^{\chi(\C^*)}|<\infty$. This concludes the proof by Remark~\ref{rem:genCochar}.
%Since $\chi(\C^*)<T$, we obviously have an inclusion of sets $\mr{Gr}_\mb{e}(M)^{T}\subseteq \mr{Gr}_\mb{e}(M)^{\chi(\C^*)}$. The other inclusion can be deduced either by considering the explicit description from Lemma  \ref{lem:fixedPtsAttractiveGrading} and checking that the $T$-action fixes all of the $\C^*$-fixed points, or by observing that the cardinality of the two sets have to coincide as they both equal the Euler-Poincar\'e characteristic of the complex projective variety $\mr{Gr}_\mb{e}(M)$.
\end{proof}
\begin{cor}\label{cor:BBfiltrDeltan} Let $M$ be a nilpotent representation of $\Delta_n$. Then, the $T$-variety $\mr{Gr}_{\mb{e}}(M)$ is BB-filterable.
\end{cor}
\begin{proof} $|\mr{Gr}_\mb{e}(M)^T|=|\mr{Gr}_\mb{e}(M)^{\chi(\C^*)}|<\infty$, from the proof of the previous theorem, implies Property~(BB1) of Definition~\ref{def:BB-filterable}.
%(BB1) follows from the proof of the previous theorem, as $|\mr{Gr}_\mb{e}(M)^T|=|\mr{Gr}_\mb{e}(M)^{\chi(\C^*)}|<\infty$. 

Let us take the generic cocharacter $\chi\in\mathfrak{X}_*(T)$ as in Theorem~\ref{trm:genericCocharDeltan}. From the proof of the latter result, we know that the $\C^*$-action induced by $\chi$ coincides with the one coming from the attractive grading. We can hence apply Theorem~\ref{trm:cell_decomp-approx-lin-deg-aff-flag}, to deduce that $W_L$ is an affine space for any $L\in\mr{Gr}_\mb{e}(M)$. So it is smooth (and hence rationally smooth). This gives us (BB2).
\end{proof}

Recall that any point of $W_L$, corresponding to the collection $(K_i)_{i\in [n]}\in \prod \binom{[m_i]}{e_i}$, can be described by a collection of tuples \[\Big( \big(u^{(i)}_{j,k}\mid k\in K_i, \ j\not\in K_i,\ j\in [m_i]\setminus[k]\big)\Big)_{i\in [n]}\] as in \eqref{eqn:CoordinateDescriptionPointCell}. In order to describe the $T$-action on $W_L$, in terms of such a description, we introduce some notation: for a basis vector $v^{(i)}_k$ we denote by $s_k^{(i)}$ the segment on which it lies on, and by $p_k^{(i)}$ its position. Then it is immediate to see that
\begin{equation}\label{eqn:TActionCoordinates}
\gamma. \Big( \big(u^{(i)}_{j,k}\big)\Big)=\Big( \big(u^{(i)}_{j,k}\gamma_0^{p_j^{(i)}-p_k^{(i)}}\gamma_{s_j^{(i)}}\gamma_{s_k^{(i)}}^{-1}\big)\Big).
\end{equation}
\begin{ex}
The points in the cell $\mathcal{C}(L_2)$ from Example~\ref{ex:C-fixed-points-and-cells} are described by the following collection \[ \Big( \big(u_{2,1}^{(1)} \big), \big( u_{2,1}^{(2)}, u_{4,1}^{(2)}, u_{4,3}^{(2)} \big)\Big) \]
with relations as computed in that example. The coefficient quiver of the representation $M$ of this running example is shown in Example~\ref{ex:C-fixed-points} and has four segments. Hence $T= (\C^*)^{4+1}$ acts on the quiver Grassmannian $\mr{Gr}_\mb{e}(M)$. It follows from the position of the vertices in the coefficient quiver that
\begin{align*}
    \gamma.u_{2,1}^{(1)} &= u_{2,1}^{(1)}\gamma_0^0\gamma_2\gamma_1^{-1}\\
    \gamma.u_{2,1}^{(2)} &= u_{2,1}^{(2)}\gamma_0^0\gamma_2\gamma_1^{-1}\\
    \gamma.u_{4,1}^{(2)} &= u_{4,1}^{(2)}\gamma_0^1\gamma_4\gamma_1^{-1}\\
    \gamma.u_{4,3}^{(2)} &= u_{4,3}^{(2)}\gamma_0^0\gamma_4\gamma_3^{-1}
\end{align*}
because
\[ s^{(i)}_k=k, p^{(1)}_1=p^{(1)}_2=1, p^{(2)}_1=p^{(2)}_2=2 \ \mr{and} \ p^{(2)}_3=p^{(2)}_4=1 \]
hold in our running example.
\end{ex}

%-------------------------------------------------------------------------------------
\section{GKM-Variety Structure}\label{sec:GKM-VarietyStructure}
\subsection{One-dimensional Torus Orbits}
%-------------------------------------------------------------------------------------
We deal now with the one-dimensional torus orbits on $\mr{Gr}_{\mb{e}}(M)$, where, as in the previous sections, $M$ is a nilpotent representation of $\Delta_n$. We recall that for a basis vector $v_j^{(i)}\in B^{(i)}$ such that $M_iv^{(i)}_j\neq 0$, we have denoted by $j'$ the unique element in $[m_{i+1}]$ such that $M_iv^{(i)}_j=v_{j'}^{(i+1)}$. 

Before proving that $T$ acts with finitely many one-dimensional orbits, we give a definition.
\begin{defi}\label{dfn:terminal}Let $W_L$ be the cell corresponding to the collection $(K_i)_{i\in [n]}\in\prod \binom{[m_i]}{e_i}$. The triple $(i,j,k)$, with $i\in [n]$, $k\in K_i$ and $j\in [m_i]\setminus ([k]\cup K_{i})$, is said to be \f{terminal} for $W_L$ if either $ M_iv^{(i)}_{j}=0$ or $j'\in K_{i+1}$ and for all $v^{(i')}_{\ell}$ with $s^{(i')}_{\ell}=s^{(i)}_{k}$, $p^{(i')}_{\ell}<p^{(i)}_{k}$ and $\ell \in K_{i'}$ there exists a $v^{(i')}_{q}$ with $s^{(i')}_{q}=s^{(i)}_{j}$ and $p^{(i)}_{k}-p^{(i')}_{\ell}=p^{(i)}_{j}-p^{(i')}_{q}$.
\end{defi}
Here we used the same notation as in the proof of Corollary~\ref{cor:BBfiltrDeltan}.
\begin{rem}
Notice that, once the collection $(K_i)_{i\in [n]}$ is fixed, for any $k \in K_i$, there is at most one vertex $v^{(i)}_{j}$ on each segment such that $(i,j,k)$ is terminal. The only case in which a segment has no such vertex, is when the starting point of the segment is contained in one of the $K_j$'s (and hence the whole segment is contained in $\bigcup_{i\in [n]} K_i$). 
\end{rem}
\begin{ex}
The terminal triples for the cell of $L_2$ from Example~\ref{ex:C-fixed-points} are $(2,2,1)$ and $(2,4,3)$. The corresponding terminal vertices are $v^{(2)}_{2}$ and $v^{(2)}_{4}$.
\end{ex}

\begin{prop}\label{prop:finitely-many-1-dim-orb} Let $M$ be a nilpotent representation of $\Delta_n$ with $d_0$ indecomposable direct summands, and let $\mb{e} \leq \bdim M$ be such that $\mr{Gr}_\mb{e}(M)$ is non-empty. Let $T=(\C^*)^{d_0+1}$ act as in \S\ref{sec:act-bigger-T}. Then, 
\begin{enumerate}
\item the number of one-dimensional $T$-orbits on $\mr{Gr}_{\mb{e}}(M)$ is finite;
\item the one-dimensional orbits contained in the cell $W_L$ are parametrised by the terminal triples for $W_L$.
\end{enumerate}

\end{prop}

\begin{proof} Since any cell $W_L$ is $T$-stable, any $T$-orbit is contained in a unique cell and we can therefore apply the coordinate description from \eqref{eqn:CoordinateDescriptionPointCell}, to analyse the $T$-orbit of a point. Assume to have a point in the cell corresponding to the collection $(K_i)_{i\in [n]}$ and consider its coordinate description $((u^{(i)}_{j,k}))$.

By \eqref{eqn:TActionCoordinates}, 
\[
T.\Big( \big(u^{(i)}_{j,k}\big)\Big)=\Big\{\Big( \big(u^{(i)}_{j,k}\gamma_0^{p_j^{(i)}-p_k^{(i)}}\gamma_{s_j^{(i)}}\gamma_{s_k^{(i)}}^{-1}\big)_{j,k}\Big)_i\mid \gamma_0, \gamma_1, \ldots, \gamma_{d_0}\in\C^*\Big\}.
\]
Observe that $((u^{(i)}_{j,k}))$ consists of all 0's if and only if the corresponding point is the fixed point of $W_L$. Since we are interested in one-dimensional orbits, we can assume that there is at least one non-zero entry, say $u^{(i)}_{j,k}$.
%The orbit is one dimensional if and only if the torus acts on it via a character.
 We see immediately that if $u^{(i)}_{j,k}\neq 0$, then the orbit is one-dimensional only if $u^{(r)}_{h,\ell}=0$ unless $s_h^{(r)}=s_j^{(i)}$, $s_{\ell}^{(r)}=s_k^{(i)}$, and  $p_h^{(r)}-p_\ell^{(r)}=p_j^{(i)}-p_k^{(i)}$. 
 %We hence assume that this is the case.

Since we have assumed that $u^{(i)}_{j,k}\neq 0$, there exists a terminal triple $(r,h,\ell)$ such that $s_{\ell}^{(r)}=s_k^{(i)}$,  $s_{h}^{(r)}=s_j^{(i)}$ and  $p_{h}^{(r)}-p_j^{(i)}\geq 0$. We show now by induction on $p_{h}^{(r)}-p_j^{(i)}$ that $u^{(i)}_{j,k}=u_{h,\ell}^{(r)}$. If  $p_{h}^{(r)}-p_j^{(i)}=0$, then $(i,j,k)=(r,h,\ell)$ and the statement is trivial. Otherwise, $M_iv^{(i)}_{j}\neq 0$ and by induction $u^{(i+1)}_{j',k'}=u_{h,\ell}^{(r)}$. Observe that if $q\in K_{i+1}\setminus K_i'$, then $s^{(i+1)}_{q}\neq s^{(i)}_k$ and hence $u_{j',q}^{(i+1)}=0$. Therefore, by  \eqref{eqn1:CoordinateDescriptionConditions}, we conclude that $u_{j,k}^{(i)}=u_{h,\ell}^{(r)}$ as desired. This also implies $p_h^{(r)}-p_\ell^{(r)}=p_j^{(i)}-p_k^{(i)}$.

%Moreover, by \eqref{eqn2:CoordinateDescriptionConditions}, we see that all the other entries of the tuple have to vanish
From what we have just discussed, we conclude that there is at most one orbit of dimension one for any terminal triple. Since there are only finitely many terminal triples, we deduce that $T$ acts on $\mr{Gr}_{\mb{e}}(M)$ with a finite number of one-dimensional orbits.

To prove that there is a one-dimensional orbit for any terminal triple $(r,h,\ell)$, we only have to observe that the tuple $(u^{(i)}_{j,k})$ given by 
\[\mu^{(i)}_{j,k}=\left\{
\begin{array}{ll}
1&\hbox{ if }s^{(i)}_j=s^{(r)}_h, \ s^{(i)}_k=s^{(r)}_\ell\hbox{ and }p_h^{(r)}-p^{(i)}_j=p_\ell^{(r)}-p_k^{(i)}\geq 0,\\
0&\hbox{otherwise}
\end{array}
\right.
\] satisfies both \eqref{eqn1:CoordinateDescriptionConditions} and \eqref{eqn2:CoordinateDescriptionConditions}, and its orbit is one-dimensional by the above considerations.
\end{proof}
\begin{cor}\label{cor:CellDimension}
Let $M$ be a nilpotent representation of $\Delta_n$ with $d_0$ indecomposable direct summands, and let $\mb{e} \leq \bdim M$ be such that $\mr{Gr}_\mb{e}(M)$ is non-empty. Then for any $L\in \mr{Gr}_\mb{e}(M)^T$
\[
\dim W_L=\#\{(i,j,k)\mid (i,j,k) \hbox{ is a terminal triple for }W_L\}.
\]
\end{cor}
\begin{proof}By the previous result, all one-dimensional $T$-orbits in $W_L$, which contain the fixed point $L$ in their closures, are parametrised by the set of terminal triples for $W_L$. By \cite[\S~1.4, Corollary~2]{Brion1999}, the number of closed curves through $L$ (which is finite) coincides with the dimension of $W_L$, since $L$ is the unique isolated fixed point of the rationally smooth $T$-variety $W_L$.
\end{proof}

\begin{trm}
\label{trm:t-stable_filtration-equi-cycle}
Let $M$ be a nilpotent representation of $\Delta_n$ with $d_0$ indecomposable direct summands, and let $\mb{e} \leq \bdim M$ be such that $\mr{Gr}_\mb{e}(M)$ is non-empty. Let $T := (\C^{*})^{d_0+1}$ act on $\mr{Gr}_\mb{e}(M)$ as in \S\ref{sec:act-bigger-T}. Then $(\mr{Gr}_\mb{e}(M),T)$ is a projective BB-filterable GKM-variety.
\end{trm}
\begin{proof}
% Quiver Grassmannians as linear intersections of Grassmannians of vector subspaces are projective varieties.
By Corollary~\ref{cor:BBfiltrDeltan}, $\mr{Gr}_\mb{e}(M)$ is a BB-filterable projective $T$-variety. By Proposition~\ref{prop:finitely-many-1-dim-orb}, the number of one-dimensional $T$-orbits is finite. Hence the $T$-action on $\mr{Gr}_\mb{e}(M)$ is skeletal, since the number of one-dimensional $T$-orbits is finite by Theorem~\ref{trm:genericCocharDeltan}. Then Theorem~\ref{trm:t-stable_filtration-general-setting} implies that $(\mr{Gr}_\mb{e}(M),T)$ is also a GKM-variety.
\end{proof}
%-------------------------------------------------------------------------------------
\subsection{Combinatorial description of the Moment Graph}
%-------------------------------------------------------------------------------------
We have just proven that any non-empty quiver Grassmannian for a nilpotent representation of the quiver $\Delta_n$, admits the structure of a GKM-variety. In order to be able to apply the techniques presented in \S\ref{sec:TorusActionCellsGKMTheory}, first of all we  need to describe the moment graph arising from the torus action.  
\begin{defi}\label{dfn:successorClosed}
Let $Q$ be a quiver, and let $M$ be an object of $\mathrm{rep}_\Kbb(Q)$, and let $\mb{e} \leq \bdim M$ be such that $\mr{Gr}_\mb{e}(M)$ is non-empty. A subquiver $Q'$ of $Q(M,B)$ is said to be \f{successor closed} with dimension vector $\mb{e}$ if $\#\;Q_0'\cap B^{(i)}=e_i$ for any $i\in [n]$, and if for all $a\in Q(M,B)_1$ with $s_a\in Q'_0$, then also $t_a\in Q'_0$.
\end{defi}
 
Denote by $\mr{SC}_\mb{e}^Q(M)$ the set of successor closed subquivers of $Q(M,B)$ with dimension vector $\mb{e}$. Notice that each $S \in \mr{SC}_\mb{e}^Q(M)$ is a collection of (successor closed) subsegments of the segments of $Q(M,B)$. For the rest of this section, we restrict to the case where $Q=\Delta_n$ and a basis $B$ such that $Q(M,B)$ is aligned.

\begin{defi}Let $S \in \mr{SC}_\mb{e}^{\Delta_n}(M)$. A connected subquiver of a segment of $S$ is called a {\bf  movable part} of $S$ if it has the same starting point as the segment.  
\end{defi}
For $S \in \mr{SC}_\mb{e}^{\Delta_n}(M)$, we denote by $\mr{MP}(S)$ the set of movable parts of $S$. Notice that all the segments of $S$ are contained in $\mr{MP}(S)$.

\begin{defi}\label{defn:FundMutation}For $S,H \in \mr{SC}_\mb{e}^{\Delta_n}(M)$ we say that $H$ is obtained from $S$ by a \f{fundamental mutation} if we obtain $H$ from $S$ by moving down exactly one movable part of $S$. 
\end{defi}
The definition of successor closed subquiver is known and for example used in \cite{Cerulli2011}, whereas movable parts and fundamental mutations are new to the current paper.
\begin{rem}Whenever we speak about subquivers of $Q(M,B)$, we mean, by abuse of terminology,  full subquivers, so that they are uniquely determined by their set of vertices. In the above definition, the quiver $H$ is the full subquiver of $Q(M,B)$, whose set of vertices is obtained by removing from $S_0$ the set of vertices belonging to the movable part and adding the set of vertices corresponding to the target vertices of the mutation. 
\end{rem}
\begin{rem}\label{rem:mutation1}
In Definition~\ref{defn:FundMutation}, downwards means that the operation is index increasing in our preferred basis $B$ (such that $Q(M,B)$ is aligned): if $v^{(i)}_k$ is the starting point of the movable part of $S$, then it can only be moved to some $v^{(i)}_j$ with $j>k$ (and $v^{(i)}_j\not \in S_0$).
 The condition $S,H \in \mr{SC}_\mb{e}^{\Delta_n}(M)$ implies that the target $v^{(r)}_h$ of the end point $v^{(r)}_\ell$ of the moved part is 
 \begin{enumerate}
 \item either the end point of the segment $s^{(r)}_h$ of $Q(M,B)$ (i.e., $M_iv_h^{(r)}=0$) 
 \item or the predecessor (in $Q(M,B)$) of the starting point of the segment  $s^{(r)}_h\cap H$ (i.e. $v_{h'}^{(r+1)}\in H_0$).
\end{enumerate}  
Notice that a fundamental mutation of $S\in  \mr{SC}_\mb{e}^{\Delta_n}(M)$ is uniquely determined by the vertices $v^{(r)}_\ell$ and $v^{(r)}_h$ above. For convenience, we will denote such a fundamental mutation by $\mu_{(s_\ell^{(r)}, p_\ell^{(r)}),( s_h^{(r)}, p_h^{(r)})}$. 
\end{rem}
\begin{ex}Consider $M$ and $\mb{e}$ as in Example~\ref{ex:C-fixed-points}. To describe an element $S$ of $\mr{SC}_\mb{e}^{\Delta_2}(M)$ we take $Q(M,B)$ and fill in white the vertices which are not contained in $S_0$. The fixed points $L_1$, $L_2$ and $L_4$ have the following coefficient quivers:
\begin{center}
\begin{tikzpicture}[scale=.4]
	\node at (-0.7,-2.5) {$S_1 = $};	
	\draw[fill=black] (1,-1) circle (.15); 
	\draw[fill=white] (1,-2) circle (.15); 

	\draw[fill=black] (2,-1) circle (.15); 
	\draw[fill=black] (2,-2) circle (.15); 
	\draw[fill=white] (2,-3) circle (.15); 
	\draw[fill=white] (2,-4) circle (.15);

	\draw[arrows={-angle 90}, shorten >=2, shorten <=2]  (1,-1) -- (2,-1); 
	\draw[arrows={-angle 90}, shorten >=2, shorten <=2]  (1,-2) -- (2,-2); 
 
\end{tikzpicture} 
$\qquad$
\begin{tikzpicture}[scale=.4]
	\node at (-0.7,-2.5) {$S_2 = $};	
	\draw[fill=black] (1,-1) circle (.15); 
	\draw[fill=white] (1,-2) circle (.15); 

	\draw[fill=black] (2,-1) circle (.15); 
	\draw[fill=white] (2,-2) circle (.15); 
	\draw[fill=black] (2,-3) circle (.15); 
	\draw[fill=white] (2,-4) circle (.15);

	\draw[arrows={-angle 90}, shorten >=2, shorten <=2]  (1,-1) -- (2,-1); 
	\draw[arrows={-angle 90}, shorten >=2, shorten <=2]  (1,-2) -- (2,-2); 
 
\end{tikzpicture} 
$\qquad$
\begin{tikzpicture}[scale=.4]
	\node at (-0.7,-2.5) {$S_4 = $};
	\draw[fill=white] (1,-1) circle (.15); 
	\draw[fill=black] (1,-2) circle (.15); 

	\draw[fill=black] (2,-1) circle (.15); 
	\draw[fill=black] (2,-2) circle (.15); 
	\draw[fill=white] (2,-3) circle (.15); 
	\draw[fill=white] (2,-4) circle (.15);

	\draw[arrows={-angle 90}, shorten >=2, shorten <=2]  (1,-1) -- (2,-1); 
	\draw[arrows={-angle 90}, shorten >=2, shorten <=2]  (1,-2) -- (2,-2); 
 
\end{tikzpicture} 
\end{center}
The mutation $\mu_{(2,1),(3,0)}: S_1 \to S_2$ is subject to Condition~(1) of Remark~\ref{rem:mutation1} whereas the mutation $\mu_{(1,0),(2,0)}:S_1 \to S_4$ is subject to Condition~(2) of Remark~\ref{rem:mutation1}.
\end{ex}
\begin{rem}
There are no restrictions on the difference $j-k$ in Remark~\ref{rem:mutation1}, so that it can happen that a fundamental mutation is the concatenation of two (or more) other fundamental mutations.
\end{rem}

\begin{ex}\label{ex:fund-mutation}
Let $Q = \Delta_2$ and take the $\Delta_2$-representation 
\[M = U(1;1)\oplus U(2;2)\oplus U(2;2) \oplus U(2;1)\]
and $\mb{e} := (1,2)$. We now apply the above constructions in this setting. %To describe an element $S$ of $\mr{SC}_\mb{e}^{\Delta_2}(M)$ we take $Q(M,B)$ and fill in white the vertices which are not contained in $S_0$.

The quiver Grassmannian $\mr{Gr}_{\mb{e}}(M)$ is isomorphic to the Feigin degeneration of the classical flag variety $\mathcal{F}l_3$, as explained in \cite[Proposition~2.7]{CFR2012}. With the algorithm as described in the proof of Proposition~\ref{prop:AlignabilityDeltan}, we obtain the following coefficient quiver for $M$. The successor closed subquivers in the set $\mr{SC}_\mb{e}^{\Delta_2}(M)$ are:

\begin{center}
\begin{tikzpicture}[scale=0.35]
\node at (-1.7,2) {$S_1 = $};
\draw[fill=white] (0,3) circle (.12);
\draw[arrows={-angle 90}, shorten >=2, shorten <=2]  (0,3) -- (1.5,3);
\draw[fill=white] (1.5,3) circle (.12);
\draw[fill=white] (0,2) circle (.12);
\draw[arrows={-angle 90}, shorten >=2, shorten <=2]  (0,2) -- (1.5,2);
\draw[fill=black] (1.5,2) circle (.12);
\draw[fill=black] (0,1) circle (.12);
\draw[fill=black] (1.5,1) circle (.12);
\end{tikzpicture}
\begin{tikzpicture}[scale=0.35]
\node at (-1.7,2) {$S_2 = $};
\draw[fill=white] (0,3) circle (.12);
\draw[arrows={-angle 90}, shorten >=2, shorten <=2]  (0,3) -- (1.5,3);
\draw[fill=black] (1.5,3) circle (.12);
\draw[fill=white] (0,2) circle (.12);
\draw[arrows={-angle 90}, shorten >=2, shorten <=2]  (0,2) -- (1.5,2);
\draw[fill=white] (1.5,2) circle (.12);
\draw[fill=black] (0,1) circle (.12);
\draw[fill=black] (1.5,1) circle (.12);
\end{tikzpicture}
\begin{tikzpicture}[scale=0.35]
\node at (-1.7,2) {$S_3 = $};
\draw[fill=white] (0,3) circle (.12);
\draw[arrows={-angle 90}, shorten >=2, shorten <=2]  (0,3) -- (1.5,3);
\draw[fill=white] (1.5,3) circle (.12);
\draw[fill=black] (0,2) circle (.12);
\draw[arrows={-angle 90}, shorten >=2, shorten <=2]  (0,2) -- (1.5,2);
\draw[fill=black] (1.5,2) circle (.12);
\draw[fill=white] (0,1) circle (.12);
\draw[fill=black] (1.5,1) circle (.12);
\end{tikzpicture}
\begin{tikzpicture}[scale=0.35]
\node at (-1.7,2) {$S_4 = $};
\draw[fill=white] (0,3) circle (.12);
\draw[arrows={-angle 90}, shorten >=2, shorten <=2]  (0,3) -- (1.5,3);
\draw[fill=black] (1.5,3) circle (.12);
\draw[fill=white] (0,2) circle (.12);
\draw[arrows={-angle 90}, shorten >=2, shorten <=2]  (0,2) -- (1.5,2);
\draw[fill=black] (1.5,2) circle (.12);
\draw[fill=black] (0,1) circle (.12);
\draw[fill=white] (1.5,1) circle (.12);
\end{tikzpicture}
\begin{tikzpicture}[scale=0.35]
\node at (-1.7,2) {$S_5 = $};
\draw[fill=white] (0,3) circle (.12);
\draw[arrows={-angle 90}, shorten >=2, shorten <=2]  (0,3) -- (1.5,3);
\draw[fill=black] (1.5,3) circle (.12);
\draw[fill=black] (0,2) circle (.12);
\draw[arrows={-angle 90}, shorten >=2, shorten <=2]  (0,2) -- (1.5,2);
\draw[fill=black] (1.5,2) circle (.12);
\draw[fill=white] (0,1) circle (.12);
\draw[fill=white] (1.5,1) circle (.12);
\end{tikzpicture}
\begin{tikzpicture}[scale=0.35]
\node at (-1.7,2) {$S_6 = $};
\draw[fill=black] (0,3) circle (.12);
\draw[arrows={-angle 90}, shorten >=2, shorten <=2]  (0,3) -- (1.5,3);
\draw[fill=black] (1.5,3) circle (.12);
\draw[fill=white] (0,2) circle (.12);
\draw[arrows={-angle 90}, shorten >=2, shorten <=2]  (0,2) -- (1.5,2);
\draw[fill=white] (1.5,2) circle (.12);
\draw[fill=white] (0,1) circle (.12);
\draw[fill=black] (1.5,1) circle (.12);
\end{tikzpicture}
\begin{tikzpicture}[scale=0.35]
\node at (-1.7,2) {$S_7 = $};
\draw[fill=black] (0,3) circle (.12);
\draw[arrows={-angle 90}, shorten >=2, shorten <=2]  (0,3) -- (1.5,3);
\draw[fill=black] (1.5,3) circle (.12);
\draw[fill=white] (0,2) circle (.12);
\draw[arrows={-angle 90}, shorten >=2, shorten <=2]  (0,2) -- (1.5,2);
\draw[fill=black] (1.5,2) circle (.12);
\draw[fill=white] (0,1) circle (.12);
\draw[fill=white] (1.5,1) circle (.12);
\end{tikzpicture}
\end{center}

Observe that the basis of $M$, as obtained with the algorithm from Proposition~\ref{prop:AlignabilityDeltan}, is a permutation of the basis described in \cite[Section~6.4]{CFFFR2017} and used in \cite[Remark~3.14]{CFR2013}. Here we order the segments ending over a fixed vertex from long to short. For their computations shorter segments are above the longer ones if they end over the same vertex (see \cite[Example~4]{CFFFR2017}). For the Feigin-degenerate flag variety, their order of segments is described in Section~\ref{subsec:FeiginDegeneration}.

%The order of the fixed points is based on the dimension of the corresponding cells as computed in Corollary~\ref{cor:CellDimension}. 
To determine which fundamental mutation we are applying, we enumerated the four segments of $Q(M,B)$, increasingly from the top to the bottom and from left to right. From $S_4$, we can obtain $S_1$ via the mutation $\mu_{(1,1),(4,0)}$, or by first applying $\mu_{(2,1),(4,0)}$ which gives us $S_2$ and then applying $\mu_{(1,1),(2,1)}$.
\end{ex}	
 The combinatorics of such moves on coefficient quivers can be used to describe the moment graph, associated to the $T$-action on $\mr{Gr}_{\mb{e}}(M)$. 
 
For $(\gamma_0, \ldots, \gamma_{d_0})\in T$, we define %denote by
\[
\epsilon_i:T\rightarrow \C^*, \quad (\gamma_0,\gamma_1, \ldots, \gamma_{d_0})\mapsto \gamma_i \qquad (i\geq 1)
\] 
and 
\[
\delta:T\rightarrow \C^*, \quad (\gamma_0, \gamma_1,\ldots, \gamma_{d_0})\mapsto \gamma_0.
\] 
 \begin{trm} 
\label{trm:comb-moment-graph}
Let $M$ be a nilpotent $\Delta_n$-representation. The vertices of the moment graph $\mathcal{G}_\mb{e}(M)$ are in bijection with the set of successor closed subquivers $SC^{\Delta_n}_\mb{e}(M)$. For $S,H \in SC^{\Delta_n}_\mb{e}(M)$ there exists an arrow from $S$ to $H$ in the moment graph if and only if there exists a fundamental mutation $\mu_{(i,p),(j,q)}(S)=H$. If this is the case, the label of such an edge is given by $\epsilon_j - \epsilon_i + (q-p)\delta$.  
\end{trm}
\begin{proof}
The successor closed subquivers in the set $SC^{\Delta_n}_\mb{e}(M)$ are in bijection with the $\C^*$-fixed points of $\mr{Gr}_{\mb{e}}(M)$ \cite[Proposition 1]{Cerulli2011}. By Theorem~\ref{trm:genericCocharDeltan}, the $\C^*$-fixed points are exactly the $T$-fixed points and hence the vertices of the moment graph.

Let $S\in SC^{\Delta_n}_\mb{e}(M)$. Under the bijection from \cite[Proposition 1]{Cerulli2011}, we have that the corresponding collection $(K_i)_{i\in [n]}$ is given by $K_i:=S_0\cap B^{(i)}$. Let $L$ be the corresponding quiver representation. It follows from Remark~\ref{rem:mutation1} that there is a bijection
\[
\left\{(r,h,\ell)\mid \begin{array}{c}
 (r,h,\ell) \hbox{ is a terminal }\\
 \hbox{triple for }W_L
 \end{array}
 \right\}\leftrightarrow
 \left\{
\mu_{(i,p)(j,q)}\mid \begin{array}{c}
\hbox{there exists }H\in SC^{\Delta_n}_\mb{e}(M)\\
\hbox{ such that }H=\mu_{(i,p)(j,q)}S
\end{array}  
 \right\}
\]
which sends $(r,h,\ell)$ to the fundamental mutation $\mu_{(s_\ell^{(r)},p_\ell^{(r)}),(s_h^{(r)},p_h^{(r)})}$. The edge label is a direct consequence of the description of the $T$-action in the proof of Proposition~\ref{prop:finitely-many-1-dim-orb}.
\end{proof}
\begin{rem}
The orientation of the moment graph depends on the choice of the basis $B$ for the $Q$-representation $M$ and on the attractive grading on $Q(M,B)_0$, which determines the cocharacter $\chi: \C^* \to T$ as in Theorem~\ref{trm:genericCocharDeltan}. But the vertices and unoriented edges only depend on the $T$-action, which is independent of the order of the segments in the coefficient quiver. Hence, the unoriented graph for Example~\ref{ex:fund-mutation} coincides with the one given in Section~\ref{subsec:FeiginDegeneration}. %In other words, the orientation of the moment graph depends on the choice of the attractive grading 
\end{rem}
\begin{rem}\label{rem:partialOrderDeltan} We can equip the set $SC^{\Delta_n}_\mb{e}(M) $ with a partial order, given by the transitive closure of the relation $H\leq S$ if $H=\mu_{(h,p),(\ell,q)}(S)$. It is then possible to refine this to a total order, which is compatible with Theorem~\ref{trm:t-stable_filtration-general-setting}. This is used to determine the order of the fixed points in Example~\ref{ex:fund-mutation} and Example~\ref{ex:cohomology-generators-loop-quiver}.
\end{rem}
%-------------------------------------------------------------------------------------
\section{Special cases}\label{sec:SpecialCases}
%-------------------------------------------------------------------------------------
%-------------------------------------------------------------------------------------
\subsection{Quiver Grassmannians for Equioriented Quivers of Type A}
%-------------------------------------------------------------------------------------
The constructions, as introduced in Section~\ref{sec:quiver-grass-for-cycle} and Section~\ref{sec:TorusAction}, also apply to equioriented quivers of type $A$. The following result is a special case of Theorem~\ref{trm:t-stable_filtration-equi-cycle}.
\begin{cor}
\label{cor:t-stable_filtration-equi-type-A}
Let $Q$ be an equioriented quiver of type $A$ on $n$ vertices. Let $M$ be a representation of $Q$ with $d_0$ indecomposable direct summands. Take a dimension vector $\mb{e} \leq \bdim M$ such that $\mr{Gr}_\mb{e}(M)$ is non-empty. With respect to the action of the torus $T := (\C^{*})^{d_0+1}$, this quiver Grassmannian is a projective BB-filterable GKM-variety.
\end{cor}
%\alex{Should we also include $T'$ in the statement?}
\begin{proof}
All indecomposable representations of $Q$ are indecomposable nilpotent representations of $\Delta_{n}$.
\end{proof}
\subsubsection{Recovering the Bruhat graph}
Let $\mathcal{F}l_{n+1}$ denote the variety of complete flags of subspaces in $\C^{n+1}$ as in Example~\ref{ex:DynkinAEquioriented}. This variety can be obtained as $\mr{Gr}_\mb{e}(M)$ for $M=U_{1,n}\otimes \C^{n+1}$ and $\mb{e}=(1,2,\ldots, n)$. The coefficient quiver consists of $n+1$ segments, all of length $n$, starting in $1$. A subquiver $S$ of $Q(M,B)$ is successor closed if and only if $K_i:=S_0\cap B^{(i)}\subset K_{i+1}:=S_0\cap B^{(i+1)}$ for any $i\in [n-1]$. Therefore $SC^{\Delta_n}_\mb{e}(M)$ is in bijection with the set $S_{n+1}$ of permutations of $[n+1]$: if $S\in SC^{\Delta_n}_\mb{e}(M)$, then the corresponding permutation $\sigma_S$ sends $i\in[n+1]$ to the unique element, contained in $K_i\setminus K_{i-1}$, where, by convention, $K_0=\emptyset$ and $K_{n+1}= [n+1]$. 

Let $v_k^{(i)}$ be the starting point of a segment in $S\in SC^{\Delta_n}_\mb{e}(M)$, and assume that $j>k$ is such that $v_j^{(i)}\not\in K_i$, then there is exactly one movable part of the $i$-th segment which can be moved to the segment $j$. We observe that, 
in permutation terms, this is equivalent to left multiplying $\sigma_S$ by $(i,j)$. Since the end point $v_k^{(r)}$ of the movable part lies on the $k$-th segment in position $r$,  and the  vertex $v_j^{(r)}$ lies on the $j$-th segment, also in position $r$, we deduce that the corresponding edge is labelled by the torus character $\epsilon_j-\epsilon_i$. 
Thus, the $0$-th coordinate of any element $(\gamma_0, \gamma_1, \ldots, \gamma_{n+1})\in T=\C^*\times (\C^*)^{n+1}=\C^*\times T'$ acts trivially, and the action of $T'$ on $\mathcal{F}l_{n+1}$ coincides with the action of the maximal torus of diagonal matrices in $GL_{n+1}$ (induced by the natural action of $GL_{n+1}$ on $\C^{n+1}$).

\begin{rem}
The moment graph we have just described is the so-called Bruhat graph. The partial order, obtained as in Remark~\ref{rem:partialOrderDeltan}, is in this case nothing but the (opposite) Bruhat order.
\end{rem}
\subsubsection{Feigin degeneration of $\mathcal{F}l_{n+1}$}\label{subsec:FeiginDegeneration}
Replacing the identity maps of the quiver representation $M$ from the previous subsection by arbitrary linear maps, we obtain the linear degenerations of the flag variety introduced in \cite{CFFFR2017}. It is therefore possible to apply Theorem~\ref{trm:comb-moment-graph} and Theorem~\ref{trm:cohomology-generators-general-setting} to this class of varieties.

In particular, we can recover the moment graph, for the Feigin degeneration of the flag variety (denoted by $\mathcal{F}l_{n+1}^a$), as constructed in \cite[Section~3.2]{CFR2013}. This is a special degeneration where we replace the identity map along the $i$-th arrow by the projection $\mr{pr}_{i+1}$. 

For $n=4$, the coefficient quiver of this representation is displayed on the left below and after reordering the segments we arrive at the right picture. 
\begin{center}
\begin{tikzpicture}[scale=.34]

	\draw[fill=black] (0,2) circle(.2);	\draw[arrows={-angle 90}, shorten >=2.8]  (0,2) -- (1.5,2);
	\draw[fill=black] (1.5,2) circle(.2);	\draw[arrows={-angle 90}, shorten >=2.8] (1.5,2) -- (3,2);
	\draw[fill=black] (3,2) circle(.2);		\draw[arrows={-angle 90}, shorten >=2.8] (3,2) -- (4.5,2);
	\draw[fill=black] (4.5,2) circle(.2);

	\draw[fill=black] (0,1) circle(.2);	%\draw[arrows={-angle 90}, shorten >=2.8]  (0,1) -- (1.5,1);
	\draw[fill=black] (1.5,1) circle(.2);	\draw[arrows={-angle 90}, shorten >=2.8] (1.5,1) -- (3,1);
	\draw[fill=black] (3,1) circle(.2);		\draw[arrows={-angle 90}, shorten >=2.8] (3,1) -- (4.5,1);
	\draw[fill=black] (4.5,1) circle(.2);

	\draw[fill=black] (0,0) circle(.2);	\draw[arrows={-angle 90}, shorten >=2.8]  (0,0) -- (1.5,0);
	\draw[fill=black] (1.5,0) circle(.2);	%\draw[arrows={-angle 90}, shorten >=2.8] (1.5,0) -- (3,0);
	\draw[fill=black] (3,0) circle(.2);		\draw[arrows={-angle 90}, shorten >=2.8] (3,0) -- (4.5,0);
	\draw[fill=black] (4.5,0) circle(.2);
	
	\draw[fill=black] (0,-1) circle(.2);	\draw[arrows={-angle 90}, shorten >=2.8]  (0,-1) -- (1.5,-1);
	\draw[fill=black] (1.5,-1) circle(.2);	\draw[arrows={-angle 90}, shorten >=2.8] (1.5,-1) -- (3,-1);
	\draw[fill=black] (3,-1) circle(.2);		%\draw[arrows={-angle 90}, shorten >=2.8] (3,-1) -- (4.5,-1);
	\draw[fill=black] (4.5,-1) circle(.2);
	
	\draw[fill=black] (0,-2) circle(.2);	\draw[arrows={-angle 90}, shorten >=2.8]  (0,-2) -- (1.5,-2);
	\draw[fill=black] (1.5,-2) circle(.2);	\draw[arrows={-angle 90}, shorten >=2.8] (1.5,-2) -- (3,-2);
	\draw[fill=black] (3,-2) circle(.2);		\draw[arrows={-angle 90}, shorten >=2.8] (3,-2) -- (4.5,-2);
	\draw[fill=black] (4.5,-2) circle(.2);
	
\end{tikzpicture}$\qquad \qquad \qquad \qquad \qquad$
\begin{tikzpicture}[scale=.34]

    \draw[arrows={-angle 90}, shorten >=2.8] (3,2) -- (4.5,2);
    \draw[arrows={-angle 90}, shorten >=2.8] (1.5,1) -- (3,1);
	\draw[arrows={-angle 90}, shorten >=2.8] (3,1) -- (4.5,1);
	\draw[arrows={-angle 90}, shorten >=2.8]  (0,0) -- (1.5,0);
    \draw[arrows={-angle 90}, shorten >=2.8] (1.5,0) -- (3,0);
	\draw[arrows={-angle 90}, shorten >=2.8] (3,0) -- (4.5,0);
	\draw[arrows={-angle 90}, shorten >=2.8]  (0,-1) -- (1.5,-1);
    \draw[arrows={-angle 90}, shorten >=2.8] (1.5,-1) -- (3,-1);
	\draw[arrows={-angle 90}, shorten >=2.8] (3,-1) -- (4.5,-1);
	\draw[arrows={-angle 90}, shorten >=2.8]  (0,-2) -- (1.5,-2);
    \draw[arrows={-angle 90}, shorten >=2.8] (1.5,-2) -- (3,-2);   
	\draw[arrows={-angle 90}, shorten >=2.8]  (0,-3) -- (1.5,-3);    
	
	\draw[fill=black] (4.5,3) circle(.2);
	
	\draw[fill=black] (3,2) circle(.2);
	\draw[fill=black] (4.5,2) circle(.2);

	\draw[fill=black] (1.5,1) circle(.2);
	\draw[fill=black] (3,1) circle(.2);
	\draw[fill=black] (4.5,1) circle(.2);

	\draw[fill=black] (0,0) circle(.2);
	\draw[fill=black] (1.5,0) circle(.2);
	\draw[fill=black] (3,0) circle(.2);
	\draw[fill=black] (4.5,0) circle(.2);
	
	\draw[fill=black] (0,-1) circle(.2);
	\draw[fill=black] (1.5,-1) circle(.2);
	\draw[fill=black] (3,-1) circle(.2);
	\draw[fill=black] (4.5,-1) circle(.2);
	
	\draw[fill=black] (0,-2) circle(.2);
	\draw[fill=black] (1.5,-2) circle(.2);
	\draw[fill=black] (3,-2) circle(.2);
	
	\draw[fill=black] (0,-3) circle(.2);
	\draw[fill=black] (1.5,-3) circle(.2);
	
	\draw[fill=black] (0,-4) circle(.2);
	
\end{tikzpicture}
\end{center}
The basis corresponding to the right picture satisfies the assumptions of an aligned coefficient quiver and is the same as described in \cite[Remark~3.14]{CFR2013}. We obtain an attractive grading if we use the row indices as weights for the corresponding basis vectors. Hence we can apply the results from Section~\ref{sec:GKM-VarietyStructure} in this setting.
%TODO If we choose the basis $B$ as described in Example~\ref{ex:fund-mutation}, we obtain the same set of vertices and edges. 

%Now we describe how to recover their description of the moment graph. 
By \cite[Equation~2.1, Remark~3.14]{CFR2013}, the admissible collections $\bf{S}$, which parametrise the vertices of the moment graph, are a special case of the tuple of index sets $(K_i)_{i\in [n]}$ as in Lemma~\ref{lem:fixedPtsAttractiveGrading}. The one-dimensional orbits, starting at the vertex parametrised by $\bf{S}$, are parametrised by $\bf{S}$-effective pairs \cite[Definition~3.5]{CFR2013}. Their geometric interpretation in \cite[Remark~3.6]{CFR2013}  coincides with the description of fundamental mutations of the coefficient quiver, corresponding to the admissible collection $\bf{S}$. Hence, the indices of the $\bf{S}$-effective pairs and the leading indices of the fundamental mutations coincide, if we use the same order of the indecomposable summands of the quiver representation.

The structure of the unoriented moment graph is independent of the chosen basis, as long as the basis yields an aligned coefficient quiver. But to recover the edge labels as in \cite[Theorem~3.18]{CFR2013}, we have to use the basis described above instead of the one from Example~\ref{ex:fund-mutation}. Nevertheless, the edge labels are slightly different, since our torus has one additional parameter, which is required to obtain finitely many one-dimensional orbits for general quiver Grassmannians for the cyclic quiver. 

For the equioriented quiver of type $A$, the parameter $\gamma_0$ is not necessary to obtain the  structure of a GKM-variety. Instead we can work with the subtorus $T'$ as defined in \S\ref{sec:act-bigger-T}, where we set $\gamma_0$ equal to one. % and linear degenerations of the type A affine flag variety \cite{Pue2020}.
In the following example, we exhibit the different edge labels for $n=3$.

\begin{ex}
The $T$- (and $T'$-)fixed points in $\mathcal{F}l_{3}^a$ are given by
\begin{center}
\begin{tikzpicture}[scale=0.35]
\node at (-1.7,2.5) {$p_1 = $};
\draw[fill=white] (1.5,4) circle (.12);
\draw[fill=white] (0,3) circle (.12);
\draw[arrows={-angle 90}, shorten >=2, shorten <=2]  (0,3) -- (1.5,3);
\draw[fill=black] (1.5,3) circle (.12);
\draw[fill=white] (0,2) circle (.12);
\draw[arrows={-angle 90}, shorten >=2, shorten <=2]  (0,2) -- (1.5,2);
\draw[fill=black] (1.5,2) circle (.12);
\draw[fill=black] (0,1) circle (.12);
\end{tikzpicture}
\begin{tikzpicture}[scale=0.35]
\node at (-1.7,2.5) {$p_2 = $};
\draw[fill=black] (1.5,4) circle (.12);
\draw[fill=white] (0,3) circle (.12);
\draw[arrows={-angle 90}, shorten >=2, shorten <=2]  (0,3) -- (1.5,3);
\draw[fill=white] (1.5,3) circle (.12);
\draw[fill=white] (0,2) circle (.12);
\draw[arrows={-angle 90}, shorten >=2, shorten <=2]  (0,2) -- (1.5,2);
\draw[fill=black] (1.5,2) circle (.12);
\draw[fill=black] (0,1) circle (.12);
\end{tikzpicture}
\begin{tikzpicture}[scale=0.35]
\node at (-1.7,2.5) {$p_3 = $};
\draw[fill=white] (1.5,4) circle (.12);
\draw[fill=white] (0,3) circle (.12);
\draw[arrows={-angle 90}, shorten >=2, shorten <=2]  (0,3) -- (1.5,3);
\draw[fill=black] (1.5,3) circle (.12);
\draw[fill=black] (0,2) circle (.12);
\draw[arrows={-angle 90}, shorten >=2, shorten <=2]  (0,2) -- (1.5,2);
\draw[fill=black] (1.5,2) circle (.12);
\draw[fill=white] (0,1) circle (.12);
\end{tikzpicture}
\begin{tikzpicture}[scale=0.35]
\node at (-1.7,2.5) {$p_4 = $};
\draw[fill=black] (1.5,4) circle (.12);
\draw[fill=white] (0,3) circle (.12);
\draw[arrows={-angle 90}, shorten >=2, shorten <=2]  (0,3) -- (1.5,3);
\draw[fill=black] (1.5,3) circle (.12);
\draw[fill=white] (0,2) circle (.12);
\draw[arrows={-angle 90}, shorten >=2, shorten <=2]  (0,2) -- (1.5,2);
\draw[fill=white] (1.5,2) circle (.12);
\draw[fill=black] (0,1) circle (.12);
\end{tikzpicture}
\begin{tikzpicture}[scale=0.35]
\node at (-1.7,2.5) {$p_5 = $};
\draw[fill=white] (1.5,4) circle (.12);
\draw[fill=black] (0,3) circle (.12);
\draw[arrows={-angle 90}, shorten >=2, shorten <=2]  (0,3) -- (1.5,3);
\draw[fill=black] (1.5,3) circle (.12);
\draw[fill=white] (0,2) circle (.12);
\draw[arrows={-angle 90}, shorten >=2, shorten <=2]  (0,2) -- (1.5,2);
\draw[fill=black] (1.5,2) circle (.12);
\draw[fill=white] (0,1) circle (.12);
\end{tikzpicture}
\begin{tikzpicture}[scale=0.35]
\node at (-1.7,2.5) {$p_6 = $};
\draw[fill=black] (1.5,4) circle (.12);
\draw[fill=white] (0,3) circle (.12);
\draw[arrows={-angle 90}, shorten >=2, shorten <=2]  (0,3) -- (1.5,3);
\draw[fill=white] (1.5,3) circle (.12);
\draw[fill=black] (0,2) circle (.12);
\draw[arrows={-angle 90}, shorten >=2, shorten <=2]  (0,2) -- (1.5,2);
\draw[fill=black] (1.5,2) circle (.12);
\draw[fill=white] (0,1) circle (.12);
\end{tikzpicture}
\begin{tikzpicture}[scale=0.35]
\node at (-1.7,2.5) {$p_7 = $};
\draw[fill=black] (1.5,4) circle (.12);
\draw[fill=black] (0,3) circle (.12);
\draw[arrows={-angle 90}, shorten >=2, shorten <=2]  (0,3) -- (1.5,3);
\draw[fill=black] (1.5,3) circle (.12);
\draw[fill=white] (0,2) circle (.12);
\draw[arrows={-angle 90}, shorten >=2, shorten <=2]  (0,2) -- (1.5,2);
\draw[fill=white] (1.5,2) circle (.12);
\draw[fill=white] (0,1) circle (.12);
\end{tikzpicture}
\end{center}
On the left below we have the moment graph for the $T$-action on $\mathcal{F}l_{3}^a$ and on the right hand side we have the moment graph for the $T'$-action, which coincides with the results in \cite[Example~3.17, Theorem~3.18]{CFR2013}.
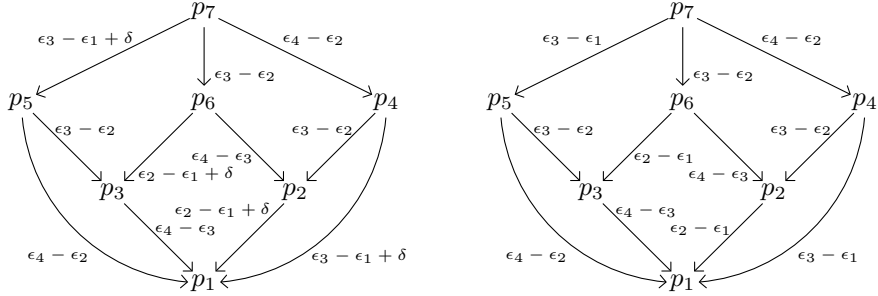
\begin{center}
\begin{tikzpicture}[scale=1.3]
 \def\centerarc[#1](#2)(#3:#4:#5);%
%Syntax: [draw options] (center) (initial angle:final angle:radius)
    {
    \draw[#1]([shift=(#3:#5)]#2) arc (#3:#4:#5);
    }

\node at (0,3) {$p_7$};
\draw[arrows={-angle 90}, shorten >=6, shorten <=6]  (0,3) -- (-2,2); \node at (-1.3,2.7) {\tiny $\epsilon_3-\epsilon_1+\delta$};
\draw[arrows={-angle 90}, shorten >=6, shorten <=6]  (0,3) -- (0,2);  \node at (1.2,2.7) {\tiny $\epsilon_4-\epsilon_2$};
\draw[arrows={-angle 90}, shorten >=6, shorten <=6]  (0,3) -- (2,2);  \node at (0.45,2.3) {\tiny $\epsilon_3-\epsilon_2$};
\node at (0,2) {$p_6$};
\draw[arrows={-angle 90}, shorten >=6, shorten <=6]  (0,2) -- (1,1); \node at (0.2,1.4) {\tiny $\epsilon_4-\epsilon_3$};
\draw[arrows={-angle 90}, shorten >=6, shorten <=6]  (0,2) -- (-1,1); \node at (-0.2,1.2) {\tiny $\epsilon_2-\epsilon_1+\delta$};
\node at (-2,2) {$p_5$};
\draw[arrows={-angle 90}, shorten >=6, shorten <=6]  (-2,2) -- (-1,1); \node at (-1.3,1.7) {\tiny $\epsilon_3-\epsilon_2$};
% \draw[arrows={-angle 90}, shorten >=2, shorten <=2]  (-2,2) -- (0,0); 
\centerarc[arrows={-angle 90}](0,2)(185:265:2cm); \node at (-1.6,0.3) {\tiny $\epsilon_4-\epsilon_2$};
\node at (2,2) {$p_4$};
\draw[arrows={-angle 90}, shorten >=6, shorten <=6]  (2,2) -- (1,1); \node at (1.3,1.7) {\tiny $\epsilon_3-\epsilon_2$};
% \draw[arrows={-angle 90}, shorten >=2, shorten <=2]  (2,2) -- (0,0); 
\centerarc[arrows={-angle 90}](0,2)(-5:-85:2cm); \node at (1.7,0.3) {\tiny $\epsilon_3-\epsilon_1+\delta$};
\node at (-1,1) {$p_3$};
\draw[arrows={-angle 90}, shorten >=6, shorten <=6]  (-1,1) -- (0,0); \node at (-0.2,0.6) {\tiny $\epsilon_4-\epsilon_3$};
\node at (1,1) {$p_2$};
\draw[arrows={-angle 90}, shorten >=6, shorten <=6]  (1,1) -- (0,0); \node at (0.2,0.8) {\tiny $\epsilon_2-\epsilon_1+\delta$};
\node at (0,0) {$p_1$};
\end{tikzpicture}$\quad \quad$
\begin{tikzpicture}[scale=1.3]
 \def\centerarc[#1](#2)(#3:#4:#5);%
%Syntax: [draw options] (center) (initial angle:final angle:radius)
    {
    \draw[#1]([shift=(#3:#5)]#2) arc (#3:#4:#5);
    }

\node at (0,3) {$p_7$};
\draw[arrows={-angle 90}, shorten >=6, shorten <=6]  (0,3) -- (-2,2); \node at (-1.2,2.7) {\tiny $\epsilon_3-\epsilon_1$};
\draw[arrows={-angle 90}, shorten >=6, shorten <=6]  (0,3) -- (0,2);  \node at (1.2,2.7) {\tiny $\epsilon_4-\epsilon_2$};
\draw[arrows={-angle 90}, shorten >=6, shorten <=6]  (0,3) -- (2,2);  \node at (0.45,2.3) {\tiny $\epsilon_3-\epsilon_2$};
\node at (0,2) {$p_6$};
\draw[arrows={-angle 90}, shorten >=6, shorten <=6]  (0,2) -- (1,1); \node at (0.4,1.2) {\tiny $\epsilon_4-\epsilon_3$};
\draw[arrows={-angle 90}, shorten >=6, shorten <=6]  (0,2) -- (-1,1); \node at (-0.2,1.4) {\tiny $\epsilon_2-\epsilon_1$};
\node at (-2,2) {$p_5$};
\draw[arrows={-angle 90}, shorten >=6, shorten <=6]  (-2,2) -- (-1,1); \node at (-1.3,1.7) {\tiny $\epsilon_3-\epsilon_2$};
% \draw[arrows={-angle 90}, shorten >=2, shorten <=2]  (-2,2) -- (0,0); 
\centerarc[arrows={-angle 90}](0,2)(185:265:2cm); \node at (-1.6,0.3) {\tiny $\epsilon_4-\epsilon_2$};
\node at (2,2) {$p_4$};
\draw[arrows={-angle 90}, shorten >=6, shorten <=6]  (2,2) -- (1,1); \node at (1.3,1.7) {\tiny $\epsilon_3-\epsilon_2$};
% \draw[arrows={-angle 90}, shorten >=2, shorten <=2]  (2,2) -- (0,0); 
\centerarc[arrows={-angle 90}](0,2)(-5:-85:2cm); \node at (1.6,0.3) {\tiny $\epsilon_3-\epsilon_1$};
\node at (-1,1) {$p_3$};
\draw[arrows={-angle 90}, shorten >=6, shorten <=6]  (-1,1) -- (0,0); \node at (-0.4,0.8) {\tiny $\epsilon_4-\epsilon_3$};
\node at (1,1) {$p_2$};
\draw[arrows={-angle 90}, shorten >=6, shorten <=6]  (1,1) -- (0,0); \node at (0.2,0.6) {\tiny $\epsilon_2-\epsilon_1$};
\node at (0,0) {$p_1$};
\end{tikzpicture}
\end{center}
\end{ex}
%-------------------------------------------------------------------------------------
\subsection{Linear Degenerations of the Affine Grassmannian and the Affine Flag Variety of Type A}
%-------------------------------------------------------------------------------------
We briefly recall the definition of the affine Grassmannian, its linear degenerations and their finite dimensional approximations (see \cite{FFR2017},\cite[Section~3 \& 6]{Pue2020} for more detail). Let $P \subset \hat{\mathfrak{gl}}_n$ be the maximal parahoric subgroup. The \f{affine Grassmannian} of type $\mathfrak{gl}_n$ is defined as $\mr{Gr}(\hat{\mathfrak{gl}}_n) := \hat{\mathfrak{gl}}_n/ P$. For $\ell \in \Z$, define 
\[ V_\ell := \mr{span}(v_\ell, v_{\ell-1}, v_{\ell-2},\dots) \]
as subspace of the infinite dimensional vectorspace $V$ with basis vectors $v_i$ for all $i \in \Z$. The \textbf{Sato Grassmannian} $\mathrm{SGr}_k$ for $k \in \Z$ is defined as
\[ \mathrm{SGr}_k := \big\{ U \subset V \  :  \ \mr{There} \ \mr{exists} \ \mr{a} \ \ell < k \ \mr{s.t.} \  V_\ell \subset U  
\  \mr{and}  \  \dim U/V_\ell = k-\ell \ \big\}. \]
%The affine flag variety as a set of cyclic chains where the vector spaces are elements of the Sato Grassmannians $\mr{SGr}_k$.
\begin{prop}(c.f. \cite[Section~1.3]{FFR2017})%,KaPe1986
\label{prop:alt-param-classical}
The affine Grassmannian $\mr{Gr}\big(\widehat{\mathfrak{gl}}_n\big)$ as subset in the Sato Grassmannian is parametrised as
\[\mr{Gr}\big(\widehat{\mathfrak{gl}}_n\big) \cong \big\{ U  \in  \mathrm{SGr}_0  \ : \  U  \subseteq s_n U \big\}, \]
where $s_n : V \to V$ maps $v_i$ to $v_{i+n}$ for all $i \in \Z$.
\end{prop}

It is possible to define linear degenerations in the same way as for the affine flag variety (see \cite[Definition~6.2]{Pue2020}): For a linear map $f: V \to V$,  the \f{f-linear degenerate affine Grassmannian} is defined as 
\[\mr{Gr}^f\big(\widehat{\mathfrak{gl}}_n\big) := \big\{ U  \in  \mathrm{SGr}_0  \ : \ f U  \subseteq  U \big\}.\] 
The degeneration 
\[\mr{Gr}^a(\hat{\mathfrak{gl}}_n) := \big\{ U  \in  \mathrm{SGr}_0  \ : \ s_{-n}\circ \mr{pr}_1 \circ \mr{pr}_2 \circ \dots \circ \mr{pr}_n U  \subseteq  U \big\}\] 
was already studied in \cite{FFR2017}.
For a positive integer $N \in \N$, the finite approximation of the $f$-linear degenerate affine Grassmannian is defined as 
\[ \mr{Gr}^f_N\big(\widehat{\mathfrak{gl}}_n\big) := \big\{ U  \in  \mr{Gr}^f\big(\widehat{\mathfrak{gl}}_n\big)  \ : \ V_{-N}  \subseteq  U \subseteq V_N \big\}.\]
Utilising the finite dimensional vectorspace 
\[ V_{(N)} := \mr{span}(v_N,v_{N-1},\dots , v_{-N+2},v_{-N+1}) \]
we can identify each finite dimensional approximation with a quiver Grassmannian (cf. \cite[Theorem~6.3]{Pue2020}): 
\[ \mr{Gr}^f_N\big(\widehat{\mathfrak{gl}}_n\big) \cong \mr{Gr}_N(M_N^f)\]
where $M_N^f:=(V_{(N)},f\vert_{V_{(N)}})$. A linear degeneration is called nilpotent if $f\vert_{V_{(N)}}$ is nilpotent for every $N \in \N$. In this case $M_N^f$ is a nilpotent representation of the loop quiver (cf. \cite[Remark~6.5]{Pue2020}).
The isomorphism classes of nilpotent degenerations are parametrised by the corank of $f$ \cite[Section~6.4]{Pue2020}. We call $\mr{Gr}^f_N\big(\hat{\mathfrak{gl}}_n\big)$ a partial degeneration if the corank of $f$ is between the corank of $s_{-n}$ and the corank of $s_{-n} \circ \mr{pr}_1 \circ \dots \circ \mr{pr}_n$. The corresponding isomorphism classes of nilpotent linear degenerations between $\mr{Gr}(\hat{\mathfrak{gl}}_n)$ and  $\mr{Gr}^a(\hat{\mathfrak{gl}}_n)$, are labeled by the integers $k \in \{ 0,1,\dots,n\}$, where zero corresponds to the affine Grassmannian and $n$ to $\mr{Gr}^a(\hat{\mathfrak{gl}}_n)$. Each isomorphism class has the representative
\[\mr{Gr}^k(\hat{\mathfrak{gl}}_n) := \big\{ U  \in  \mathrm{SGr}_0  \ : \ s_{-n}\circ \mr{pr}_1 \circ \mr{pr}_2 \circ \dots \circ \mr{pr}_k U  \subseteq  U \big\}.\] 
\begin{prop}\label{prop:approx-lin-deg-aff-Grass}
For $k \in \{ 0,1,\dots,n\}$, the finite approximations of nilpotent partial degenerations can be realised as quiver Grassmannians in the following way
\[ \mr{Gr}_N^k\big(\widehat{\mathfrak{gl}}_{n}\big) \cong \mr{Gr}_{nN}\big( A_{2N} \otimes \C^{n-k} \oplus A_{N} \otimes \C^{2k}\big), \]
where $A_N \cong \C[t]/(t^N)$. 
\end{prop}
This is a special case of the construction in \cite[Theorem~3.7, Lemma~4.14]{Pue2020}.
\begin{lma}
\label{lma:approx-are-GKM}
The quiver Grassmannians providing finite approximations for nilpotent linear degenerations of affine Grassmannians and affine flag varieties are BB-filterable GKM-varieties.
\end{lma}
\begin{proof}
The quiver representations as in Proposition~\ref{prop:approx-lin-deg-aff-Grass} and \cite[Lemma~4.14]{Pue2020}, which are used to define the approximations, satisfy the assumptions of Theorem~\ref{trm:t-stable_filtration-equi-cycle}.
\end{proof} 
These quiver Grassmannians are in general not normal, which implies that it is not possible to apply \cite[Theorem~6.9]{Gonzales2014} by Gonzales, in order to compute the $S$-module basis of the $T$-equivariant cohomology. Nevertheless by Theorem~\ref{trm:cohomology-generators-general-setting}, Gonzales' recipe also works in our setting.
\begin{ex}
\label{ex:cohomology-generators-loop-quiver}
For $n=2$ there are three isomorphism classes of linear degenerations of the affine Grassmannian. We want to consider the representative $\mr{Gr}^k\big(\widehat{\mathfrak{gl}}_{n}\big)$ for $k=1$. This corresponds to the intermediate degeneration between the (non-degenerate) affine Grassmannian and its degeneration $\mr{Gr}^a(\hat{\mathfrak{gl}}_2)$. 

For $N=1$ its approximation is isomorphic to the quiver Grassmannian \(\mr{Gr}_2( M) \) for the representation $M = A_{2}  \oplus A_{1} \oplus A_{1} $.
By Lemma~\ref{lma:approx-are-GKM}, we know that this quiver Grassmannian admits a $T$-action such that it becomes a BB-filterable GKM-variety. We can also apply Theorem~\ref{trm:comb-moment-graph} to compute its moment graph. There are four torus fixed points, corresponding to the following successor closed subquivers:
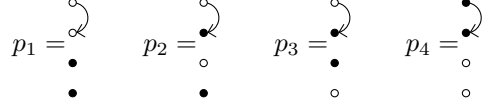
\begin{center}
\begin{tikzpicture}[scale=0.4]
\node at (-1.1,1.5) {$p_1 = $};
\draw[fill=white] (0,3) circle (.12);
\draw[fill=white] (0,2) circle (.12);
\draw[fill=black] (0,1) circle (.12);
\draw[fill=black] (0,0) circle (.12);

 \def\centerarc[#1](#2)(#3:#4:#5);%
%Syntax: [draw options] (center) (initial angle:final angle:radius)
    {
    \draw[#1]([shift=(#3:#5)]#2) arc (#3:#4:#5);
    }
\centerarc[arrows={-angle 90}](0,2.5)(76:-76:0.5cm); %three --> two
%\centerarc[arrows={-angle 90}](-1.5,1.5)(42:-42:2.15cm); %three --> zero
\end{tikzpicture}$\quad \ $
\begin{tikzpicture}[scale=0.4]
\node at (-1.1,1.5) {$p_2 = $};
\draw[fill=white] (0,3) circle (.12);
\draw[fill=black] (0,2) circle (.12);
\draw[fill=white] (0,1) circle (.12);
\draw[fill=black] (0,0) circle (.12);

 \def\centerarc[#1](#2)(#3:#4:#5);%
%Syntax: [draw options] (center) (initial angle:final angle:radius)
    {
    \draw[#1]([shift=(#3:#5)]#2) arc (#3:#4:#5);
    }
\centerarc[arrows={-angle 90}](0,2.5)(76:-76:0.5cm); %three --> two
%\centerarc[arrows={-angle 90}](-1.5,1.5)(42:-42:2.15cm); %three --> zero
\end{tikzpicture}$\quad \ $
\begin{tikzpicture}[scale=0.4]
\node at (-1.1,1.5) {$p_3 = $};
\draw[fill=white] (0,3) circle (.12);
\draw[fill=black] (0,2) circle (.12);
\draw[fill=black] (0,1) circle (.12);
\draw[fill=white] (0,0) circle (.12);

 \def\centerarc[#1](#2)(#3:#4:#5);%
%Syntax: [draw options] (center) (initial angle:final angle:radius)
    {
    \draw[#1]([shift=(#3:#5)]#2) arc (#3:#4:#5);
    }
\centerarc[arrows={-angle 90}](0,2.5)(76:-76:0.5cm); %three --> two
%\centerarc[arrows={-angle 90}](-1.5,1.5)(42:-42:2.15cm); %three --> zero
\end{tikzpicture}$\quad \ $
\begin{tikzpicture}[scale=0.4]
\node at (-1.1,1.5) {$p_4 = $};
\draw[fill=black] (0,3) circle (.12);
\draw[fill=black] (0,2) circle (.12);
\draw[fill=white] (0,1) circle (.12);
\draw[fill=white] (0,0) circle (.12);

 \def\centerarc[#1](#2)(#3:#4:#5);%
%Syntax: [draw options] (center) (initial angle:final angle:radius)
    {
    \draw[#1]([shift=(#3:#5)]#2) arc (#3:#4:#5);
    }
\centerarc[arrows={-angle 90}](0,2.5)(76:-76:0.5cm); %three --> two
%\centerarc[arrows={-angle 90}](-1.5,1.5)(42:-42:2.15cm); %three --> zero
\end{tikzpicture}
\end{center}
If we number the segments in the coefficient quiver from top to bottom, we obtain the following moment graph:
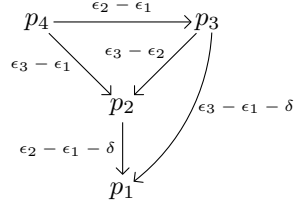
\begin{center}
\begin{tikzpicture}[scale=0.6]

\node at (-2,0) {$p_4$};
\node at (2,0) {$p_3$};
\node at (0,-2) {$p_2$};
\node at (0,-4) {$p_1$};

\node at (0,0.4) {\tiny $\epsilon_2-\epsilon_1$}; % 4-->3 
\node at (-1.9,-1) {\tiny $\epsilon_3-\epsilon_1$}; % 4-->2 
\node at (0.3,-0.7) {\tiny $\epsilon_3-\epsilon_2$}; % 3-->2
\node at (2.9,-2) {\tiny $\epsilon_3-\epsilon_1-\delta$}; % 3-->1
\node at (-1.3,-3) {\tiny $\epsilon_2-\epsilon_1-\delta$}; % 2-->1 

\draw[arrows={-angle 90}, shorten >=6, shorten <=6]  (-2,0) -- (2,0); % 4-->3 
\draw[arrows={-angle 90}, shorten >=6, shorten <=6]  (2,0) -- (0,-2); % 3-->2 
\draw[arrows={-angle 90}, shorten >=6, shorten <=6]  (-2,0) -- (0,-2); % 4-->2 
\draw[arrows={-angle 90}, shorten >=6, shorten <=6]  (0,-2) -- (0,-4); % 2-->1 

 \def\centerarc[#1](#2)(#3:#4:#5);%
%Syntax: [draw options] (center) (initial angle:final angle:radius)
    {
    \draw[#1]([shift=(#3:#5)]#2) arc (#3:#4:#5);
    }
\centerarc[arrows={-angle 90}](-2.7,0)(-3:-52:4.8cm); %three --> two
%\centerarc[arrows={-angle 90}](-1.5,1.5)(42:-42:2.15cm); %three --> zero
\end{tikzpicture}
\end{center}
%In particular, it follows that this quiver Grassmannian is not normal. 
In particular, we can apply Theorem~\ref{trm:cohomology-generators-general-setting}  in order to compute the $S$-module basis for the $T$-equivariant cohomology. The results of this computation are presented in Figure~\ref{figure:ExCohomGenLoop}. For the filtered $T$-stable subvarieties corresponding to the fixed points, we use the same notation as in Theorem~\ref{trm:t-stable_filtration-general-setting}.  Observe that $Z_4$ is not smooth in $p_2$ and $p_3$, We apply Lemma~\ref{lma:euler-class-along-resolution}.(3) to compute the equivariant Euler classes $\mr{Eu}_T(p_2,Z_4)$ and $\mr{Eu}_T(p_3,Z_4)$. These computations are described in Appendix~\ref{app:Desing}. For all other equivariant Euler classes we can apply Lemma~\ref{lma:euler-class-along-resolution}.(1). 
\begin{figure}[!h] 
\begin{align*}
\theta_1&= \begin{pmatrix} 1 &  & 1\\
  & 1 &  \\
  & 1 &  \\ \end{pmatrix}, \quad \theta_2= \begin{pmatrix} 2\epsilon_1-\epsilon_2-\epsilon_3+\delta &  & \epsilon_1-\epsilon_3+\delta\\
  & \epsilon_1-\epsilon_2+\delta &  \\
  & 0 &  \\ \end{pmatrix}\\
\theta_3&= \begin{pmatrix} (\epsilon_3-\epsilon_1)(\epsilon_3-\epsilon_2-\delta) &  & (\epsilon_3-\epsilon_2)(\epsilon_3-\epsilon_1-\delta)\\
  & 0 &  \\
  & 0 &  \\ \end{pmatrix}\\
   \theta_4&= \begin{pmatrix} (\epsilon_3-\epsilon_1)(\epsilon_2-\epsilon_1) &  & 0\\
  & 0 &  \\
  & 0 &  \\ \end{pmatrix} 
\end{align*}
\caption{$S$-module basis of the $T$-equivariant cohomology}\label{figure:ExCohomGenLoop}
\end{figure}

We conclude this example by describing the ring structure of $H_T^\bullet(\mr{Gr}_2(M))$. Under localisation, the addition and multiplication laws, are defined componentwise. It is hence immediate to see that $\theta_1$ is the unit of the ring. It is also easy to check that 
\[
\begin{array}{rcl}
\theta_2^2 & = &(\epsilon_1-\epsilon_2+\delta)\theta_2+ \theta_3+2\theta_4, \\ 
 \theta_2\theta_3=\theta_3\theta_2 & = &(\epsilon_1-\epsilon_3+\delta)\theta_3+(\epsilon_2-\epsilon_3+\delta)\theta_4,  \\ 
 \theta_2\theta_4=\theta_4\theta_2 & = & ( 2\epsilon_1-\epsilon_2-\epsilon_3+\delta)\theta_4, \\ 
 \theta_3^2 & = &(\epsilon_3-\epsilon_2)(\epsilon_3-\epsilon_1-\delta)\theta_3+ (\epsilon_1-\epsilon_2)(\epsilon_3-\epsilon_2-\delta)\theta_4, \\ 
 \theta_3\theta_4=\theta_4\theta_3 & = & (\epsilon_3-\epsilon_1)(\epsilon_3-\epsilon_2-\delta)\theta_4,\\ 
 \theta_4^2 & = &(\epsilon_3-\epsilon_1)(\epsilon_2-\epsilon_1)\theta_4.
\end{array}
\]
This completely determines the ring structure of $H_T^\bullet(\mr{Gr}_2(M))$. We observe that the equivariant cohomology is an $S$-algebra and that $S$ acts diagonally under localisation. By looking at the above multiplication table, we immediately see that there are two possible subsets of our module basis which generate $H_T^\bullet(\mr{Gr}_2(M))$ as an $S$-algebra: $\{\theta_2, \theta_3\}$ and $\{\theta_2, \theta_4\}$. Recall that we have been focusing on cohomology with rational coefficients.  Actually, our basis $\{\theta_1, \ldots, \theta_4\}$ generates the  equivariant cohomology with coefficients in any field. On the other hand, if the field has characteristic 2, only $\{\theta_2, \theta_4\}$ generates $H_T^\bullet(\mr{Gr}_2(M))$ as an algebra.
\end{ex}

\section{Open problems and further research directions}

Constructions and results presented in this paper are a first step towards the application of moment graph techniques to the investigation of quiver Grassmannians, related combinatorics and representation theory. We believe that this is only the tip of the iceberg. We list here some of the questions which remain to be addressed.

\subsection{Explicit formulae for the $\theta$-basis} In the classical setting of flag varieties, it is possible to give an explicit formula for any entry of the GKM-presentation of an equivariant Schubert class. This formula is known as Billey formula because of the article \cite{Billey}, but had already been noticed by Andersen, Jantzen and Soergel \cite[Appendix D]{AJS94}. Explicit, positive Billey formulae for varieties other than flag varieties were asked for by Tymoczko \cite[Question 15. ]{Tymoczko2012}. 

It would be hence interesting to provide such formulae at least for some special class of nilpotent $\Delta_n$-representations.

Recall that the determination of the  $\theta$ basis relies on the computation of  equivariant Euler classes, so that to get an analogue of Billey formula it is first necessary to have an explicit formula for such classes. In the flag variety case this is achieved by Arabia in \cite[\S2.7 Equation (27)]{Arabia98} exploiting Bott-Samelson resolutions. Therefore, a strategy to obtain the desired formulae is to first find an equivariant desingularisation of the quiver Grassmannian of interest, then compute the equivariant Euler classes, and finally use them to determine the formula for the $\theta$-basis, as we do in Example \ref{ex:cohomology-generators-loop-quiver}.

After this paper was written, (equivariant) resolutions for a (very special) class of nilpotent representations for the equioriented cycle have been constructed in \cite{feiginlaninipuetz2021}. Such construction has not been applied to calculate equivariant Euler classes yet. 

\subsection{Extend our construction to a broader class of quiver Grassmannians}
For string representations, $\mathbb{C}^*$-actions on quiver Grassmannians have been studied by in Cerulli Irelli \cite{Cerulli2011}. These are representations such that there exists a basis for which the coefficient quiver of the representation consists only of orientations of Dynkin diagrams of type $\tt A$. He gives a combinatorial description of the fixed points in terms of successor closed subquivers in the coefficient quiver. In the case of nilpotent representations of the equioriented cycle, the same description is valid for the fixed points of the higher rank torus as introduced above (cf. Theorem~\ref{trm:comb-moment-graph}). The results of \cite{Cerulli2011} were generalised to the setting of tree and band representations by Haupt in \cite{Haupt2012}. These are representations with coefficient quivers consisting of oriented trees and bands. 

The $\mathbb{C}^*$-action as introduced in \cite{Cerulli2011} is subject to slightly less restrictive conditions than our assumptions in the definition of attractive gradings (cf. Definition~\ref{defi:attractive-grading}). Hence we belive that it is possible to generalise our results (concerning the cellular decomposition, torus action and the moment graph structure) from the present paper to the setting of string band and tree representations which admit an attractive grading and some sort of aligned coefficient quiver (cf. Definition~\ref{rem:aligned-coeff-quiv}). This combination was important to prove existence of the cellular decomposition. 

Thus the first step towards a generalisation would be to find the class of quiver representations which admit attractive gradings and aligned coefficient quivers. Then one has to check if these assumptions are sufficient to obtain a cellular decomposition of the corresponding quiver Grassmannians. In the next step one has to adapt the construction for the action of the larger torus to this setting and check if it has the desired properties.  
%\color{red}{{\bf Alex, could you wrte something about the folliwing?} ``extending our torus action definition to a bigger class of quiver representations would greatly enlarge the family of varieties known to be GKM; more precisely, we expect our construction to be adaptable to an appropriate class of string representations (see \cite{Cerulli2011})"}

\subsection{Applications to Geometric Representation Theory} Geometric representation theory exploits geometric tools to investigate representations of groups or algebras. If on one hand the geometric realisation of a representation allows one to apply geometric methods, on the other hand it is often not very explicit. In \cite{Tymoczko2008} Tymoczko studies Weyl group representations on cohomology rings of Schubert varieties via GKM-theory, obtaining hence an explicit description of the space and at the same time the desired geometric construction. In the survey paper \cite{Tymoczko2008b} she proposes the challenge to find other spaces whose (equivariant) cohomology rings are endowed with group actions and which can be described via GKM-theory.

%We believe that under some technical assumptions it is possible to extend Tymoczko's constructions to quiver Grassmannians for nilpotent representations of the equioriented cycle. 

After the first version of this paper was written, we managed in \cite{LaniniPuetz2021} to extend Tymoczko's work \cite{Tymoczko2008} and equip, under some technical assumptions, the equivariant cohomology of quiver Grassmannians for nilpotent representations of the equioriented cycle with the action of appropriate products of symmetric groups. These representations were studied and decomposed into irreducible representations in \cite{LaniniPuetz2021}. Moreover, under the same technical assumptions,  the action of a certain Nil Hecke ring on the equivariant cohomology was obtained (see \cite[Theorem 6.8]{LaniniPuetz2021}), but not investigated. While in the classical setting of the flag variety the equivariant cohomology is a cyclic module for the action of the Nil Hecke ring, in the quiver Grassmannian case this does not hold anymore. Thus, it would be interesting to further investigate this Nil Hecke algebra module structure. For instance, it would be fascinating to determine when the equivariant cohomology is a cyclic module for the Nil Hecke algebra.

\subsection{Sheaves on moment graphs for quiver Grassmannians} Given an oriented graph without oriented cycles whose edges are labelled by elements of a $\mathbb{Z}$-module, it is possible to define the corresponding category of sheaves on it (see for example \cite{Fiebig2008}). 

If the graph is the moment graph of a torus action on a complex projective GKM-algebraic variety equipped with a $T$-stable (Whitney) stratification, as in \cite[\S1.1]{BradenMacPherson2001}, then an appropriate class of sheaves on this moment graph (the co--called BMP-sheaves) allows one to compute $T$-equivariant intersection cohomology. In general, the cellularisations of the quiver Grassmannians that we obtain in this article are not stratifications. Nevertheless, it would be interesting to find some classes of quiver Grassmannians of nilpotent representations of the equioriented cycle such that we do get a stratification. For example, in the special case treated in \cite{feiginlaninipuetz2021} 
a stratification is obtained and it is hence possible to apply Braden-MacPherson theory to determine the  the equivariant intersection cohomology. It would be interesting to see whether the Poincar\'e polynomials of the local intersection cohomology groups have interesting combinatorial features, as it happens in the flag variety case, where one gets Kazhdan-Lusztig polynomials. 
%For example, in the case of the quiver Grassmannians treated in \cite{feiginaninipuetz2021}, the resulting polynomials would be indexed by pair of Grassna

The study of categories of sheaves on moment graphs coming form the GKM-variety structure on quiver Grassmannians for nilpotent quiver representations has not yet be initiated, but we expect it to be fruitful. 

If the moment graph is the Bruhat graph of some Coxeter group (or a parabolic analogue), the full subcategory of BMP sheaves produces a (weak) categorification of the Hecke algebra (or a certain parabolic module over it, see \cite{Lanini2014}) of the underlying Coxeter group. By \cite{Fiebig2008}, it is in fact a moment graph realisation of the famous category of Soergel bimodules. This fact led to interesting categorical lifting of properties of Kazhdan-Lusztig polynomials (see, e.g., \cite{Lanini2012, Lanini2015}). 
%All global sections of sheaves on the moment graph coming from a GKM-variety are equipped with an action of the equivariant comology of the underlying variety. 

We believe that the investigation of the category $\mathcal{B}$ of BMP-sheaves at the very least for the quiver Grassmannians appearing in \cite{feiginlaninipuetz2021} is worth to be pursued. In this case, the Grothendieck group of $\mathcal{B}$ has a basis indexed by a combinatorially interesting set, that is the set of Grassmann necklaces, or of juggling patterns.% (see, for example, \cite{feiginlaninipuetz2021}). 

\subsection{Combinatorics of moment graphs coming from quiver Grassmannians}
A combinatorial study of the moment graphs obtained by our construction might produce interesting algebro-combinatorial results, as well as have geometrical consequences.

If the underlying GKM variety is coming from a (full) flag variety of an algebraic group (acted upon by a maximal torus of such a group), the obtained moment graph is called Bruhat graph and was firstly considered by Dyer in 1991 \cite{Dyer1991}. There is a vast literature on Bruhat graph combinatorics and applications, and we would be surprised if the combinatorics of our moment graphs were not of some interest itself. For example, in the case of the quiver Grassmannians studied in \cite{feiginlaninipuetz2021}, the resulting moment graphs can be described in terms of Grassmann necklace combinatorics (see \cite[Proposition 6.3]{feiginlaninipuetz2021}).

As for the geometric applications of a combinatorial study of the moment graphs, we only mention the possibility of reading off from the graphs the rational smoothness of the variety at a given fixed point (see for example \cite{Brion1999}). In the case of a Bruhat graph, this led to the so--called Carrel--Peterson sufficient and necassary criterion for a KL-polynomial to be equal to 1.
It would be certainly relevant to have a purely combinatorial criterion for a successor closed subquiver to index a rationally smooth point of a quiver Grassmannian for a nilpotent $\Delta_n$-representation.

\appendix
\section{Equivariant Desingularisations of Quiver Grassmannians}\label{app:Desing}
By Lemma~\ref{lma:euler-class-along-resolution}.(3), equivariant desingularisations can be used to compute the equivariant Euler classes at singular points. In \cite{CFR2013} desingularisations of quiver Grassmannians for Dynkin quivers are provided. More general constructions for quiver Grassmannians can be found in \cite{KeSc2014,Scherotzke2017}. The explicit nature of \cite{CFR2013} seemed to us better suited for our purposes and to the approach of the present article, allowing to work with a coordinate description, useful to define torus actions, and hence to obtain the needed equivariant resolutions.

We expect that the construction in \cite{CFR2013} can be generalised to the equioriented cycle, where some of their key assumptions are not satisfied. For the rest of this appendix, we restrict us to the special case of the quiver Grassmannian from Example~\ref{ex:cohomology-generators-loop-quiver}. In this special case, we construct an equivariant resolution of singularities, by adapting methods from \cite{CFR2013b}. 

Following the procedure from \cite[Section~8.1]{CFR2013b}, we obtain
\begin{center}
\begin{tikzpicture}[scale=0.45]
\node at (-1.1-8,0.1) {$\hat{Q} = $};
 \def\centerarc[#1](#2)(#3:#4:#5);%
%Syntax: [draw options] (center) (initial angle:final angle:radius)
    {
    \draw[#1]([shift=(#3:#5)]#2) arc (#3:#4:#5);
    }
\centerarc[arrows={-angle 90}](1.5-8,-1.5)(180-48:49:2.15cm);
\centerarc[arrows={-angle 90}](1.5-8,1.5)(-49:-180+48:2.15cm);
%\centerarc[->](1,-2)(180-45:180+42:1.44cm);

    \draw[fill=black] (0-8,0) circle (.12);
    \draw[fill=black] (3-8,0) circle (.12);
  
\node at (1.5-8,1.1) {$a$}; 
\node at (1.5-8,-1.1) {$b$}; 
%\end{tikzpicture}
%\begin{tikzpicture}[scale=0.45]
\node at (-2.2,0.1) {$\mr{and} \ \ \hat{M} = $};
 \def\centerarc[#1](#2)(#3:#4:#5);%
%Syntax: [draw options] (center) (initial angle:final angle:radius)
    {
    \draw[#1]([shift=(#3:#5)]#2) arc (#3:#4:#5);
    }
\centerarc[arrows={-angle 90}](1.5,-1.5)(180-61:57:2.15cm);
\centerarc[arrows={-angle 90}](1.5,1.5)(-57:-180+58:2.15cm);
%\centerarc[->](1,-2)(180-45:180+42:1.44cm);

    \node at (0,0) {$\C^4$}; 
    \node at (3,0) {$\C$}; 
  
\node at (1.5,1.15) {$\hat{M}_a$}; 
\node at (1.5,-1.15) {$\hat{M}_b$}; 
\node at (10,0) {$\mr{with} \  ^t\hat{M}_a= \begin{pmatrix}  1\\ 0  \\ 0 \\0  \\ \end{pmatrix} \ \mr{and} \ \hat{M}_b = \begin{pmatrix}  0\\ 1  \\ 0 \\0  \\ \end{pmatrix}.$};
\end{tikzpicture}
\end{center}
%Here we do not need further extending vertices, because the representation $M$ has nilpotence parameter $N=2$. Hence the further extending maps and vector spaces, which are defined as concatenations of $M_a$ and their image (\cite[Section~8.1]{CFR2013b}), would be zero.\martina{Jemandem, der ihre Konstruktion nicht kennt, hilft diser Satz gar nicht und ist er eigentlich verwirrend, da wir nichts ueber "Erweiterungen von Darstellungen" erklaert haben. Ich wuerde diesen Satz wegmachen, und vor dem Bild sagen " By following the extension procedure from \cite[Section~8]{CFR2013b}, we obtain". Was meinst Du?}
\begin{rem}
In the notation of the previous sections, we have $\hat{Q} =\Delta_2$ and $\hat{M} = U(1; 3) \oplus U(1; 1)\otimes\C^2$. Hence $\hat{M}$ is also a nilpotent representation of a quiver of affine type $A$. The shape of $\hat{Q}$ depends on the structure of $M$ and it is in general  not of the same type as $Q$.
\end{rem}
Observe that $\hat{Q}$ and $\hat{M}$ fail to satisfy the assumptions in \cite[Proposition~7.1]{CFR2013b}. Nevertheless, we obtain similar results about the desingularisation as in \cite[Section~7]{CFR2013b}. 

Recall that each dot in the coefficient quiver $Q( \hat{M})$ stands for one basis vector of the vector spaces in the representation. We define a $T$-action on the vector space $\hat{M}_1\oplus\hat{M}_2$ by declaring that the element $(\gamma_0,\gamma_1,\gamma_2,\gamma_3)\in T$ acts on the basis vectors as follows:
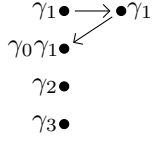
\begin{center}
\begin{tikzpicture}[scale=0.45]
%\node at (-1.1,1.5) {$\mr{and} Q(\hat{M}) = $};
\node at (-0.5,3) {$\gamma_1$}; 
\draw[fill=black] (0,3) circle (.12);
\draw[arrows={-angle 90}, shorten >=4, shorten <=4]  (0,3) -- (1.5,3);
\node at (2,3) {$\gamma_1$}; 
\draw[fill=black] (1.5,3) circle (.12);
\draw[arrows={-angle 90}, shorten >=4, shorten <=4]  (1.5,3) -- (0,2);
\node at (-0.8,2) {$\gamma_0\gamma_1$}; 
\draw[fill=black] (0,2) circle (.12);
\node at (-0.5,1) {$\gamma_2$}; 
\draw[fill=black] (0,1) circle (.12);
\node at (-0.5,0) {$\gamma_3$}; 
\draw[fill=black] (0,0) circle (.12);
\end{tikzpicture}
\end{center}
We also consider the $\C^*$-action on $\hat{M}_1\oplus\hat{M}_2$ induced by the generic cocharacter $\C^* \rightarrow T$ given by
\[z\mapsto (z, z, z^3, z^4).\]

It is immediate to see that the above defined $T$- and $\C^*$-actions on $\hat{M}_1\oplus\hat{M}_2$ extend to actions 
on the (non empty) quiver Grassmannians for $\hat{M}$. Since it corresponds to an attractive grading, the $\C^*$-action induces cellularisations of these quiver Grassmannians by Theorem~\ref{trm:cell_decomp-approx-lin-deg-aff-flag}. 

For $U \in \mr{Gr}_\mb{e}(M)$, we denote by $\mathcal{S}_U$ the subvariety of $\mr{Gr}_\mb{e}(M)$ which consists of subrepresentations in $\mr{Gr}_\mb{e}(M)$, which are isomorphic to $U$ (as quiver representations). In $\mr{Gr}_2(M)$ there are two isomorphism classes of subrepresentations with representatives $U_1= A_2$ and $U_2= A_1 \oplus A_1$. Thus we have $\mr{Gr}_2(M) = \mathcal{S}_{U_1} \sqcup \mathcal{S}_{U_2}$. For $U_1$ and $U_2$ we compute the $ \hat{Q}$-representations in the same way, as done for $M$ and obtain $\hat{U}_1  = U(1;3)$ and $\hat{U}_2 = U(1;1)\oplus U(1;1)$. Their dimension vectors are $\bdim \hat{U}_1 = (2,1)$ and $\bdim \hat{U}_2 = (2,0)$ and it follows from the definition of $\hat{M}$ that
\begin{equation}\label{eqn:Appendix}
\mr{Gr}_{(2,1)}( \hat{M} ) \cong \mr{Gr}_1(\C^3) \quad \hbox{and} \quad \mr{Gr}_{(2,0)}( \hat{M})\cong \mr{Gr}_2(\C^3).
\end{equation}
Therefore $\mr{Gr}_{(2,1)}( \hat{M})$ and $\mr{Gr}_{(2,0)}(\hat{M})$ are irreducible and smooth. 

The top dimensional cells in both Grassmannians are attractive loci of the fixed points $\hat{q}_1$ and $\hat{q}_2$ which are isomorphic to $\hat{U}_1\in\mr{Gr}_{(2,1)}(\hat{M})$ and $\hat{U}_2\in \mr{Gr}_{(2,0)}(\hat{M})$ respectively, and correspond to the following successor closed subquivers of $Q(\hat{M})$:
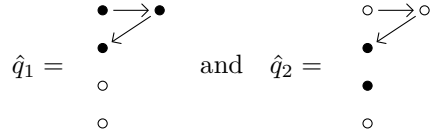
\begin{center}
\begin{tikzpicture}[scale=0.45]
\node at (-1.7,1.5) {$\hat{q}_1 = $};
\draw[fill=black] (0,3) circle (.12);
\draw[arrows={-angle 90}, shorten >=4, shorten <=4]  (0,3) -- (1.5,3);
\draw[fill=black] (1.5,3) circle (.12);
\draw[arrows={-angle 90}, shorten >=4, shorten <=4]  (1.5,3) -- (0,2);
\draw[fill=black] (0,2) circle (.12);
\draw[fill=white] (0,1) circle (.12);
\draw[fill=white] (0,0) circle (.12);
\end{tikzpicture}$\ \ $
\begin{tikzpicture}[scale=0.45]
\node at (-2.8,1.5) {$\mr{and} \quad \hat{q}_2 = $};
\draw[fill=white] (0,3) circle (.12);
\draw[arrows={-angle 90}, shorten >=4, shorten <=4]  (0,3) -- (1.5,3);
\draw[fill=white] (1.5,3) circle (.12);
\draw[arrows={-angle 90}, shorten >=4, shorten <=4]  (1.5,3) -- (0,2);
\draw[fill=black] (0,2) circle (.12);
\draw[fill=black] (0,1) circle (.12);
\draw[fill=white] (0,0) circle (.12);
\end{tikzpicture}
\end{center}
The $T$-action equips both Grassmannians with the structure of a GKM-variety. Starting from these fixed points, we compute the moment graphs of both Grassmannians analogously to Theorem~\ref{trm:comb-moment-graph}. Observe that the labels are different from the ones we would get from Theorem~\ref{trm:comb-moment-graph}, as we do not use the $T$-action as defined in Section~\ref{sec:act-bigger-T}, but rather an action which is compatible with the $T$-action on $\mr{Gr}_2(M)$. More precisely, the moment graph corresponding to $(\mr{Gr}_{(2,1)}(\hat{M}), T)$, respectively to $(\mr{Gr}_{(2,0)}(\hat{M}),T)$, is the full (labelled) subgraph of the moment graph in Example \ref{ex:cohomology-generators-loop-quiver} whose vertex set is $\{p_2,p_3,p_4\}$, respectively $\{p_1,p_2,p_3\}$.

As in \cite[Section~7]{CFR2013b}, for $U\in\{ U_1,U_2\}$ we define the map
\[ \pi_{[U]} : Gr_{\bdim \hat{U}} (\hat{M}) \to \mr{Gr}_2(M). \] %, \qquad V \mapsto \mr{res} V 
In our case, it has the explicit form $V= (V_1, V_2)\mapsto V_1$, so that $\pi_{[U_1]}(\hat{q}_1)=p_4$ and $\pi_{[U_2]}(\hat{q}_2)=p_3$. As before, $p_3$ and $p_4$ are the fixed points in $\mr{Gr}_2(M)$ as computed in Example~\ref{ex:cohomology-generators-loop-quiver}.

For our example, the same conclusions as in \cite[Theorem~7.5]{CFR2013b} hold true:
\begin{prop}\label{prop:desing}With the same notation as before, for $U\in\{ U_1,U_2\}$, we have:
\begin{enumerate}
\item $Gr_{\bdim \hat{U}} (\hat{M})$ is smooth and irreducible,
\item the map $\pi_{[U]}$ is projective and $T$-equivariant,
\item the image of $\pi_{[U]}$ is closed in $\mr{Gr}_2(M)$ and contains $\overline{\mathcal{S}_U}$,
\item the map $\pi_{[U]}$ is one-to-one over $\mathcal{S}_U$.
\end{enumerate}
\end{prop}
\begin{proof}
Part (1) follows from \eqref{eqn:Appendix}, as already noticed. Since both Grassmannians of subspaces are projective, the projectivity of $\pi_{[U]}$ is clear.

The $T$-equivariance follows immediately from the coordinate description of the maps $\pi_{[U]}$, induced by the cellular decompositions of the involved quiver Grassmannians. The image of  $\pi_{[U]}$ is closed, since projective morphisms are closed. By construction  $\pi_{[U]}$ is one-to-one even over $\overline{\mathcal{S}_U}$ and hence its image contains $\overline{\mathcal{S}_U}$ .
\end{proof}
\begin{cor}\label{cor:Resolution}With the same notation as before. Then, for $U\in\{ U_1,U_2\}$, the map 
 \[ \pi = \coprod_{U \in \{U_1,U_2\}} \pi_{[U]}  : \coprod_{U \in \{U_1,U_2\}} Gr_{\bdim \hat{U}} (\hat{M}) \to \mr{Gr}_2(M) \] 
 is a $T$-equivariant desingularisation of $\mr{Gr}_2(M)$.
\end{cor}
\begin{proof}
By Proposition~\ref{prop:desing}, $\pi$ is $T$-equivariant, and the closure of $\mathcal{S}_{\hat{U}}$ is the whole quiver Grassmannian $Gr_{\bdim \hat{U}} (\hat{M})$. Hence it is smooth and of the same dimension as the quiver Grassmannian. The rest of the proof is analogous to the proof of \cite[Corollary~7.7]{CFR2013b}.
\end{proof}
By  Corollary~\ref{cor:Resolution}, we can apply Lemma~\ref{lma:euler-class-along-resolution}.(3) to $\pi$ and compute
\begin{align*}
\mr{Eu}_T(p_2,Z_4)^{-1} &=\frac{\epsilon_2-\epsilon_3-\delta}{(\epsilon_3-\epsilon_1)(\epsilon_3-\epsilon_2)(\epsilon_2-\epsilon_1-\delta)}\\
\mr{Eu}_T(p_3,Z_4)^{-1} &=\frac{\epsilon_2-\epsilon_3+\delta}{(\epsilon_3-\epsilon_2)(\epsilon_2-\epsilon_1)(\epsilon_3-\epsilon_1-\delta)}.
\end{align*}
\begin{rem}
In particular, this example is compatible with the irreducibility conjecture, for the resolving quiver Grassmannians $Gr_{\bdim \hat{U}} (\hat{M})$, as stated in \cite[Remark~7.8]{CFR2013b}, whereas in general quiver Grassmannians for the cycle are not irreducible.
\end{rem}

\end{document}